\documentclass{article}
\usepackage[utf8]{inputenc}
\usepackage{float}
\usepackage{amsmath,amsthm,amssymb,mathtools,booktabs}
\usepackage{hyperref}
\usepackage{subcaption}
\usepackage{enumitem}

\usepackage{multirow}
\usepackage[a4paper, total={6in, 10in}]{geometry}
\usepackage{comment}
\usepackage[colorinlistoftodos]{todonotes}

\usepackage[boxed, ruled, linesnumbered,vlined]{algorithm2e}
\SetKwInput{KwData}{Input}
\SetKwInput{KwResult}{Output}

\SetCommentSty{mycommfont}
\SetArgSty{textrm}

\newtheorem{theorem}{Theorem}[section]
\newtheorem{lemma}[theorem]{Lemma}
\newtheorem{prop}[theorem]{Proposition}
\newtheorem{example}[theorem]{Example}

\newtheorem*{theorem*}{Theorem}
\newtheorem{definition}[theorem]{Definition}
\newtheorem{remark}[theorem]{Remark}

\newtheorem{construction}[theorem]{Construction}

\DeclareMathOperator{\conv}{conv}
\DeclareMathOperator{\Pot}{Pow}
\DeclareMathOperator{\R}{\mathbb{R}}

\usepackage{amsmath, amssymb, amsfonts} 

\usepackage{ifthen}

\usepackage{tikz}
\usetikzlibrary{calc}
\usetikzlibrary{positioning}
\usetikzlibrary{shapes}
\usetikzlibrary{patterns}
\usepackage{circuitikz}

\makeatletter
\edef\texforht{TT\noexpand\fi
  \@ifpackageloaded{tex4ht}
    {\noexpand\iftrue}
    {\noexpand\iffalse}}
\makeatother

\makeatletter
\newif\iftikz@node@phantom
\tikzset{
  phantom/.is if=tikz@node@phantom,
  text/.code=%
    \edef\tikz@temp{#1}%
    \ifx\tikz@temp\tikz@nonetext
      \tikz@node@phantomtrue
    \else
      \tikz@node@phantomfalse
      \let\tikz@textcolor\tikz@temp
    \fi
}
\usepackage{etoolbox}
\patchcmd\tikz@fig@continue{\tikz@node@transformations}{%
  \iftikz@node@phantom
    \setbox\pgfnodeparttextbox\hbox{}
  \fi\tikz@node@transformations}{}{}
\makeatother

\newcommand{\tikzAngleOfLine}{\tikz@AngleOfLine}
\def\tikz@AngleOfLine(#1)(#2)#3{%
  \pgfmathanglebetweenpoints{%
    \pgfpointanchor{#1}{center}}{%
    \pgfpointanchor{#2}{center}}
  \pgfmathsetmacro{#3}{\pgfmathresult}%
}

\tikzset{ 
    vertexNodePlain/.style = {fill=none, shape=circle, inner sep=0pt, minimum size=2pt, text=none},
    vertexNodePlain/.default=white,
    vertexPlain/labels/.style = {
        vertexNode/.style={vertexNodePlain=##1},
        vertexLabel/.style={gray}
    },
    vertexPlain/nolabels/.style = {
        vertexNode/.style={vertexNodePlain=##1},
        vertexLabel/.style={text=none}
    },
    vertexPlain/.style = vertexPlain/#1,
    vertexPlain/.default=labels
}
\tikzset{
    vertexNodeNormal/.style = {fill=none, shape=circle, inner sep=0pt, minimum size=4pt, text=none},
    vertexNodeNormal/.default = blue,
    vertexNormal/labels/.style = {
        vertexNode/.style={vertexNodeNormal=##1},
        vertexLabel/.style={blue}
    },
    vertexNormal/nolabels/.style = {
        vertexNode/.style={vertexNodeNormal=##1},
        vertexLabel/.style={text=none}
    },
    vertexNormal/.style = vertexNormal/#1,
    vertexNormal/.default=labels
}
\tikzset{
    vertexNodeBallShading/pdf/.style = {ball color=#1},
    vertexNodeBallShading/svg/.style = {fill=#1},
    vertexNodeBallShading/.code = {
        \if\texforht
            \tikzset{vertexNodeBallShading/svg=white}
        \else
            \tikzset{vertexNodeBallShading/pdf=white}
        \fi
    },
    vertexNodeBall/.style = {shape=circle, vertexNodeBallShading=#1, inner sep=2pt, outer sep=0pt, minimum size=3pt, font=\tiny},
    vertexNodeBall/.default = white,
    vertexBall/labels/.style = {
        vertexNode/.style={vertexNodeBall=##1, text=black},
        vertexLabel/.style={text=none}
    },
    vertexBall/nolabels/.style = {
        vertexNode/.style={vertexNodeBall=##1, text=none},
        vertexLabel/.style={text=none}
    },
    vertexBall/.style = vertexBall/#1,
    vertexBall/.default=labels
}
\tikzset{ 
    vertexStyle/.style={vertexNormal=#1},
    vertexStyle/.default = labels
}

\newcommand{\vertexLabelR}[4][]{
    \ifthenelse{ \equal{#1}{} }
        { \node[vertexNode] at (#2) {#4}; }
        { \node[vertexNode=#1] at (#2) {#4}; }
    \node[vertexLabel, #3] at (#2) {#4};
}
\newcommand{\vertexLabelA}[4][]{
    \ifthenelse{ \equal{#1}{} }
        { \node[vertexNode] at (#2) {#4}; }
        { \node[vertexNode=#1] at (#2) {#4}; }
    \node[vertexLabel] at (#3) {#4};
}

\newcommand{\edgeLabelColor}{blue!20!white}
\tikzset{
    edgeLineNone/.style = {draw=none},
    edgeLineNone/.default=black,
    edgeNone/labels/.style = {
        edge/.style = {edgeLineNone=##1},
        edgeLabel/.style = {fill=\edgeLabelColor,font=\small}
    },
    edgeNone/nolabels/.style = {
        edge/.style = {edgeLineNone=##1},
        edgeLabel/.style = {text=none}
    },
    edgeNone/.style = edgeNone/#1,
    edgeNone/.default = labels
}
\tikzset{
    edgeLinePlain/.style={line join=round, draw=#1},
    edgeLinePlain/.default=black,
    edgePlain/labels/.style = {
        edge/.style={edgeLinePlain=##1},
        edgeLabel/.style={fill=\edgeLabelColor,font=\small}
    },
    edgePlain/nolabels/.style = {
        edge/.style={edgeLinePlain=##1},
        edgeLabel/.style={text=none}
    },
    edgePlain/.style = edgePlain/#1,
    edgePlain/.default = labels
}
\tikzset{
    edgeLineDouble/.style = {very thin, double=#1, double distance=.8pt, line join=round},
    edgeLineDouble/.default=gray!90!white,
    edgeDouble/labels/.style = {
        edge/.style = {edgeLineDouble=##1},
        edgeLabel/.style = {fill=\edgeLabelColor,font=\small}
    },
    edgeDouble/nolabels/.style = {
        edge/.style = {edgeLineDouble=##1},
        edgeLabel/.style = {text=none}
    },
    edgeDouble/.style = edgeDouble/#1,
    edgeDouble/.default = labels
}
\tikzset{
    edgeStyle/.style = {edgePlain=#1},
    edgeStyle/.default = labels
}

\newcommand{\faceColorY}{yellow!60!white}   
\newcommand{\faceColorB}{blue!60!white}     
\newcommand{\faceColorC}{cyan!60}           
\newcommand{\faceColorR}{red!60!white}      
\newcommand{\faceColorG}{green!60!white}    
\newcommand{\faceColorO}{orange!50!yellow!70!white} 

\newcommand{\faceColor}{\faceColorY}
\newcommand{\faceColorSwap}{\faceColorC}

\tikzset{
    face/.style = {fill=#1},
    face/.default = \faceColor,
    faceY/.style = {face=\faceColorY},
    faceB/.style = {face=\faceColorB},
    faceC/.style = {face=\faceColorC},
    faceR/.style = {face=\faceColorR},
    faceG/.style = {face=\faceColorG},
    faceO/.style = {face=\faceColorO}
}
\tikzset{
    faceStyle/labels/.style = {
        faceLabel/.style = {}
    },
    faceStyle/nolabels/.style = {
        faceLabel/.style = {text=none}
    },
    faceStyle/.style = faceStyle/#1,
    faceStyle/.default = labels
}
\tikzset{ face/.style={fill=#1} }
\tikzset{ faceSwap/.code=
    \ifdefined\swapColors
        \tikzset{face=\faceColorSwap}
    \else
        \tikzset{face=\faceColor}
    \fi
}

\usepackage{tikz-3dplot}
\title{Construction of Toroidal Polyhedra corresponding to perfect Chains of wild Tetrahedra}
\author{
    Reymond Akpanya\thanks{RWTH Aachen University, Chair of Algebra and Representation Theory, Pontdriesch 10-12, 52062 Aachen, Germany. Email: \texttt{reymond.akpanya@rwth-aachen.de}} \and 
    Vanishree Krishna Kirekod\thanks{RWTH Aachen University, Chair of Algebra and Number Theory, Pontdriesch 14-16, 52062 Aachen, Germany. Email: \texttt{krishna.kirekod.vanishree@rwth-aachen.de}} \and 
    Alice C. Niemeyer\thanks{RWTH Aachen University, Chair of Algebra and Representation Theory, Pontdriesch 10-12, 52062 Aachen, Germany. Email: \texttt{alice.niemeyer@momo.math.rwth-aachen.de}} \and 
    Daniel Robertz\thanks{RWTH Aachen University, Chair of Algebra and Number Theory, Pontdriesch 14-16, 52062 Aachen, Germany. Email: \texttt{daniel.robertz@rwth-aachen.de}}
}
\date{}

\begin{document}
\maketitle
\begin{abstract}
 In 1957, Steinhaus proved that a chain of regular tetrahedra, meeting face-to-face and forming a closed loop does not exist \cite{steinhaus1957}. Over the years, various modifications of this statement have been considered and analysed. 
Weakening the statement by only requiring the tetrahedra of a chain to be wild, i.e.\ having all faces congruent, results in various examples of such chains. In this paper, we elaborate on the construction of these chains of wild tetrahedra. We therefore introduce the notions of chains and clusters of wild tetrahedra and relate these structures to simplicial surfaces. We establish that clusters and chains of wild tetrahedra can be described by polyhedra in Euclidean $3$-space. As a result, we present methods to construct toroidal polyhedra arising from chains and provide a census of such toroidal polyhedra consisting of up to $20$ wild tetrahedra.
 Here, we classify toroidal polyhedra with respect to self-intersections and reflection symmetries. We further prove the existence of an infinite family of toroidal polyhedra emerging from chains of wild tetrahedra and present clusters of wild tetrahedra that yield polyhedra of higher genera. 
\end{abstract}

\section{Introduction}
\label{section:introduction}
Since ancient times mathematicians have been fascinated by the existence of polyhedra. 
The Platonic solids are well-known examples of such polyhedra whose properties have been the focus of many studies. For instance, geometric and combinatorial properties of regular polyhedra, i.e.\ polyhedra whose faces are regular polygons, are well established, see \cite{onetriangle,symmetry,ghentpaper,coxeterhelix,eppstein,deltahedra}.
Nevertheless, the properties of regular and irregular polyhedra continue to inspire a wide range of investigations.
 In this paper, we study a weakened version of the \textbf{Steinhaus conjecture} using the theory of simplicial surfaces. We recall a question summarising this conjecture from \cite{steinhaus1957}.
\begin{itemize}
\label{question:SteinhausConjecture}
  \item[Q1:] \textit{Is there a finite sequence of congruent regular tetrahedra $T_1,\ldots, T_n $ such that for $i=1,\ldots,n$ the tetrahedra $T_i$ and $T_{i+1}$ meet face-to-face? (We read the subscripts modulo $n$.) }
\end{itemize}
Here, two congruent tetrahedra $T,T'\subseteq \mathbb{R}^3$ are said to \emph{meet face-to-face} if the intersection $T\cap T'$ is a triangle that forms a face of both $T$ and $T'$.
In the literature, a sequence that arises from iteratively glueing regular tetrahedra together is often referred to as a \emph{chain of regular tetrahedra}. Two examples of such chains consisting of three and six regular tetrahedra are illustrated in Figure~\ref{fig:tetrahedrasequence}. We call such a chain of regular tetrahedra \emph{perfect}, if a face of the last tetrahedron coincides with a face of the first one.
\begin{figure}[H]
  \centering
\begin{subfigure}{0.3\textwidth}
  \centering
\includegraphics[height=4cm]{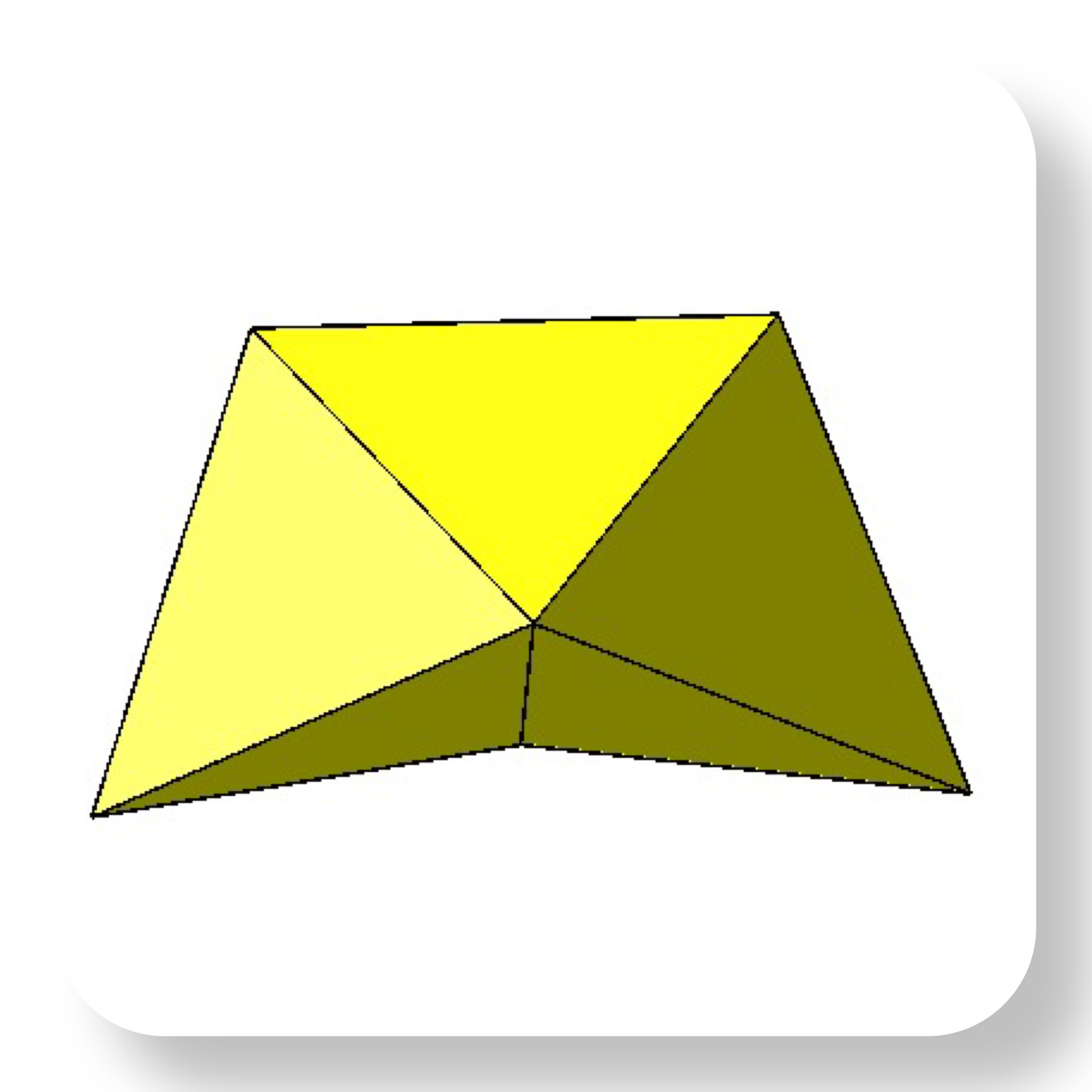}
\caption{}
\label{3cacti}
\end{subfigure} 
\begin{subfigure}{0.3\textwidth}
\centering
\includegraphics[height=4.cm]{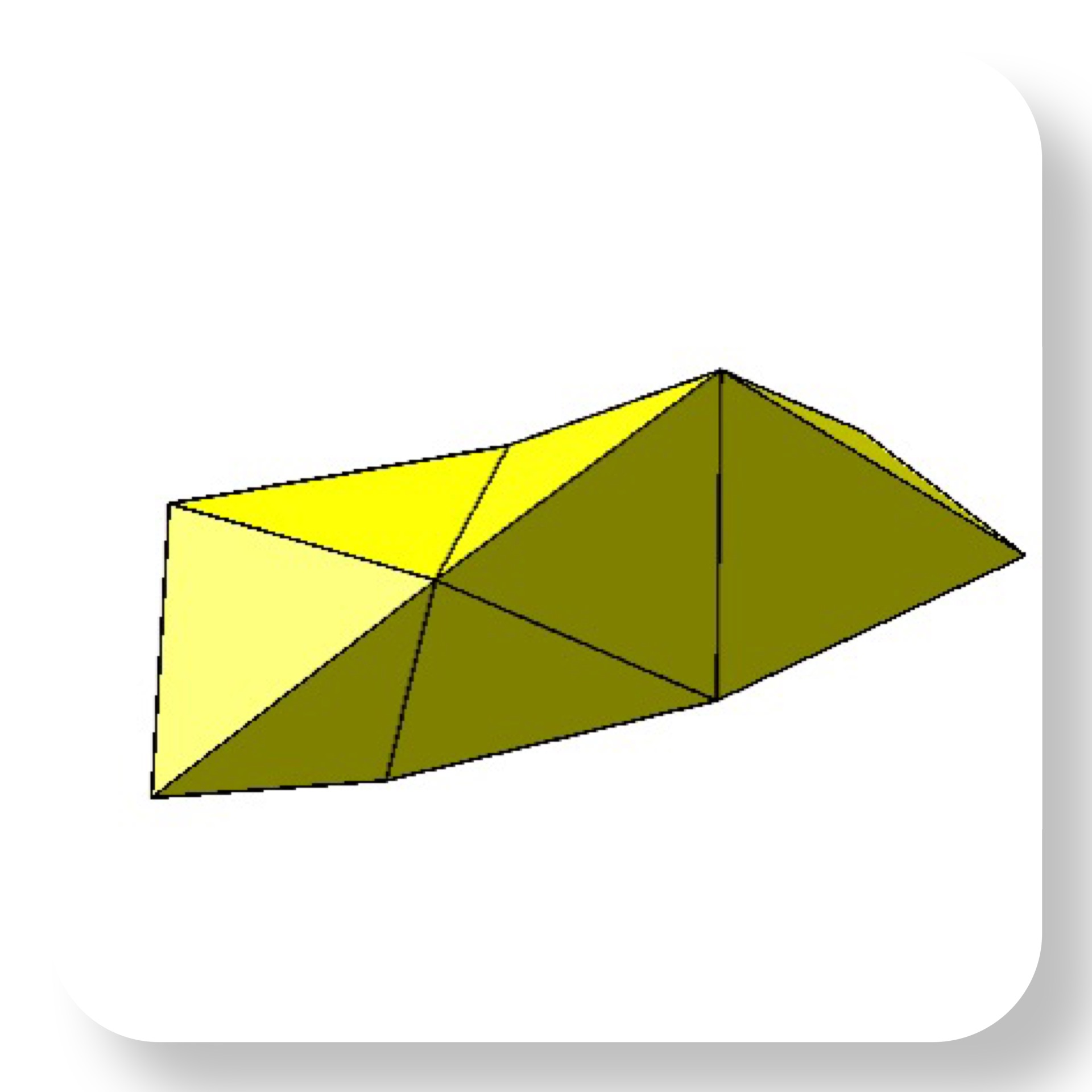}
\caption{}

\end{subfigure}
  \caption{Chains of regular tetrahedra consisting of three (a) and six (b) regular tetrahedra.}
  \label{fig:tetrahedrasequence}
\end{figure}

Over the years, several independent publications have established the non-existence of such perfect chains of regular tetrahedra. For instance, some of these proofs can be found in \cite{dekkerproof,masonproof,stewart,swierczkowski1959chains}. Since then, a lot of research has been motivated by the absence of a perfect chain of regular tetrahedra. These studies can be divided into the following two lines of research:
\begin{itemize}
  \item Constructing chains of regular tetrahedra and then introducing edge-length perturbations to some of the edges of the tetrahedra involved to form closed loops, see \cite{quadrahelix, asymptotic_helix};
  \item Computing chains of regular tetrahedra so that the gap between a face of the first and a face of the last tetrahedron of the chain  is constrained to be within a predefined error accuracy, see \cite{platonicgap}.
\end{itemize}
In this paper, we elaborate on the construction of chains of tetrahedra, where the faces of the tetrahedra involved are all congruent, but not equilateral triangles. We refer to a tetrahedron whose faces are congruent triangles as a \emph{wild tetrahedron}. Hence, we investigate the following question:
\begin{itemize}
  \item[Q2:] \textit{Does there exist a finite sequence of congruent wild tetrahedra $T_1,\ldots, T_n,$ such that for $i=1,\ldots,n$ the tetrahedra $T_i$ and $T_{i+1}$ meet face-to-face? (We read the subscripts modulo $n$.)}\label{question:weakenedsteinhaus}
\end{itemize}
Here, we shall refer to such a sequence of wild tetrahedra as a \emph{perfect chain of wild tetrahedra}. 
In addition, we refer to a sequence of wild tetrahedra $T_1,\ldots,T_n$ as a \emph{cluster of wild tetrahedra} (see Definition~\ref{defintion:tetrahedron}), if for every $2\leq i\leq n$ there exists an $1\leq j<i$ such that $T_i$ and $T_j$ meet face-to-face. 
Our investigations show that modifying the Steinhaus conjecture~(summarised in Q\ref{question:SteinhausConjecture}) in the above sense leads to various examples of perfect chains of wild tetrahedra.
In particular, we present the following theorem:

\begin{theorem*}
There exists a perfect chain of wild tetrahedra without self-intersections consisting of $14$ tetrahedra.
\end{theorem*}

This perfect chain of wild tetrahedra, whose tetrahedra consist of triangles with edge lengths $\big(1,\tfrac{\sqrt{150+30\sqrt{5}}}{10},\tfrac{\sqrt{5}+1}{2}\big)$ is illustrated in Figure~\ref{fig:smallest_torus}. 
\begin{figure}[H]
  \centering  
  \includegraphics[height=5cm]{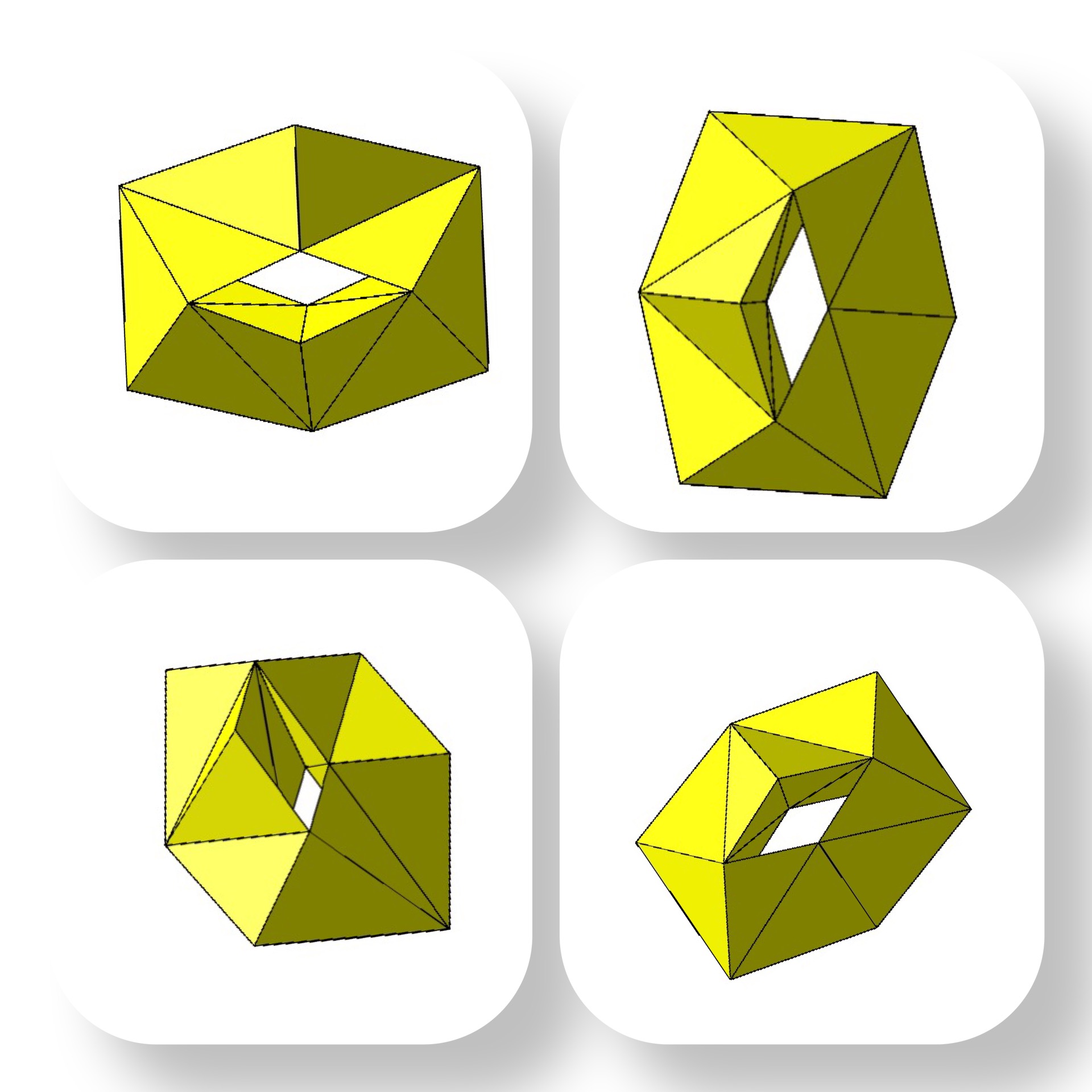}
  \caption{Various views of a perfect chain of wild tetrahedra with $14$ tetrahedra.}
  \label{fig:smallest_torus}
\end{figure} 
We conjecture the above perfect chain of wild tetrahedra to be minimal with respect to the number of tetrahedra. Further, we provide a census of the perfect chains of wild tetrahedra that we have discovered during our investigations. Moreover, we construct an infinite family of such chains.

\begin{theorem*}\label{theorem:numberOfChains}
 There exist infinitely many chains of wild tetrahedra consisting of non-isosceles faces.
\end{theorem*} 

The key to constructing perfect chains of wild tetrahedra is to describe clusters of wild tetrahedra as polyhedra that correspond to \emph{simplicial surfaces}, see Definition~\ref{def:simplicialsurface} and Definition~\ref{definition:relation}. 
In particular, a simplicial surface can be used to describe the combinatorial structure of a polyhedron, whose faces are all triangles. Hence, such a simplicial surface yields the incidence relations between the vertices, edges and faces of the underlying polyhedron. 
Note that the edges incident to the same face of the corresponding polyhedron might differ in length. If the faces of the underlying polyhedron are all congruent triangles, the information about the edge lengths of the triangles can be encoded into the simplicial surface by introducing a $3$-colouring of the edges called a \emph{wild-colouring}, see \cite{simplicialsurfacebook}.
In literature, such a colouring is also known as a Grünbaum-colouring \cite{gruenbaum1,grunbaum1969conjecture}.
This colouring introduces the challenge of realising a given simplicial surface as a polyhedron consisting of scalene triangles as faces.
Using the combinatorial structure of a wild-coloured simplicial surface to find corresponding polyhedra in $\R^3$ opens up the construction of a wide range of polyhedra, whose geometric properties can be further analysed by exploiting group theoretic methods \cite{automorphism,icosahedron}.

 The simplicial surfaces that arise from clusters of wild tetrahedra form \emph{multi-tetrahedral surfaces}, see Definition~\ref{definition:relation}. 
 We focus on perfect chains of wild tetrahedra that give rise to simplicial tori, i.e.\ closed surfaces with Euler characteristic 1.
In this paper, we refer to such a simplicial torus as a \emph{multi-tetrahedral torus}.
 Spherical simplicial surfaces corresponding to clusters of wild tetrahedra have been well studied from a combinatorial and graph-theoretic perspective, see \cite{apolloniannetwork_Intro,Birkhoff_1930,JOUR,simplicialsurfacebook,takeo}. 
 Moreover, examining multi-tetrahedral surfaces from a geometric point of view facilitates various applications and investigations.
 For instance, given the combinatorics of a wild-coloured multi-tetrahedral surface, it is relatively uncomplicated to realise this wild-coloured simplicial surface as a polyhedron in $\R^3$ whose faces form triangles of the same congruence type, see Section~\ref{section:relation}.

\subsubsection*{Structure of the paper:}
This paper is structured as follows:
In order to construct perfect chains of wild tetrahedra, we introduce the theory of simplicial surfaces in Section~\ref{section:theorybackground}, as these surfaces form the foundation of our investigations. 
Since we examine polyhedra with congruent triangular faces, we give the notions of wild-colourings and embeddings of wild-coloured simplicial surfaces into $\R^3$ in Subsections~\ref{subsection:simplicialsurfaces} and~\ref{subsection:embeddings}.
In Section~\ref{subsection:chainsoftetra} we present a formal definition of clusters of wild tetrahedra and observe that the set of chains of wild tetrahedra is a special class of these clusters.
Moreover, we employ the theory of simplicial surfaces to study chains of wild tetrahedra by establishing a relation between these structures. This is achieved by introducing multi-tetrahedral surfaces in Section~\ref{section:relation}.
Furthermore, we present all embeddings of the \emph{wild-coloured simplicial tetrahedron} into $\R^3$, i.e.\ the simplicial surface corresponding to the regular tetrahedron that is equipped with a wild-colouring. We describe a procedure to construct embeddings of wild-coloured multi-tetrahedral surfaces in Section~\ref{section:embedcacti}. 
In addition, we present a method to construct a simplicial surface with corresponding embedding by applying affine transformations to the polyhedra corresponding to two given simplicial surfaces.
In Section~\ref{section:construction}, we describe the construction of perfect chains of wild tetrahedra using the theory of simplicial surfaces. In particular, we construct toroidal polyhedra with respect to reflection symmetries and self-intersections. These constructions are detailed in Algorithm~\ref{alg:torus1} and Algorithm~\ref{alg:torus2}.
Additionally, in Section~\ref{section:infinitefamily} we present an infinite family of toroidal polyhedra that is based on the double tetra-helices introduced in \cite{quadrahelix}. Moreover, we discuss a space-filling that has been introduced by Sommerville in \cite{sommerville} and elaborate on the possibility of using this space-filling to compute various examples of polyhedra of higher genera. We explore polyhedra of higher genera corresponding to chains of wild tetrahedra, which are elaborated in Section~\ref{section:highergenus}.
Finally, we discuss the implementations of our constructions and provide details on the computed census of toroidal polyhedra in Section~\ref{section:implementation}.

The methods to achieve the described constructions are implemented in the computer algebra systems GAP \cite{GAP4} and Maple \cite{maple}. We make use of the \textsc{GAP}-package \textit{SimplicialSurfaces} \cite{simplicialsurfacegap} to examine combinatorial properties of the multi-tetrahedral surfaces and the Maple-package \textit{SimplicialSurfaceEmbeddings} \cite{SimplSurfEmb} to study the geometric properties of the polyhedra that arise from these simplicial surfaces. The data corresponding to the toroidal polyhedra that we computed in our investigations are available in \cite{dataTori}. Note that all our computations to construct the different polyhedra in this paper have been carried out algebraically. Hence, we provide exact descriptions of the presented polyhedra by their vertex coordinates and edge lengths of the polyhedra without numerical errors.
Further, we computed HTML-files that allow the reader to visualise some of the constructed toroidal polyhedra and get a better understanding of these objects, see \cite{dataTori}. 
\section{Theoretical background}
\label{section:theorybackground}
Since we aim to conduct our investigations by translating given chains of wild tetrahedra into simplicial surfaces, we present basic definitions and notions from the theory of simplicial surfaces. Furthermore, we introduce the definitions of a wild-colouring and an embedding of a given simplicial surface.

\subsection{Simplicial surfaces}
\label{subsection:simplicialsurfaces}

In this subsection, we introduce the definition of a simplicial surface by exploiting the notion of a simplicial complex. 

\begin{definition}
Let $\emptyset \neq V$ be a finite set and $\Pot(V)$ be the power set of $V$. A non-empty set $X\subset \Pot(V)$ with $\emptyset \notin X$ defines a \emph{simplicial complex}, if $x\in X$ implies $ s\in X$ for all non-empty subsets $s$ of $x.$ We say that $s\in X$ is incident to $x\in X,$ if $s\subseteq x$ holds.
\end{definition}

\begin{definition}
  Let $X$ be a simplicial complex. We define $X_i$ as a set of all the sets in $X$ of cardinality $i+1$ elements.
  Furthermore, for $x\in X_j$ we define the set $X_i(x)$ as 
  \begin{align*}
       X_i(x):=
       \begin{cases}
    \{s \in X_i \mid s\subseteq x \}, & i \leq j \\
    \{s \in X_i \mid x\subseteq s\in X\}, & i>j.
  \end{cases}
  \end{align*}
\end{definition}
For simplicity, we write $v$ for an element $\{v\}\in X_0.$ Thus, the set $X_0$ is represented by the underlying set $\bigcup_{v\in X_0}v\subseteq V.$
Using the above definition of a simplicial complex, we define simplicial surfaces.

\begin{definition}\label{def:simplicialsurface}
We say that a simplicial complex $X$ is a \emph{simplicial surface}, if the following conditions are satisfied:
  \begin{enumerate}
    \item $\vert x \vert \leq 3$ for all $x\in X$,
    \item $\vert X_2(e)\vert=2$ for all $e\in X_1$ and
    \item for all $v\in X_0$, there exists an $n\in \mathbb{N}_{\geq 3}$ such that $X_2(v)$ contains exactly $n$ elements that satisfy the following: These $n$ elements can be arranged in a sequence $(f_1,\ldots,f_n)$ such that for $i=1,\ldots,n$ the statements $v\in X_0(f_i)$ and $\vert f_i\cap f_{i+1}\vert =2$ hold, where we read the subscripts modulo $n.$ We denote the \emph{degree of $v$}, i.e.\ the number of faces that are incident to $v$, by $\deg(v)=n.$ (Umbrella condition)
  \end{enumerate}
  \end{definition}
   Note that the definition of a simplicial surface introduced in \cite{simplicialsurfacebook} is more general, but in this paper, we simplify the definition of a simplicial surface by excluding open simplicial surfaces and surfaces that cannot be described as a simplicial complex, i.e.\ non-vertex-faithful simplicial surfaces.
Here, we refer to the elements of $X_0,X_1$ and $X_2$ as the \emph{vertices}, \emph{edges} and \emph{faces} of the simplicial surface $X$, respectively.
Since the vertices and edges of a simplicial surface $X$ are subsets of the faces of $X$, the surface $X$ is uniquely determined by its set of faces. 
For instance, the \emph{simplicial tetrahedron}, illustrated in Figure~\ref{fig:simplicialsurfaces}, is the simplicial surface that is determined by the following faces:
  \[
  \{\{v_1,v_2,v_3\},\{v_1,v_2,v_4\},\{v_1,v_3,v_4\},\{v_2,v_3,v_4\}\}.
  \]
    Next, we introduce the notion of two simplicial surfaces being isomorphic and the Euler characteristic of a simplicial surface. 
  \begin{definition}\label{def:isomorphic}
      Two simplicial surfaces $X$ and $Y$ are said to be \emph{isomorphic}, if there exists a bijective map $g:X\to Y$ such that for all $s,x \in X$ the incidence $s\subseteq x$ holds in $X$ if and only if $g(s)$ is incident to $g(x)$ in $Y$.
  \end{definition}
\begin{definition}
Let $X$ be a simplicial surface. The \emph{Euler characteristic} of $X$ is given by 
\[\chi(X)=\vert X_0\vert-\vert X_1\vert+\vert X_2\vert.\]
\end{definition}
 Properties such as the genus of a surface, connectivity, orientability, etc.\ are defined in the usual way, see \cite{simplicialsurfacebook}. We refer to a connected, orientable simplicial surface as a \emph{simplicial torus} if its Euler characteristic is $0$, and as a \emph{simplicial sphere} if its Euler characteristic is $2$.
Now, we formally define the wild-colouring of a simplicial surface.

\begin{definition}\label{def:wildcolouring}
  A \emph{wild-colouring} of a simplicial surface $X$ is a map $\omega: X_1 \to \{1,2,3\}$, such that for any face $f\in X_2$ the restriction of $\omega$ to $X_1(f)$ yields a bijection. We call $(X,\omega)$ a \emph{wild-coloured simplicial surface}.
\end{definition}
In this paper, we illustrate the images $1,2,3$ of the edges of $X$ under a wild-colouring $\omega$ by the colours red,blue, green, respectively. As an example, we illustrate a wild-colouring of the simplicial tetrahedron in Figure~\ref{fig:wildcoloured_tetrahedra}
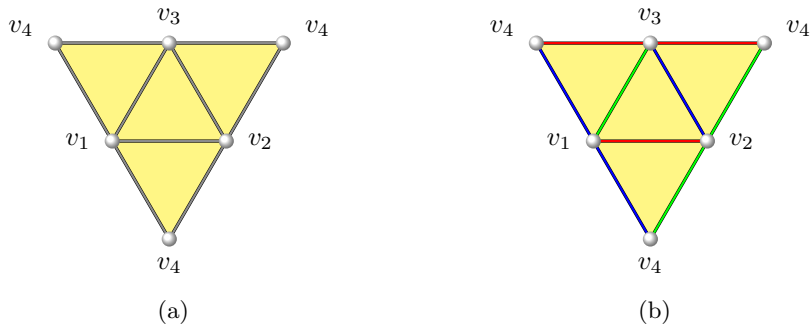
\begin{figure}[H]
  \centering
  \begin{subfigure}{.3\textwidth}
  \centering
\begin{tikzpicture}[vertexBall, edgeDouble=nolabels, faceStyle=nolabels, scale=1.5]

\coordinate (V1_1) at (0., 0.);
\coordinate (V2_1) at (1., 0.);
\coordinate (V3_1) at (0.4999999999999999, 0.8660254037844386);
\coordinate (V4_1) at (0.5000000000000001, -0.8660254037844386);
\coordinate (V4_2) at (-0.4999999999999999, 0.8660254037844385);
\coordinate (V4_3) at (1.5, 0.8660254037844388);

\fill[face]  (V2_1) -- (V3_1) -- (V1_1) -- cycle;
\node[faceLabel] at (barycentric cs:V2_1=1,V3_1=1,V1_1=1) {$1$};
\fill[face]  (V1_1) -- (V4_1) -- (V2_1) -- cycle;
\node[faceLabel] at (barycentric cs:V1_1=1,V4_1=1,V2_1=1) {$2$};
\fill[face]  (V2_1) -- (V4_3) -- (V3_1) -- cycle;
\node[faceLabel] at (barycentric cs:V2_1=1,V4_3=1,V3_1=1) {$3$};
\fill[face]  (V3_1) -- (V4_2) -- (V1_1) -- cycle;
\node[faceLabel] at (barycentric cs:V3_1=1,V4_2=1,V1_1=1) {$4$};

\draw[edge] (V2_1) -- node[edgeLabel] {$1$} (V1_1);
\draw[edge] (V1_1) -- node[edgeLabel] {$2$} (V3_1);
\draw[edge] (V4_1) -- node[edgeLabel] {$3$} (V1_1);
\draw[edge] (V1_1) -- node[edgeLabel] {$3$} (V4_2);
\draw[edge] (V3_1) -- node[edgeLabel] {$4$} (V2_1);
\draw[edge] (V2_1) -- node[edgeLabel] {$5$} (V4_1);
\draw[edge] (V4_3) -- node[edgeLabel] {$5$} (V2_1);
\draw[edge] (V4_2) -- node[edgeLabel] {$6$} (V3_1);
\draw[edge] (V3_1) -- node[edgeLabel] {$6$} (V4_3);

\vertexLabelR{V1_1}{left}{$ $}
\vertexLabelR{V2_1}{left}{$ $}
\vertexLabelR{V3_1}{left}{$ $}
\vertexLabelR{V4_1}{left}{$ $}
\vertexLabelR{V4_2}{left}{$ $}
\vertexLabelR{V4_3}{left}{$ $}
\node at (0.5,1.1) {$v_3$};
\node at (-0.3,0.) {$v_1$};
\node at (1.3,0.) {$v_2$};
\node at (0.5,-1.1) {$v_4$};
\node at (1.8,1.) {$v_4$};
\node at (-0.8,1.) {$v_4$};
\end{tikzpicture}
\caption{}\label{fig:simplicialsurfaces}
  \end{subfigure}
  \begin{minipage}{0.1\textwidth}
    \phantom{a}
  \end{minipage}
  \begin{subfigure}{.3\textwidth}
  \centering
\begin{tikzpicture}[vertexBall, edgeDouble=nolabels, faceStyle=nolabels, scale=1.5]

\coordinate (V1_1) at (0.5, 0.8660254037844386);
\coordinate (V2_1) at (0, 0);
\coordinate (V3_1) at (1, 0);
\coordinate (V4_1) at (0.4999999999999999, -0.8660254037844386);
\coordinate (V4_2) at (-0.4999999999999998, 0.8660254037844387);
\coordinate (V4_3) at (1.5, 0.8660254037844385);

\fill[face]  (V3_1) -- (V4_1) -- (V2_1) -- cycle;
\node[faceLabel] at (barycentric cs:V3_1=1,V4_1=1,V2_1=1) {$1$};
\fill[face]  (V1_1) -- (V4_3) -- (V3_1) -- cycle;
\node[faceLabel] at (barycentric cs:V1_1=1,V4_3=1,V3_1=1) {$2$};
\fill[face]  (V2_1) -- (V4_2) -- (V1_1) -- cycle;
\node[faceLabel] at (barycentric cs:V2_1=1,V4_2=1,V1_1=1) {$3$};
\fill[face]  (V2_1) -- (V1_1) -- (V3_1) -- cycle;
\node[faceLabel] at (barycentric cs:V2_1=1,V1_1=1,V3_1=1) {$4$};

\draw[edge=green] (V1_1) -- node[edgeLabel] {$2$} (V2_1);
\draw[edge=blue] (V3_1) -- node[edgeLabel] {$4$} (V1_1);
\draw[edge=red] (V1_1) -- node[edgeLabel] {$6$} (V4_2);
\draw[edge=red] (V4_3) -- node[edgeLabel] {$6$} (V1_1);
\draw[edge=red] (V3_1) -- node[edgeLabel] {$1$} (V2_1);
\draw[edge=blue] (V2_1) -- node[edgeLabel] {$3$} (V4_1);
\draw[edge=blue] (V4_2) -- node[edgeLabel] {$3$} (V2_1);
\draw[edge=green] (V4_1) -- node[edgeLabel] {$5$} (V3_1);
\draw[edge=green] (V3_1) -- node[edgeLabel] {$5$} (V4_3);

\vertexLabelR{V1_1}{left}{$ $}
\vertexLabelR{V2_1}{left}{$ $}
\vertexLabelR{V3_1}{left}{$ $}
\vertexLabelR{V4_1}{left}{$ $}
\vertexLabelR{V4_2}{left}{$ $}
\vertexLabelR{V4_3}{left}{$ $}

\node at (0.5,1.1) {$v_3$};
\node at (-0.3,0.) {$v_1$};
\node at (1.3,0.) {$v_2$};
\node at (0.5,-1.1) {$v_4$};
\node at (1.8,1.) {$v_4$};
\node at (-0.8,1.) {$v_4$};
\end{tikzpicture}
\caption{}\label{fig:wildcoloured_tetrahedra}
  \end{subfigure}
  \caption{The simplicial tetrahedron (a) and a wild-coloured simplicial tetrahedron with red, blue and green edges (b).}
  \label{fig:enter-label}
\end{figure}

A wild-colouring enriches the combinatorial structure of a simplicial surface.  In \cite{simplicialsurfacebook}, Niemeyer et al.\ show that these wild-coloured surfaces possess a fruitful group theoretical structure. 
We refer to this publication for more details.

In this paper, all the simplicial surfaces are wild-coloured and we denote a wild-coloured simplicial surface $(X,\omega)$ as $X$.
Next, we describe how to attach and remove a (simplicial) tetrahedron from a simplicial surface on a combinatorial level.
We therefore define the \emph{symmetric difference} $M { \Delta} N$ of sets $M,N$ by $(M \cup N) \setminus (M \cap N).$ 

\begin{definition}
\label{def:extension}
  Let $X$ be a simplicial surface and $f=\{v_1,v_2,v_3\}$ be a face of $X$. For $v\notin X_0,$ we define the set 
  \[
  t=\{\{v_1,v_2,v_3\},\{v_1,v_2,v\},\{v_1,v_3,v\},\{v_2,v_3,v\}\}.
  \]
 Then, the simplicial surface $T^f(X)$ is defined by the following faces:
\[
X_2 {\Delta} t.
\]
We say that $T^f(X)$ is constructed by \emph{attaching a tetrahedron onto $X.$} 
\end{definition}
For instance, attaching a tetrahedron to any face of the simplicial tetrahedron yields the \emph{simplicial double-tetrahedron} which forms a simplicial surface with $5$ vertices, $9$ edges and $6$ faces.

\begin{definition}
\label{def:reduction}
  Let $X$ be a simplicial surface that is not isomorphic to the simplicial tetrahedron and let $v$ be a vertex of degree $3$ with $X_2(v)=\{f_1,f_2,f_3\}$. For these faces, there exist vertices $v_1,v_2,v_3\in X_0$ such that 
\[
(f_1,f_2,f_3)=(\{v_1,v_2,v\},\{v_1,v_3,v\},\{v_2,v_3,v\}).
\]
The simplicial surface $T_v(X)$ with set of faces
\[
X_2 {\scriptstyle \Delta} \{\{v_1,v_2,v\},\{v_1,v_3,v\},\{v_2,v_3,v\}),\{v_1,v_2,v_3\}\}
\]
is said to be constructed from $X$ by \emph{removing a tetrahedron.}
\end{definition}
The procedures that are presented in Definition~\ref{def:extension} and Definition~\ref{def:reduction} are also called \emph{tetrahedral extension} and \emph{tetrahedral reduction}, respectively.
Note that a simplicial surface that results from applying a tetrahedral extension (reduction) to a wild-coloured simplicial surface can be equipped with a wild-colouring that arises from extending (reducing) the wild-colouring of the given surface.

\subsection{Embedding simplicial surfaces into $\mathbb{R}^3$}
\label{subsection:embeddings}

In order to realise a wild-coloured simplicial surface as a polyhedron, we have to compute an embedding of the given surface in $\R^3$ which we define in this section.
For simplicity, we define $\Lambda$ as a set that consists of all triples of positive real numbers satisfying the triangle inequalities. Hence, for $(a,b,c)\in \Lambda$ there exists a triangle with edge lengths $a,b,c$ satisfying:
\[
a+b\geq c,\phantom{a}a+c\geq b,\phantom{a}b+c\geq a.
\]

\begin{definition}\label{def:embedding}
Let $X$ be a wild-coloured simplicial surface. A map $\phi: X_0\mapsto \mathbb{R}^3$ is called a \emph{weak $(\ell_1, \ell_2, \ell_3)$-embedding} of $X$ if $(\ell_1, \ell_2, \ell_3)\in \Lambda$ is satisfied and all adjacent vertices $v_1$ and $v_2$ of $X$ satisfy the distance equation
\[
\|\phi(v_1) -\phi(v_2)\| = \ell_p,
\]
 where $\|\cdot\|$ is the Euclidean norm and the edge $e=\{v_1,v_2\}$ is coloured $p \in \{1,2,3\}$. If $\phi$ is injective, then $\phi$ is called a \emph{strong $(\ell_1, \ell_2, \ell_3)$-embedding} of $X.$ 
\end{definition}
Thus, an $(a,b,c)$-embedding gives rise to a polyhedron whose faces form congruent triangles with edge lengths $(a,b,c).$ 
\begin{remark}
    Let $X$ be a wild-coloured simplicial surface and $\phi$ be a weak $(a,b,c)$-embedding of $X.$ For a subset $M\subseteq X_0,$ we define  $\phi(M):=\{\phi(v)\mid v\in M\}.$ Hence, for an edge $e=\{v,v'\}$ of $X$ and a face $f=\{w,w',w''\}$ of $X$, the equalities $\phi(e)=\{\phi(v),\phi(v')\}$ and $\phi(f)=\{\phi(w),\phi(w'),\phi(w'')\}$ hold.
\end{remark}
An example of  a strong embedding of a wild-coloured simplicial tetrahedron id given below. Note that $x^t$ for $x\in \mathbb{R}^3$ denotes the transposed vector $x.$
\begin{example}
Let $\mathcal{T}$ be the wild-coloured simplicial tetrahedron with $\mathcal{T}_0=\{v_1,v_2,v_3,v_4\}.$ Then, a strong $(1,\tfrac{9}{10},\tfrac{3}{4})$-embedding $\phi$ of $\mathcal{T}$ is determined by the following coordinates:
\[
\left(\phi(v_1),\phi(v_2),\phi(v_3),\phi(v_4) \right) =((\tfrac{1}{2}, 0, 0)^t,(-\tfrac{1}{2}, 0, 0)^t,(-\tfrac{99}{800},\tfrac{\sqrt{150199}}{800},\tfrac{\sqrt{298}}{40})^t,(\tfrac{99}{800},-\tfrac{\sqrt{150199}}{800},\tfrac{\sqrt{298}}{40})^t).
\]
\end{example}

\begin{figure}[H]
  \centering
\includegraphics[height=3.75cm]{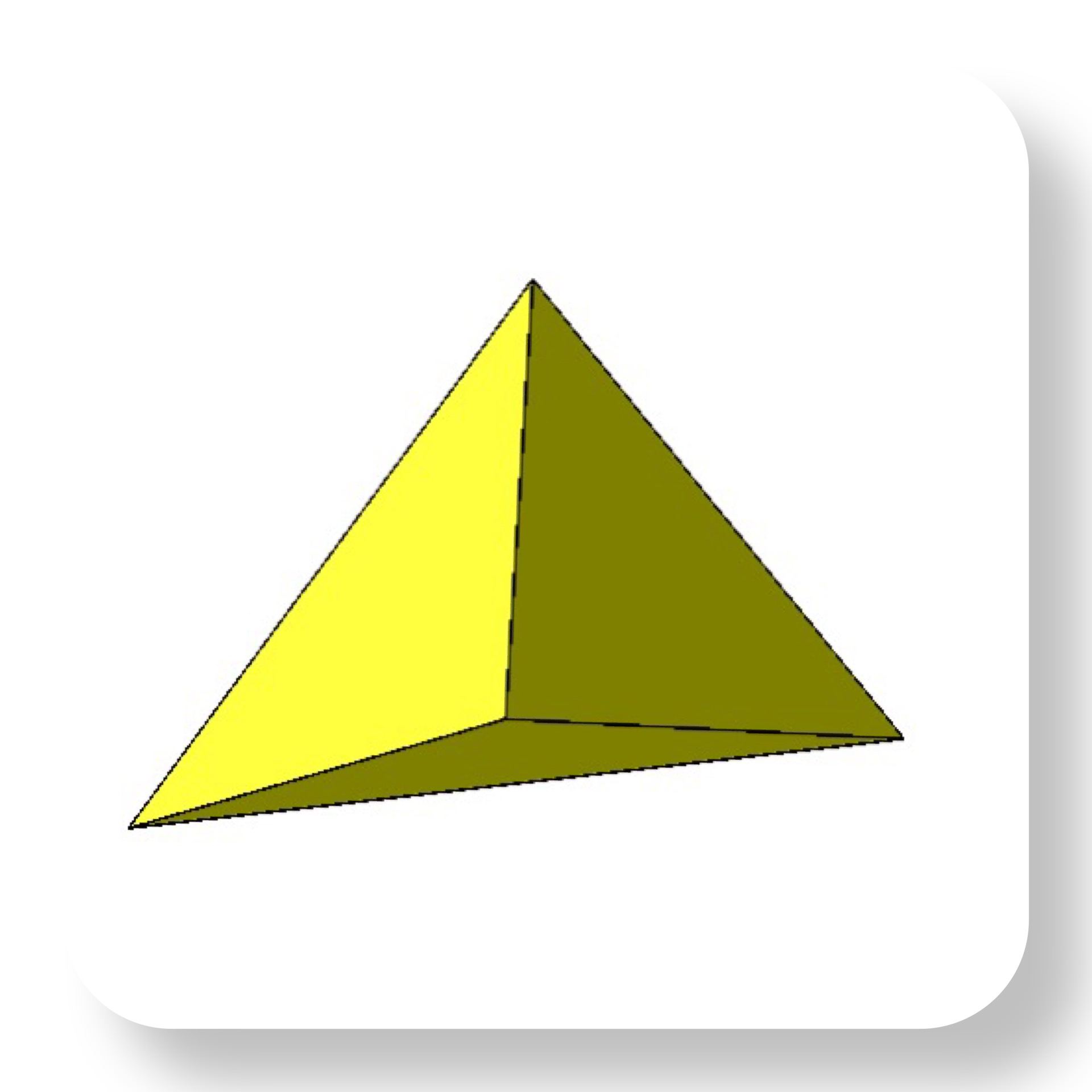}
  \caption{Polyhedron arising from a $(1,\tfrac{9}{10},\tfrac{3}{4})$-embedding of the wild-coloured simplicial tetrahedron}
  \label{fig:wildcolouredtetrahedron}
\end{figure}

In general, one has to solve a system of multivariate quadratic equations to construct an embedding of a given wild-coloured simplicial surface in $\R^3$. Finding solutions for such systems can be a strenuous task. In \cite{icosahedron}, Brakhage et al.\ examine such a system of quadratic equations and construct all $(1,1,1)$-embeddings of the combinatorial icosahedron such that the symmetry groups of the resulting polyhedra are non-trivial.
In addition, we refer to \cite{ghentpaper} for the realisations of the incidence structures of the icosahedron, the dodecahedron and various Archimedian solids as polyhedra in $\R^3$ with a prescribed symmetry group and all edge lengths being $1$.

 For a given wild-coloured simplicial surface, it is not clear whether the surface can be embedded into the Euclidean $3$-space as a polyhedron consisting of congruent triangular faces.
On the other hand, for the wild-coloured simplicial surfaces that arise from chains of wild tetrahedra, there is more hope. In Section~\ref{section:construction}, we present all embeddings of the wild-coloured simplicial tetrahedron such that the faces of the resulting polyhedra are acute- or right-angled triangles.
We further exploit these embeddings to obtain embeddings of the wild-coloured simplicial surfaces that arise from our chains of wild tetrahedra.

\section{Clusters of wild tetrahedra}
\label{subsection:chainsoftetra}
In this section, we formally introduce clusters and chains of wild tetrahedra and their properties along with various examples. We also present the tetra-symbol of a cluster of wild tetrahedra which yields a combinatorial description of the given cluster. First, we give an abbreviated description of the tetrahedra in the following remark. For simplicity, we denote the convex hull of a subset $M$ of $\R^3$ by $\conv(M).$ 

\begin{remark}
We view a \emph{tetrahedron} $T$ as the convex hull of four vertices in $\R^3$, i.e.\ $T=\conv(\{p_1,\ldots,p_4\}),$ where $p_1,\ldots,p_4$ are three-dimensional coordinates. Thus,
\begin{itemize}
  \item for all $1\leq i<j\leq 4$ the set $\conv(\{p_i,p_j\})$ is an edge of the tetrahedron $T,$  
  \item for all $1\leq i<j<k\leq 4$ the set $\conv(\{p_i,p_j,p_k\})$ forms a face of the tetrahedron $T.$ 

\end{itemize}
\end{remark}
With this notion, we can define clusters and chains of wild tetrahedra.
\begin{definition}
\label{defintion:tetrahedron}
\begin{enumerate}[leftmargin=*]
  \item A tetrahedron $T$ is said to be \emph{wild}, if all the faces of $T$ are congruent triangles with edge lengths $(a,b,c) \in \Lambda.$
  \item Let $T_1, \ldots, T_n$ be pairwise distinct wild tetrahedra. For $n\geq 2$ a sequence $(T_1,\ldots,T_n)$ is called a \emph{cluster of wild tetrahedra}, if for $j=2,\ldots,n$ there exists $1\leq i<j$ such that the tetrahedra $T_i$ and $T_{j}$ have exactly one face in common. Equivalently, $T_i$ and $T_{j}$ share exactly three vertices and three edges.
   \item A cluster of wild tetrahedra $(T_1,\ldots,T_n)$ is called a \emph{chain of wild tetrahedra} if for all $1\leq i\leq n-1$ the tetrahedra $T_i$ and $T_{i+1}$ have a face in common. The chain $(T_1,\ldots,T_n)$ is called \emph{perfect} if additionally $T_1$ and $T_n$ share a common face.
\end{enumerate}
\end{definition} Note that if wild tetrahedra $T_1$ and $T_2$ meet face-to-face then all their faces have edge lengths $(a,b,c)\in \Lambda$.
 We define the set of all vertices, edges and faces of the tetrahedra $T_1,\ldots, T_n$ of a cluster of wild tetrahedra $\tau=(T_1,\ldots,T_n)$ by $\mathcal{V}_\tau,\mathcal{E}_\tau$ and $\mathcal{F}_\tau$, respectively. It can be observed that the incidence structure of a wild tetrahedron can be described by the wild-coloured simplicial tetrahedron.
 In the following, we illustrate some examples of clusters of wild tetrahedra.

\begin{example}
Up to isomorphism, isometry and choice of edge lengths, there exist exactly three clusters of wild tetrahedra consisting of exactly four wild tetrahedra.
Figure~\ref{fig:4cacti} illustrates the three possibilities of forming clusters consisting of four regular tetrahedra. 

\begin{figure}[H]
  \centering
\begin{subfigure}{.3\textwidth}
\centering
\includegraphics[height=3.8cm]{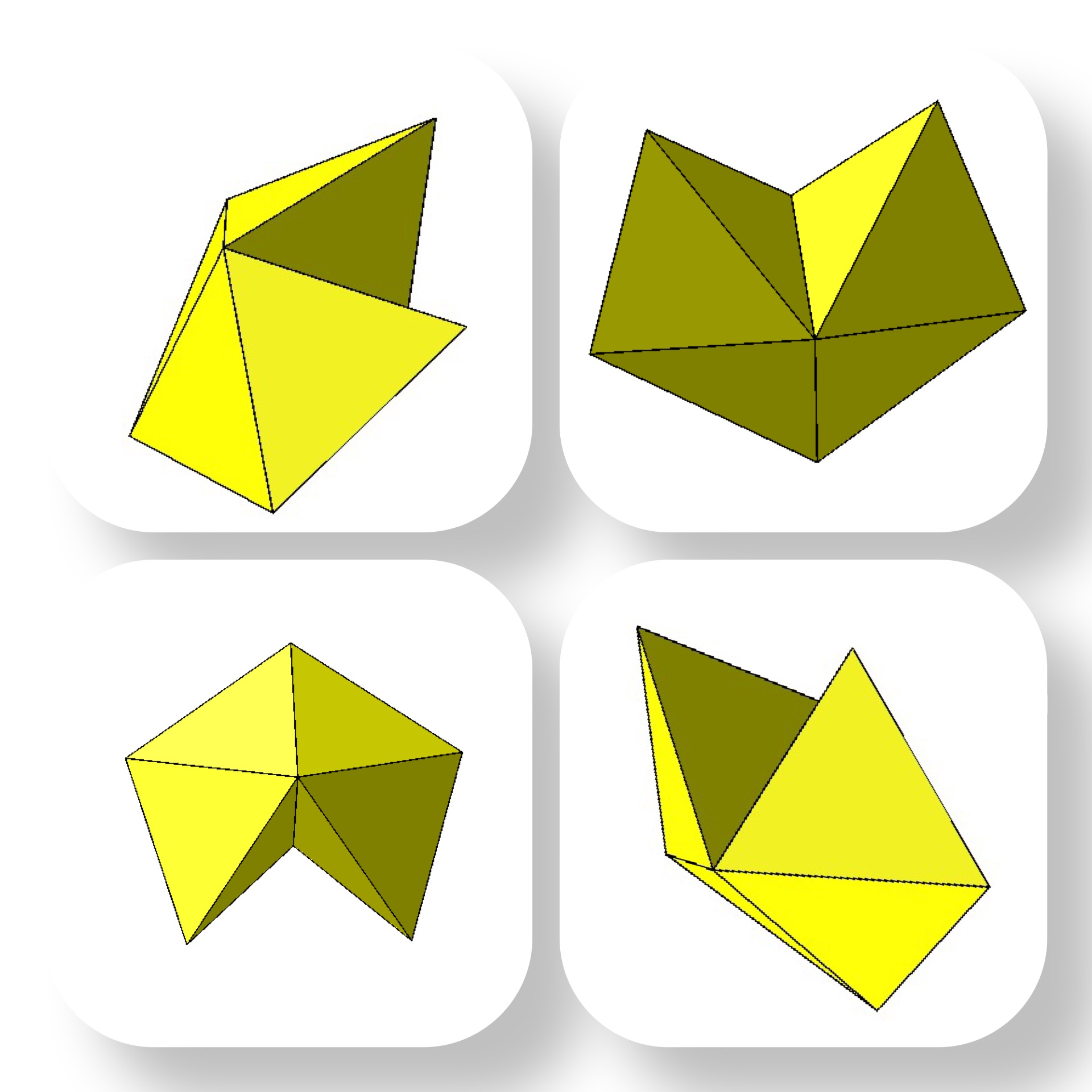}
\caption{}
\label{a}
\end{subfigure}
\begin{subfigure}{.3\textwidth}
\centering
\includegraphics[height=3.8cm]{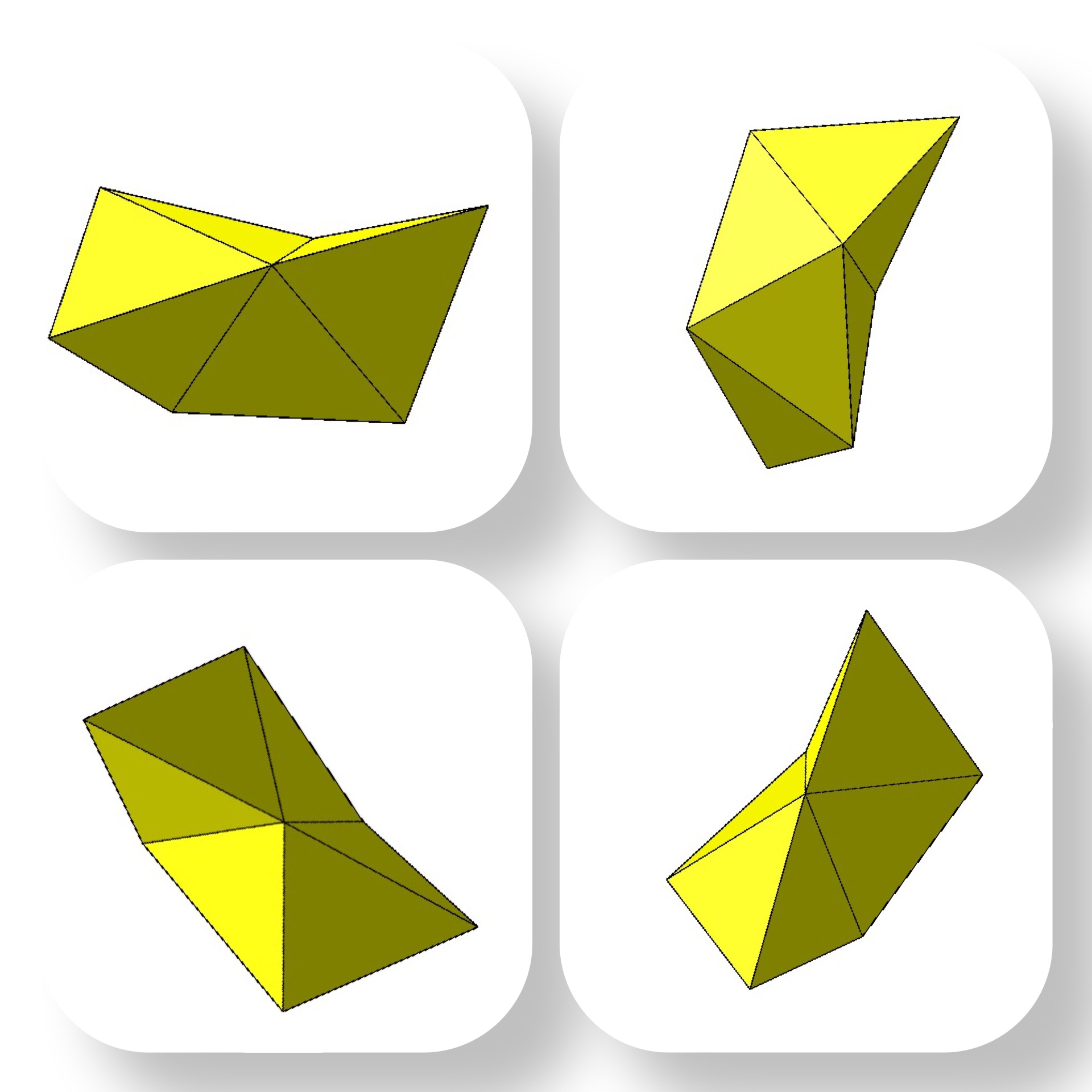}
\caption{}
\label{b}
\end{subfigure}
\begin{subfigure}{.3\textwidth}
\centering
\includegraphics[height=3.8cm]{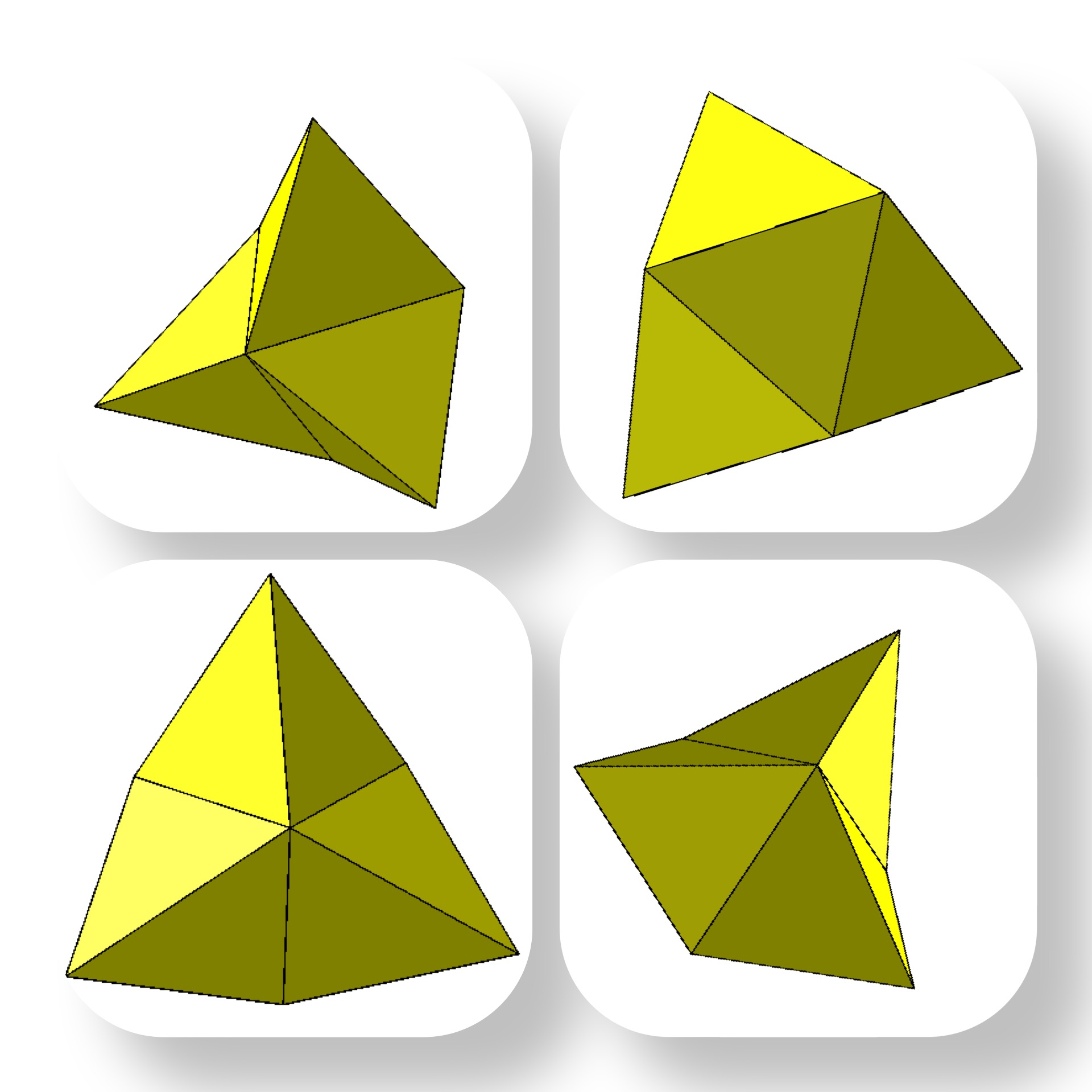}
\caption{}
\label{c}
\end{subfigure}
  \caption{Various views of two chains ((a) and (b)) and a cluster (c) of wild tetrahedra consisting of four regular tetrahedra.}
  \label{fig:4cacti}
\end{figure}
For additional examples of chains consisting of three and six tetrahedra, we refer to Figure~\ref{fig:tetrahedrasequence}.
\end{example}

In order to represent a cluster of wild tetrahedra concisely, we introduce a notation called the \emph{tetrahedral-symbol}. This symbol facilitates an easy description of the constructions of larger clusters of wild tetrahedra. 
First, we present a vertex-colouring of the tetrahedra of a given cluster.

\begin{definition}
Let $\tau=(T_1,\ldots,T_n)$ be a cluster of wild tetrahedra. Further, for $1\leq i \leq n$ let $V_i$ be the set of vertices of the tetrahedron $T_i$. Thus, $T_i=\conv(V_i)$ holds. A \emph{tetra-colouring} is a map $\gamma:\mathcal{V}_\tau\to \{1,2,3,4\}$ such that for all $1\leq i \leq n$ the restriction to $V_i$ yields a bijection.
\end{definition}

\begin{example}
Let $T_1$ and $T_2$ be regular tetrahedra given by
\begin{align*}
    &T_1=\conv(\underset{=V_1}{\underbrace{\{(0, 0, 0)^t, (1, 0, 0)^t, (\tfrac{1}{2}, \tfrac{\sqrt{3}}{2}, 0)^t, (\tfrac{1}{2}, \tfrac{\sqrt{3}}{6}, \tfrac{\sqrt{6}}{3})^t\}}}),\\
    &T_2=\conv(\underset{=V_2}{\underbrace{\{(0, 0, 0)^t, (1, 0, 0)^t, (\tfrac{1}{2}, \tfrac{\sqrt{3}}{2}, 0)^t, (\tfrac{1}{2}, \tfrac{\sqrt{3}}{6}, -\tfrac{\sqrt{6}}{3})^t\}}}).
\end{align*}
Thus, $\tau=(T_1,T_2)$ forms a chain of wild tetrahedra consisting of two regular tetrahedra with 
\begin{align*}
\mathcal{V}_\tau =V_1 \cup V_2=\{(0, 0, 0)^t, (1, 0, 0)^t, (\tfrac{1}{2}, \tfrac{\sqrt{3}}{2}, 0)^t, (\tfrac{1}{2}, \tfrac{\sqrt{3}}{6}, \tfrac{\sqrt{6}}{3})^t,(\tfrac{1}{2}, \tfrac{\sqrt{3}}{6}, -\tfrac{\sqrt{6}}{3})^t\}.
\end{align*}
Hence, a tetra-colouring of the chain $\tau$ is given by $\gamma:\mathcal{V}_\tau\to \{1,2,3,4\},$ where the images of the vertices in $\mathcal{V}_\tau$ are defined as follows:
\begin{align*}
(\gamma((0, 0, 0)^t), \gamma((1, 0, 0)^t), \gamma((\tfrac{1}{2}, \tfrac{\sqrt{3}}{2}, 0)^t), \gamma((\tfrac{1}{2}, \tfrac{\sqrt{3}}{6}, \tfrac{\sqrt{6}}{3})^t),\gamma((\tfrac{1}{2}, \tfrac{\sqrt{3}}{6}, -\tfrac{\sqrt{6}}{3})^t))=(1,2,3,4,4).
\end{align*}
\end{example}
Now, we give the notation of the tetrahedral-symbol. This notation has been first introduced and discussed in \cite{simplicialsurfacebook}.

\begin{definition}
  Let $\tau=(T_1,\ldots,T_{n})$ be a cluster of wild tetrahedra, where the tetrahedra consist of triangles with edge lengths $(a,b,c)\in \Lambda$. Let $\gamma:\mathcal{V}_{\tau}\to \{1,2,3,4\}$ be a tetra-colouring of $\tau$. Furthermore, for $1\leq i \leq n$ we denote the set of vertices of the tetrahedron $T_i$ by $V_i.$ This implies $T_i=\conv(V_i)$. For $n=1$, the \emph{tetra-symbol} of $\tau$ is given by $\sigma_{\tau}:=()^{(a,b,c)}.$ For $n>1,$ we define the \emph{tetra-symbol} of $\tau$ by $\sigma_\tau:=({(m_1)}_{k_ 1}\ldots{(m_{n-1})}_{k_ {n-1}})^{(a,b,c)},$ where the integers $m_1,\ldots m_{n-1},k_ 1,\ldots,k_{n-1}$
  are defined as follows:
  \begin{itemize}
  \item for $1\leq i \leq n-1,$ the integer $1\leq m_i \leq i$ is defined as the smallest number such that the tetrahedra $T_{m_i}$ and $T_{i+1}$ have exactly one face in common,
  \item for $1 \leq i \leq n-1,$ the integer $k_i$ is defined as $k_ i:=\gamma(p)=\gamma(p'),$ where the vertices $p,p'\in \mathbb{R}^3$ satisfy $p\in V_{i+1}\setminus (V_{m_i}\cap V_{i+1})$ and $p'\in V_{m_i}\setminus (V_{m_i}\cap V_{i+1}).$  
  \end{itemize}
\end{definition}
As an example, we consider a chain of three regular tetrahedra (with all edge lengths being $1$). This chain of wild tetrahedra can be described by the symbol $(1_11_2)^{(1,1,1)}.$ Additionally, the tetra-symbols of the clusters of wild tetrahedra shown in Figure~\ref{a}, \ref{b} and \ref{c} are given by
\begin{align*}
(1_42_13_4)^{(1,1,1)},\, (1_42_13_2)^{(1,1,1)}\, \text{and} \,
(1_42_12_3)^{(1,1,1)},
\end{align*}
respectively.
To enhance clarity, we illustrate the above definition with the following explanation:
The integers $m_i$ and $k_i$ of a tetra-symbol $\sigma_\tau$ give insights about a given cluster of wild tetrahedra $\tau.$ More precisely, the entry ${(m_i)}_{k_i}$ of $\sigma_\tau$ corresponds to the fact that the $(i+1)$-st tetrahedron is added to the cluster consisting of the first $i$ tetrahedra so that the $m_i$-th tetrahedron and the new tetrahedron have coinciding faces. Moreover, the coinciding face of the $m_i$-th tetrahedron does not contain a vertex of colour $k_i$ with respect to the corresponding tetra-colouring. It can be observed that the tetra-symbol of a given cluster of wild tetrahedra $\tau$ depends on the choice of the tetra-colouring.
For instance, $(1_1)^{(1,1,1)},(1_2)^{(1,1,1)},(1_3)^{(1,1,1)}$ and $(1_4)^{(1,1,1)}$ are distinct tetra-symbols that describe a chain of two regular tetrahedra (with all edge lengths being $1$).  

\begin{remark}
\label{remark:reconstruct}
A cluster of wild tetrahedra $\tau$ can be completely reconstructed from its tetra-symbol $\sigma_\tau$. 
\end{remark}
 To illustrate the advantages of the above symbol, we introduce two families of chains of wild tetrahedra, namely the tetra-helices and the double tetra-helices. In the case that the tetra-helices and double tetra-helices consist of regular tetrahedra, they have been largely investigated in \cite{quadrahelix}. 
 
\begin{definition}
Let $(a,b,c)\in \Lambda$ be a triple. Furthermore, let $(f_n)_{n\in \mathbb{N}}$ be the sequence defined by 
\begin{align*}
  f_1:=4, \phantom{a} f_2:=2, \phantom{a}f_3:=3, \phantom{a} f_4:=1
\end{align*}
and $f_{i+4j}:=f_i$ for $1\leq i \leq 4$ and $j \in \mathbb{N}$. 
Then, we define the \emph{tetra-helix} $\mathcal{H}_k^{(a,b,c)}$ as the chain of wild tetrahedra consisting of $k$ tetrahedra that arises from the following tetra-symbol
\[
\left(\prod_{t=1}^{k}t_{f_t}\right)^{(a,b,c)}.
\]
\end{definition}

\begin{example}
Figure~\ref{fig:helix} shows the tetra-helices $\mathcal{H}_5^{(1,1,1)},\mathcal{H}_6^{(1,1,1)}$ and $\mathcal{H}_7^{(1,1,1)}$ that are obtained by the tetra-symbols
$
\left(1_{4}2_{2}3_{3}4_1\right)^{(1,1,1)}, \left(1_{4}2_{2}3_{3}4_15_4\right)^{(1,1,1)} \text{and} \;\left(1_{4}2_{2}3_{3}4_15_46_2\right)^{(1,1,1)},
$ respectively.
 \begin{figure}[H]
  \centering
  \begin{subfigure}{.3\textwidth}
  \centering
\includegraphics[height=3.75cm]{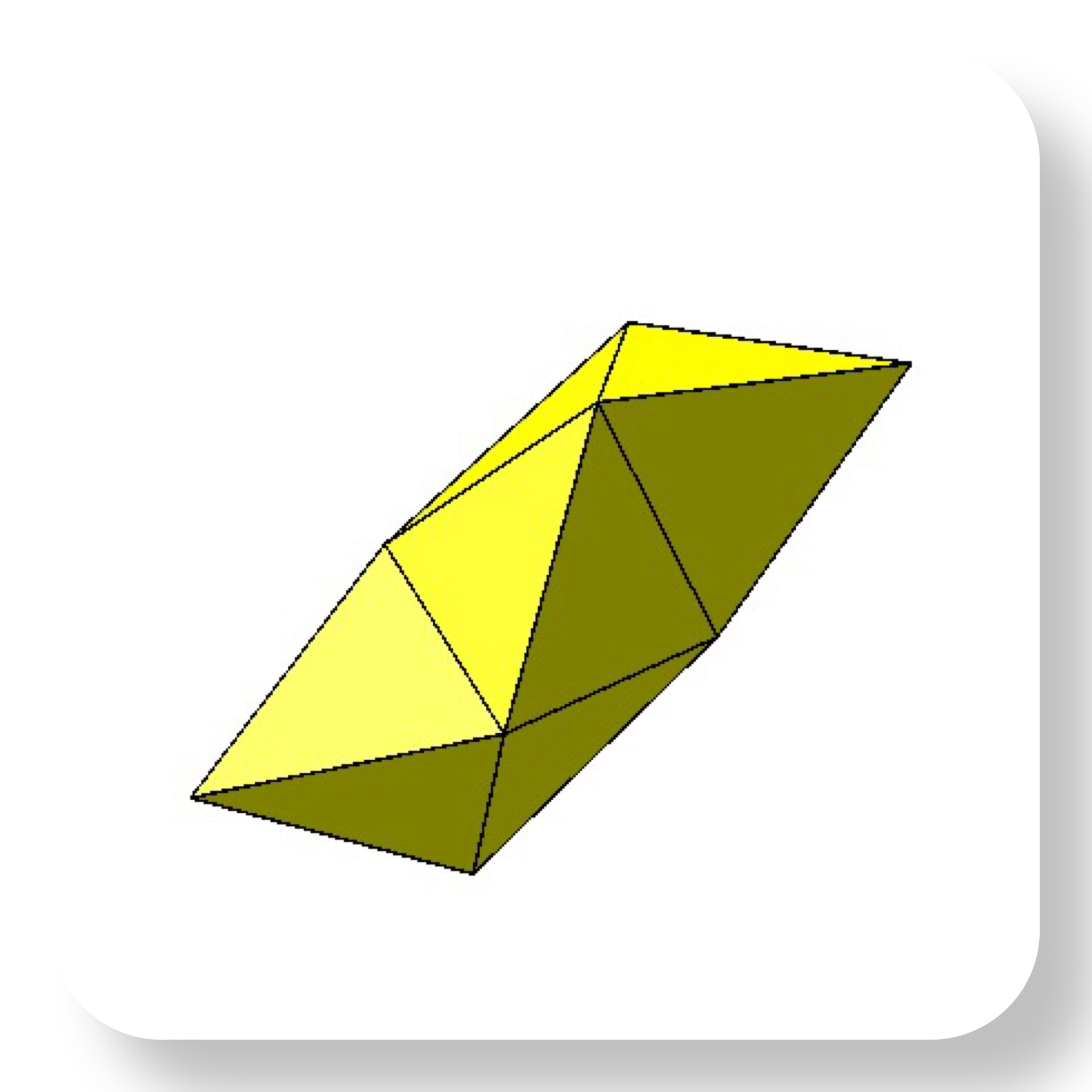}
  \caption{}
  \end{subfigure}
  \begin{subfigure}{.3\textwidth}
  \centering
\includegraphics[height=3.75cm]{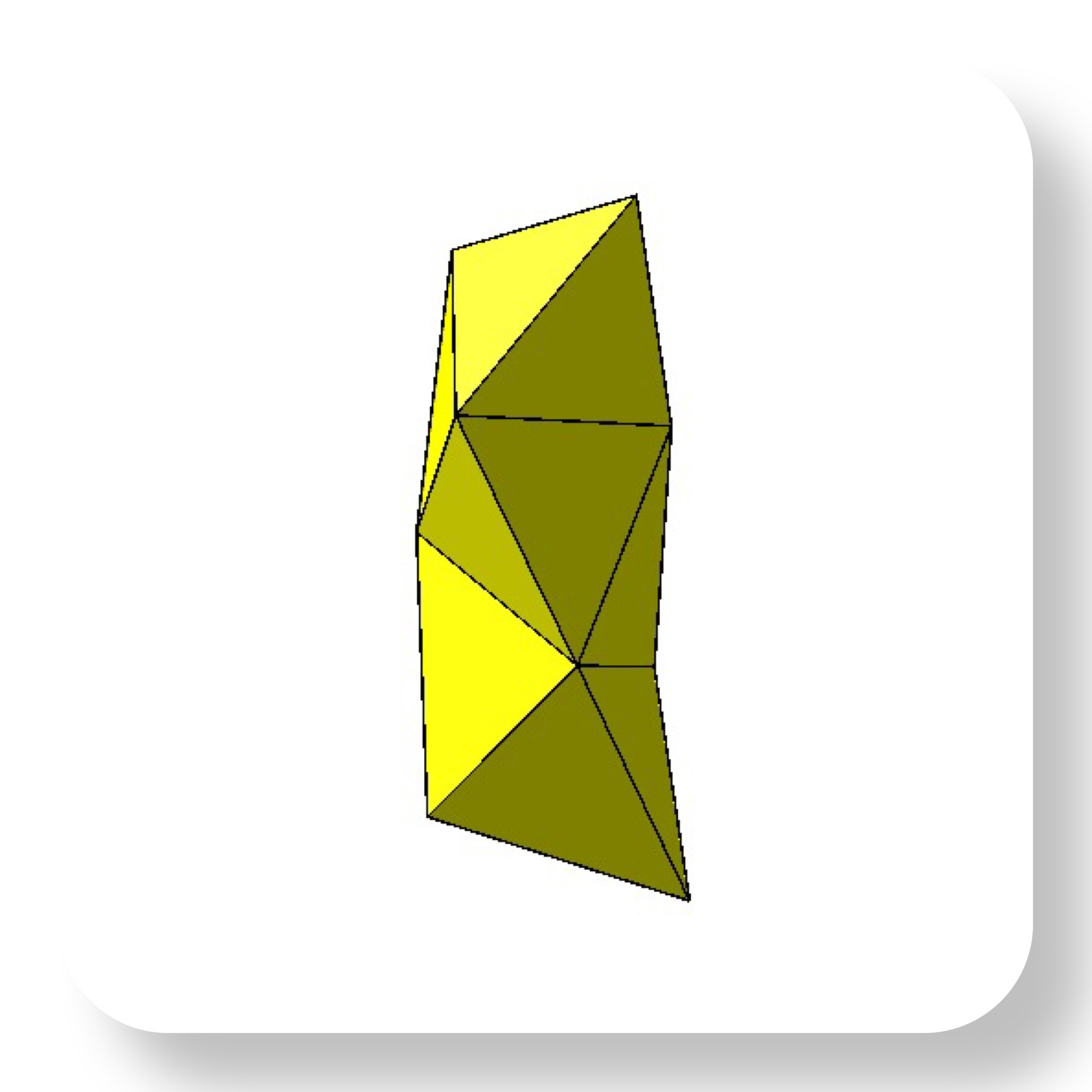}
\caption{}
  \end{subfigure}
  \begin{subfigure}{.3\textwidth}
 \centering
\includegraphics[height=3.75cm]{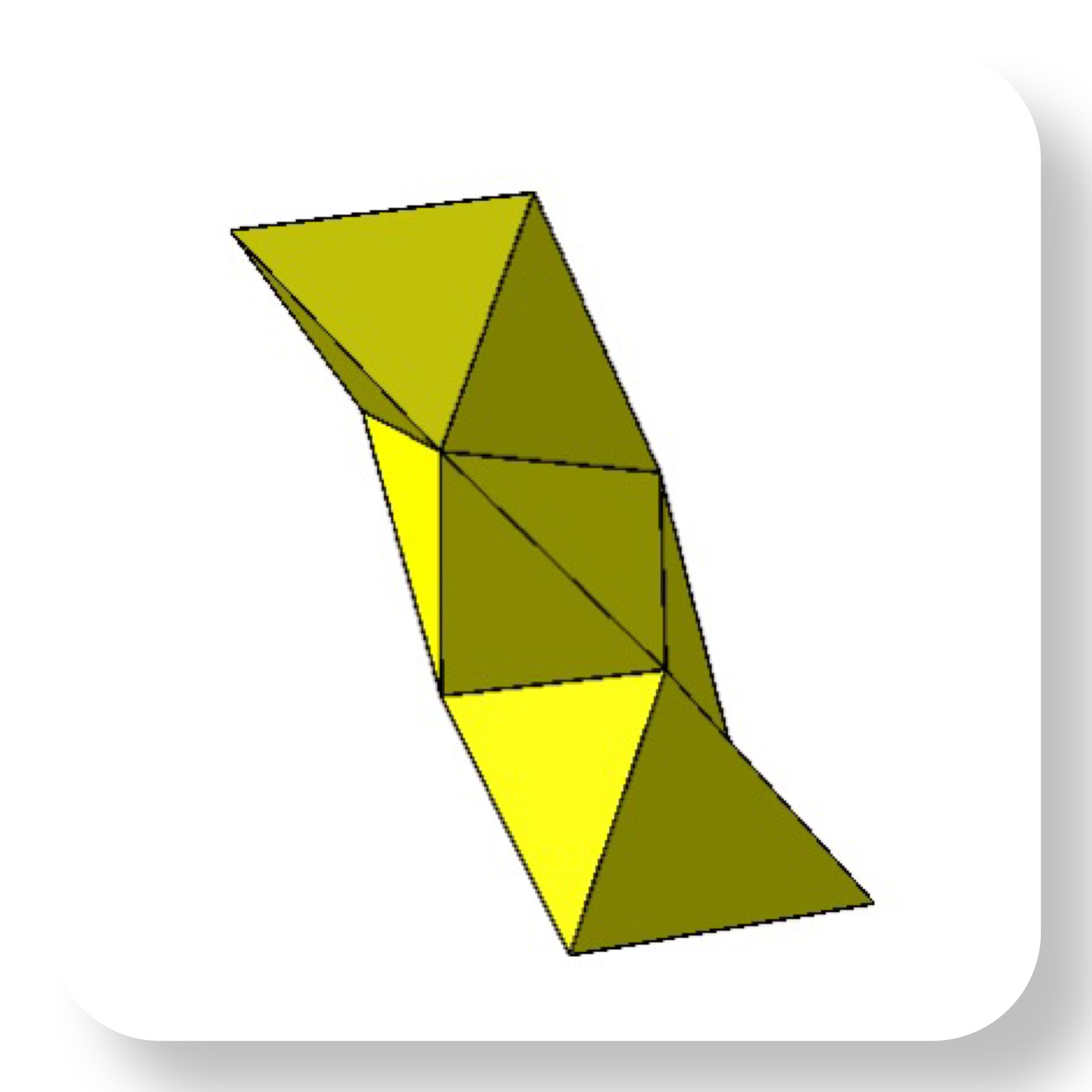}
\caption{}
  \end{subfigure}
  \caption{Different tetra-helices consisting of $5$ (a), $6$ (b) and $7$ (c) regular tetrahedra.}
  \label{fig:helix}
\end{figure} 
\end{example}

\begin{definition}
Let $(a,b,c)$ be a triple in $\Lambda.$ Furthermore, let $(t_n)_{n\in \mathbb{N}}$ be the sequence of natural numbers given by $t_n:= n+1$, and $(f_n)_{n\in \mathbb{N}}$ be another sequence defined by 
\begin{align*}
  f_1:=4, \phantom{a}f_2:=4, \phantom{a}f_3:=2, \phantom{a}f_4:=1, \phantom{a}f_5:=3, \phantom{a} f_6:=3, \phantom{a}f_7:=1, \phantom{a} f_8:=2
\end{align*}
and $f_{i+8j}=f_i,$ for $1\leq i \leq 8$ and $j \in \mathbb{N}.$ 
Then, we define the \emph{double tetra-helix} $\mathcal{D}_{2k+1}^{(a,b,c)}$ consisting of $2k+1$ tetrahedra by the following tetra-symbol
\[
\left(1_11_2\prod_{i=1}^{2(k-1)}{(t_i)}_{f_i} \right)^{(a,b,c)}.
\]
\end{definition}
In the following example, we illustrate different double tetra-helices and their corresponding tetra-symbol that arise from the above definition. 
\begin{example}
Let $\alpha,\beta,\gamma$ be triples that are given by $$(1, \tfrac{\sqrt{5}\sqrt{6}}{5}, \tfrac{3\sqrt{5}}{5}),\;(1, \tfrac{\sqrt{2706}\sqrt{418}}{902}, \tfrac{\sqrt{2706}}{41}),\;(1, \tfrac{\sqrt{3278}\sqrt{5478}}{3278}, \tfrac{3\sqrt{3278}}{149}),$$ respectively. Observe that the triples $\alpha,\beta,\gamma\in \Lambda.$ For $k=3,$ we obtain the double tetra-helix $\mathcal{D}_7^{\alpha}$ consisting of $7$ tetrahedra represented by the symbol 
$\left(1_11_22_43_44_25_1\right)^{\alpha}.$
Tetra-symbols of the double tetra-helices $\mathcal{D}_{11}^{\beta}$ and $\mathcal{D}_{15}^{\gamma}$ are given by 
\[
\left(1_11_22_43_4 4_2 5_1 6_3 7_3 8_1 9_2\right)^{\beta} \text{and} \; \; \left(1_11_22_43_4 4_2 5_1 6_3 7_3 8_1 9_2 10_4 11_4 12_2 13_1\right)^{\gamma},
\]
respectively. Visual representations of these examples are illustrated in Figure~\ref{fig:doublehelix}.

\begin{figure}[H]
\begin{subfigure}{.3\textwidth}
  \centering
\includegraphics[height=3.75cm]{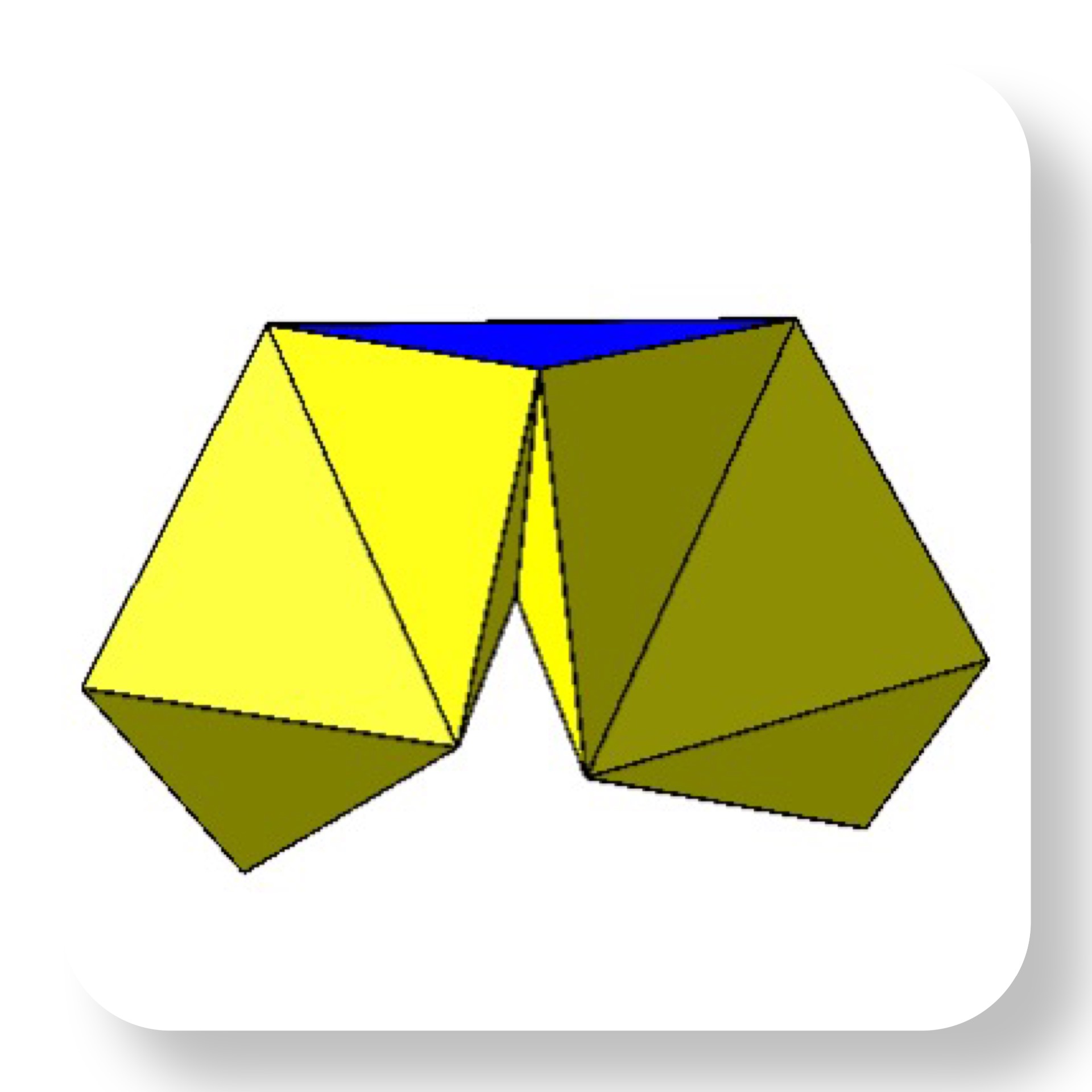}
\caption{}
\end{subfigure}
\begin{subfigure}{.3\textwidth}
  \centering
\includegraphics[height=3.75cm]{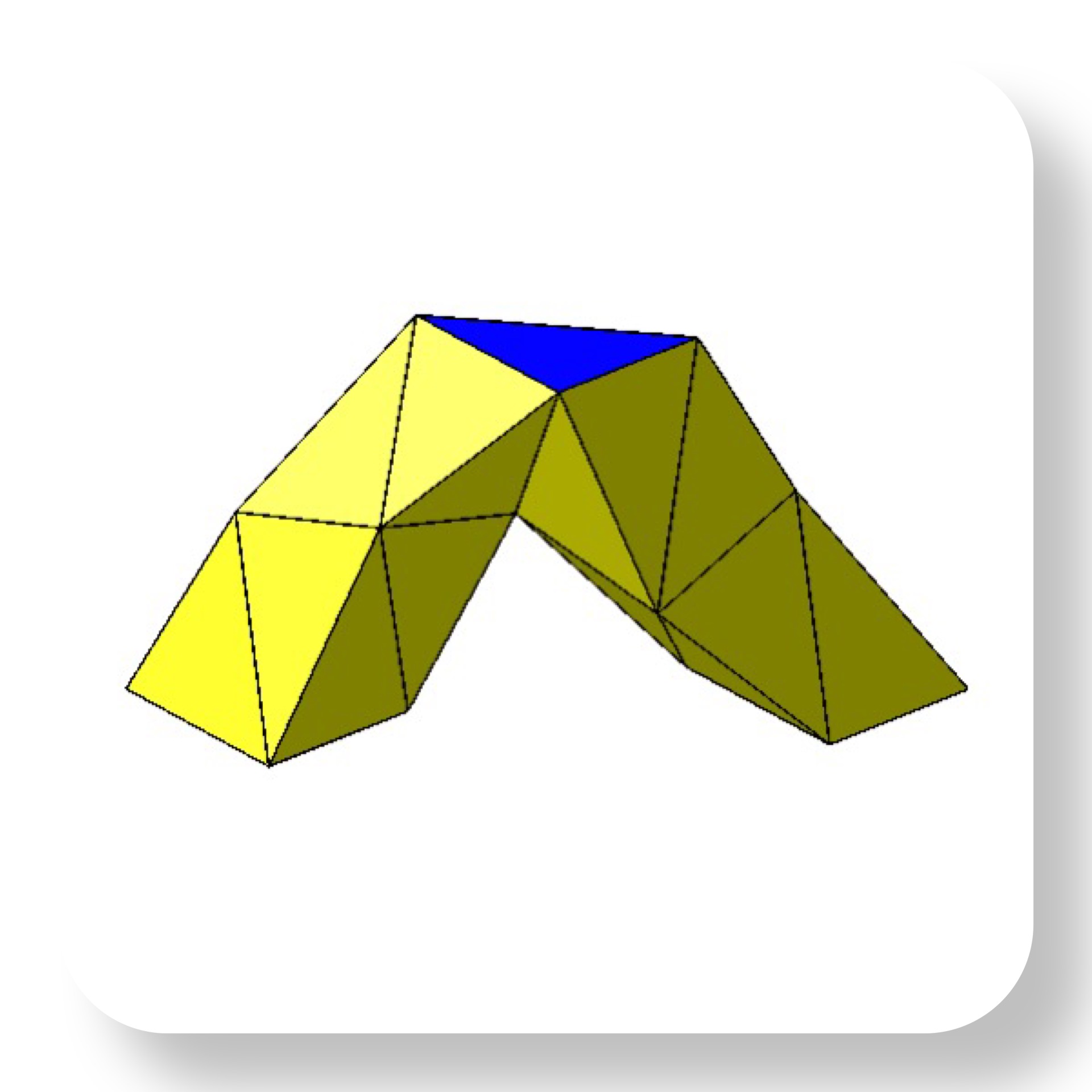}
\caption{}
\end{subfigure}
\begin{subfigure}{.3\textwidth}
  \centering
\includegraphics[height=3.75cm]{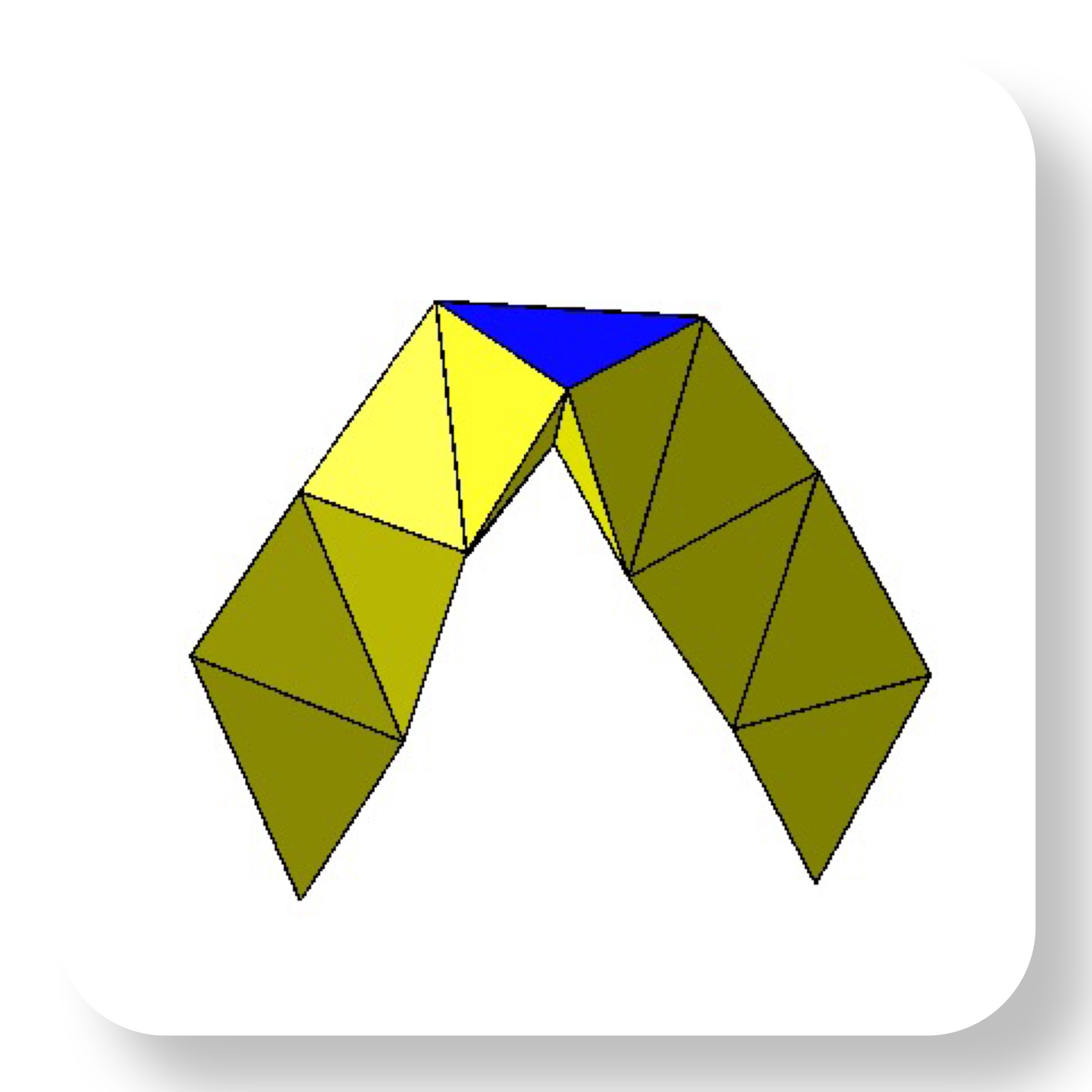}
\caption{}
\end{subfigure}
  \caption{Examples of double tetra-helices consisting of (a) $7$ tetrahedra with edge lengths $(1, \tfrac{\sqrt{5}\sqrt{6}}{5}, \tfrac{3\sqrt{5}}{5})$,\quad (b) $11$ tetrahedra with edge lengths $(1, \tfrac{\sqrt{2706}\sqrt{418}}{902}, \tfrac{\sqrt{2706}}{41})$ and \quad (c) $15$ tetrahedra with edge lengths $(1, \tfrac{\sqrt{3278}\sqrt{5478}}{3278}, \tfrac{3\sqrt{3278}}{149})$; The first tetrahedra of the different double tetra-helices are coloured in blue.}
  \label{fig:doublehelix}
\end{figure} 
\end{example}

\section{Relation between clusters of wild tetrahedra and simplicial surfaces}
\label{section:relation}

In this section, we establish a relationship between clusters of wild tetrahedra and simplicial surfaces which underpins our investigations. More precisely, we compute a simplicial surface that represents a partial combinatorial structure of a given cluster of wild tetrahedra. We shall refer to such a simplicial surface as a \emph{multi-tetrahedral surface}, see Definition~\ref{definition:relation}. Additionally, we elaborate on the advantages of using simplicial surfaces to describe clusters of wild tetrahedra. For instance, these multi-tetrahedral surfaces facilitate a complete description of the corresponding clusters of wild tetrahedra.
By exploiting the notion of the tetra-symbol, we assign a simplicial surface to a cluster of wild tetrahedra.
\begin{definition} \label{definition:relation}
  Let $\tau=(T_1,\ldots,T_n)$ be a cluster of wild tetrahedra and $\mathcal{V}_{\tau}=\{p_1,\ldots,p_m\}$ be the set of vertices of the given cluster. We recall from Definition~\ref{defintion:tetrahedron} that the sets of all edges and faces of $\tau$ are denoted by $\mathcal{E}_\tau$ and $\mathcal{F}_\tau$, respectively.
  For integers $1\leq i<j<k\leq m,$ we define the number $\eta_{i,j,k}$ as follows: 
  If $\conv(\{p_i,p_j,p_k\})\notin \mathcal{F}\tau,$ we define $\eta_{i,j,k}:=0.$ If $\conv(\{p_i,p_j,p_k\})$ forms a face in $\mathcal{F}_\tau,$ then we define $\eta_{i,j,k}$ as the number of tetrahedra containing this face as a subset. Note that $\eta_{i,j,k}\in \{0,1,2\}$ holds. If the set 
  \[
  X_\tau =\{1,\ldots,m\} \cup \{\{i,j\}\mid \conv(\{p_i,p_j\}) \in \mathcal{E}_{\tau}\} \cup \{\{i,j,k\}\mid \eta_{i,j,k}=1\} 
  \]
 satisfies the axioms of a simplicial surface, we call this surface a \emph{multi-tetrahedral surface}. Note that by excluding the $3$-subsets $\{i,j,k\}$ with $\eta_{i,j,k}=2,$ we remove the faces that share two tetrahedra in a cluster of wild tetrahedra. We call $X_{\tau}$ a \emph{multi-tetrahedral sphere}, if the simplicial surface $X_{\tau}$ satisfies $\chi(X_\tau)=2.$ We say that the multi-tetrahedral surface $X_{\tau}$ is \emph{proper} if the cluster $\tau$ is a chain. We refer to the surface $X_{\tau}$ as a \emph{multi-tetrahedral torus} if the surface satisfies $\chi(X_{\tau})=0.$
\end{definition}
For instance, with the above remark, we can exploit the tetra-helices and double tetra-helices defined in Section~\ref{subsection:chainsoftetra} to obtain proper multi-tetrahedral spheres with corresponding embeddings. 
 Different clusters of wild tetrahedra can give rise to isomorphic multi-tetrahedral surfaces, see Definition~\ref{def:isomorphic}.  Furthermore, if the edge lengths of the wild tetrahedra of the given cluster differ, then the edge lengths induce a wild-colouring of the corresponding multi-tetrahedral surface.
 \begin{remark}
 It can be shown that up to the permutation of colours, if a wild-colouring of a multi-tetrahedral surface exists, then it is unique, see \cite{fowler}. Furthermore, we know that for a given multi-tetrahedral sphere, there exists exactly one wild-colouring. 
 \end{remark}
In this paper, multi-tetrahedral surfaces are defined as the partial incidence geometries of clusters of wild tetrahedra. However, in \cite{simplicialsurfacebook} an alternative approach to define multi-tetrahedral spheres is used. In that context, multi-tetrahedral spheres are defined by iteratively attaching tetrahedra to a given simplicial tetrahedron without removing faces which occur in two tetrahedra.
For more information on multi-tetrahedral spheres, we refer to \cite{maRey,simplicialsurfacebook}.

\section{Constructing embeddings of multi-tetrahedral surfaces}
\label{section:embedcacti}

In this section, we study the embeddings of the wild-coloured simplicial tetrahedra that result in wild tetrahedra. Investigations dealing with the existence of wild tetrahedra in $\R^3$ satisfying certain edge length restrictions have been addressed in different studies. For instance, in \cite{acuteangleold} and \cite{menger} the authors establish that a wild tetrahedron with non-co-planar vertices has acute-angled triangles as faces.
Additionally, in \cite{ansgar} Strzelczyk shows that up to isometries there exists exactly one wild tetrahedron with edge lengths $(a,b,c)$ if and only if $(a,b,c)$ gives rise to a non-obtuse-angled triangle.
As described above, these publications focus on the existence and uniqueness of tetrahedra with respect to their edge lengths. Here, we describe a wild tetrahedron by providing explicit coordinates for the vertices of the given tetrahedron. These coordinates facilitate an easy description of the wild tetrahedra and thereby enable us to conduct more efficient computations in Maple.
The following theorem presents the described embedding of the wild-coloured simplicial tetrahedron.

\begin{theorem}
\label{lemma:embbedingtetrahedron}
Let $\mathcal{T}$ be a wild-coloured simplicial tetrahedron and $(1,b,c)\in \Lambda$ be edge lengths such that the resulting triangle is acute- or right-angled. This means that the following inequalities are satisfied:
\[
1+b^2\geq c^2, \phantom{a} b^2+c^2\geq 1, \phantom{a} 1+c^2\geq b^2.
\]
Then there exist $x,h \in [0,\infty)$ with $(x,h)\neq(0,0)$ such that the map $\phi:\mathcal{T}_0\to \mathbb{R}^3$ given by
\[
\bigg[\phi(v_1),\phi(v_2),\phi(v_3),\phi(v_4)\bigg]\
=\bigg[\Big(\tfrac{1}{2}, 0, 0\Big)^t,\Big(-\tfrac{1}{2}, 0, 0\Big)^t, \Big(\tfrac{x^2 - 1}{2(x^2 + 1)}, \tfrac{x}{x^2 + 1}, h\Big)^t, \Big(\tfrac{-x^2 + 1}{2(x^2 + 1)}, \tfrac{-x}{x^2 + 1}, h\Big)^t\bigg],
\]
 forms a strong $(1,b,c)$-embedding of $\mathcal{T}.$ 
\end{theorem}
\begin{proof}
Here, we prove that there exist $x$ and $h$ such that the polyhedron $\mathcal{P}$ resulting from $\phi$ is a wild tetrahedron consisting of faces with edge lengths $(1,b,c).$
Since \[\sqrt{(\tfrac{x^2 - 1}{2(x^2 + 1)})^2+(\tfrac{x}{x^2 + 1})^2}=\sqrt{(\tfrac{-x^2 + 1}{2(x^2 + 1)})^2+(\tfrac{-x}{x^2 + 1})^2}=\tfrac{1}{2},\] the vectors $(\tfrac{x^2 - 1}{2(x^2 + 1)}, \tfrac{x}{x^2 + 1})^t$ and $(\tfrac{-x^2 + 1}{2(x^2 + 1)}, \tfrac{-x}{x^2 + 1})^t$ lie on a circle around the origin with radius $\tfrac{1}{2}$ in $\mathbb{R}^2$. Hence, there exists a unique $\alpha \in [0,2\pi)$ such that the following equality holds: 
\[
\big[\phi(v_3),\phi(v_4)\big]
=\big[ \Big(\tfrac{1}{2}\cos(\alpha),\tfrac{1}{2}\sin(\alpha), h\Big)^t, \Big(-\tfrac{1}{2}\cos(\alpha),-\tfrac{1}{2}\sin(\alpha), h\Big)^t\big].
\]
Thus, it is sufficient to show that we can define $\alpha$ and $h$ such that $\phi$ is a strong $(1,b,c)$-embedding. For this, we make use of the wild-colouring of the simplicial tetrahedron illustrated in Figure~\ref{fig:wildcoloured_tetrahedra}. We recall that the edges coloured $1,2,3$ with respect to the illustrated wild-colouring are visualised with the colours red, green and blue, respectively. Thus, we need to find $\alpha\in [0,2\pi)$ and $h \in [0,\infty)$ such that
\begin{align*}
  \Vert \phi(v_1)-\phi(v_2)\Vert=1, \phantom{aa} \Vert \phi(v_3)-\phi(v_4)\Vert=1,\\
  \Vert \phi(v_1)-\phi(v_3)\Vert=b, \phantom{aa} \Vert \phi(v_2)-\phi(v_4)\Vert=b,\\
  \Vert \phi(v_1)-\phi(v_4)\Vert=c, \phantom{aa} \Vert \phi(v_2)-\phi(v_3)\Vert=c  \phantom{|}.
\end{align*}
This instance is illustrated in Figure~\ref{fig:circles}.
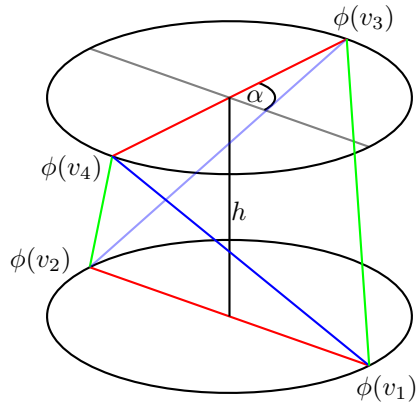
\begin{figure}[H]
    \centering
    

\begin{tikzpicture}
    \tdplotsetmaincoords{65}{40}
    \begin{tikzpicture}[xscale=0.8,yscale=0.8,tdplot_main_coords]
    \def\radius{3}

    \tdplotdrawarc[thick]{(0,0,0)}{\radius}{0}{360}{}{}

    \draw[thick, red] (-\radius,0,0) -- (\radius,0,0);
    
    \node at (3.8,-0.5,0) {$\phi(v_1)$};
    \node at (-3.8,-0.3,0) {$\phi(v_2)$};

    \tdplotdrawarc[thick]{(0,0,4)}{\radius}{0}{360}{}{}
  \tdplotdrawarc[thick]{(0,0,4)}{0.25*\radius}{0}{90}{}{}
    \draw[thick, red] (0,-\radius,4) -- (0,\radius,4);
        \draw[thick,opacity=0.5] (-\radius,0,4) -- (\radius,0,4);
    \node at (-0.3,3.8,4) {$\phi(v_3)$};
    \node at (-0.2,-3.8,4) {$\phi(v_4)$};

    \draw[thick, black] (0,0,0) -- (0,0,4) ;
    \node at (0.2,0,2) {$h$};
        \node at (0.3,0.3,4) {$\alpha$};
    \draw[thick, green] (-3,0,0) -- (0,-3,4) ;
    \draw[thick, blue] (3,0,0) -- (0,-3,4) ;
    \draw[thick, blue,opacity=0.4] (-3,0,0) -- (0,3,4) ;
    \draw[thick, green] (3,0,0) -- (0,3,4) ;
\node at (0,0,7.5) {$ $};
\node at (7,7,0) {$ $};
    \end{tikzpicture}
\end{tikzpicture}
    \vspace{-1cm}
    \caption{The polyhedron that results from the convex hull $\conv(\{\phi(v_1),\phi(v_2),\phi(v_3),\phi(v_4)\})$ with edges coloured in red, green and blue}
    \label{fig:circles}
\end{figure}
It is easy to verify that the edges $\conv(\{\phi(v_1),\phi(v_2)\})$ and $\conv(\{\phi(v_3),\phi(v_4)\})$ of $\mathcal{P}$ have length $1$. 
Additionally, we observe that the edges $\conv(\{\phi(v_2),\phi(v_3)\})$ and $\conv(\{\phi(v_2),\phi(v_4)\})$ arise from the edges $\conv(\{\phi(v_1),\phi(v_4)\})$ and $\conv(\{\phi(v_1),\phi(v_3)\})$ by a $180$ degree rotation around the vector $(0,0,1)^t$, respectively, i.e.\ by using the rotation-matrix
\begin{align*}
    \left(\begin{tabular}{ccc}
         -1&0&0  \\
         0&-1&0  \\
         0&0&1  \\
    \end{tabular}\right)
\end{align*}
to rotate the corresponding edges of $\mathcal{T}$. Since the Euclidean norm is invariant under isometries of the Euclidean $3$-space, it suffices to show that there exist $\alpha$ and $h$ such that
\[
\Vert \phi(v_1)-\phi(v_3)\Vert=b, \phantom{aa} \Vert \phi(v_1)-\phi(v_4)\Vert=c.
\]
Simplifying both equalities leads to 
\begin{align*}
  h=\sqrt{b^2-\tfrac{1}{2}+\tfrac{1}{2}\cos(\alpha)}, \phantom{aa}
  h=\sqrt{c^2-\tfrac{1}{2}-\tfrac{1}{2}\cos(\alpha)}.
\end{align*}
Thus, we conclude 
\[
\sqrt{b^2-\tfrac{1}{2}+\tfrac{1}{2}\cos(\alpha)}=\sqrt{c^2-\tfrac{1}{2}-\tfrac{1}{2}\cos(\alpha)}
\iff c^2-b^2=\cos(\alpha).
\]

Since $-1\leq c^2-b^2 \leq 1$, the last equation has a solution in $\alpha.$ Hence, we obtain an identity for $h$ and the result follows.
\end{proof}

The above theorem yields a wild tetrahedron whose faces consists of triangles with edge lengths 
\begin{align}
\label{lengthtetra}
\biggl(1,\sqrt{\tfrac{h^2x^2+x ^2+1}{x^2+1}},\sqrt{\tfrac{h^2x^2+x ^2+h^2}{x^2+1}}\biggr)=:\ell(x,h).
\end{align}
Moreover, we can use this theorem to describe an arbitrary wild tetrahedron with edge lengths $(a,b,c)\in \Lambda.$ This can be achieved by constructing the wild tetrahedron with edge lengths $(1,\tfrac{b}{a},\tfrac{c}{a})$ and then scaling the constructed tetrahedron to obtain the desired polyhedron.

Next, we describe a procedure to embed a wild-coloured multi-tetrahedral surface into $\R^3$. 
This task of computing embeddings of such surfaces turns out to be easier than the general embedding problem for simplicial surfaces. This is due to the fact that wild-coloured multi-tetrahedral surfaces can be embedded into $\R^3$ by considering only linear equations.

\begin{remark}
\label{remark:embVertex}
  Let $X$ be a wild-coloured multi-tetrahedral surface, $\phi$ a weak $(a,b,c)$-embedding of $X$ and $f=\{v_1,v_2,v_3\}$ a face of $X$. Furthermore, let $v\in X_0$ be the unique vertex that is adjacent to $v_1,v_2$ and $v_3,$ and $x_v\in \mathbb{R}^3$ the coordinate that is generated by reflecting the coordinate $\phi(v)$ through the plane containing $\phi(v_1),\phi(v_2),\phi(v_3).$
  Then, two weak $(a,b,c)$-embeddings $\phi_1,\phi_2$ of the surface 
  $Y=T^f(X)$, obtained by attaching a tetrahedron to the face $f$, are given by the following maps:
\begin{align*}
  \phi_1:Y_0 \to \mathbb{R}^3, w\mapsto 
  \begin{cases}
    \phi(w), & \text{ if }w\in X_0 \subset Y_0\\
    \phi(v), & \text{ otherwise }
  \end{cases},\\
   \phi_2:Y_0 \to \mathbb{R}^3, w\mapsto 
  \begin{cases}
    \phi(w), & \text{ if }w\in X_0 \subset Y_0\\
    x_v, & \text{ otherwise }
  \end{cases}.
\end{align*}
Since the equality $\phi_1(v')=\phi_1(v)$ holds for $v'\in Y_0 \setminus X_0$, the map $\phi_1$ is not a strong $(a,b,c)$-embedding of $Y.$  
Thus, the wild-coloured multi-tetrahedral surface $Y$ has exactly $2^k$ weak $(a,b,c)$-embeddings, where $k$ is the number of tetrahedra of the cluster corresponding to $Y$.
Whether $\phi_2$ gives rise to a strong $(a,b,c)$-embedding of $Y$ depends on the injectivity of the map $\phi_2$.
\end{remark}
Thus, we can embed a wild-coloured multi-tetrahedral surface $X$ into $\R^3$ by applying tetrahedral extensions to the wild-coloured simplicial tetrahedron and using the above remark in every step of the process.
For simplicity, we refer to a polyhedron that arises from a strong $(a,b,c)$-embedding $\phi$ of $X$ that is constructed as described above, by $X^{\phi}_{(a,b,c)}$. 
\begin{remark}
\label{remark:obtuse}
Since a wild-coloured multi-tetrahedral sphere $X$ can be constructed by applying tetrahedral extensions to the wild-coloured simplicial tetrahedron, we know that $X$ has a weak $(a,b,c)$-embedding if and only if the wild coloured simplicial tetrahedron has a weak $(a,b,c)$-embedding. This is exactly the case if 
the triangle with edge lengths $(a,b,c)$ is right- or acute-angled, see \cite{ansgar}.
\end{remark}
Further, we construct all strong embeddings of the wild-coloured simplicial double-tetrahedron by using the method introduced in Remark~\ref{remark:embVertex}.

\begin{example}
\label{example:doubletetra}
Let $X$ be the wild-coloured simplicial tetrahedron given by the faces
\[
\{\{v_1,v_2,v_3\},\{v_1,v_2,v_4\},\{v_1,v_3,v_4\},\{v_2,v_3,v_4\}\}.
\]
As established in Theorem~$\ref{lemma:embbedingtetrahedron},$ the map $\phi: X_0\to \mathbb{R}^3,$ where the images of the vertices are given by 
\[
\bigg[\phi(v_1),\phi(v_2),\phi(v_3),\phi(v_4)\bigg]\
=\bigg[\Big(\tfrac{1}{2}, 0, 0\Big)^t,\Big(-\tfrac{1}{2}, 0, 0\Big)^t, \Big(\tfrac{x^2 - 1}{2(x^2 + 1)}, \tfrac{x}{x^2 + 1}, h\Big)^t, \Big(\tfrac{-x^2 + 1}{2(x^2 + 1)}, \tfrac{-x}{x^2 + 1}, h\Big)^t\bigg],
\]
forms a strong $\ell(x,h)$-embedding of $X,$ see Equation (\ref{lengthtetra}).
Attaching a tetrahedron to the face $f=\{v_1,v_2,v_3\}$ of $X$ results in the wild-coloured simplicial double-tetrahedron $Y=T^f(X)$ whose faces are determined by
\[
\{\{v_1,v_2,v_4\},\{v_1,v_3,v_4\},\{v_2,v_3,v_4\},\{v_1,v_2,v_5\},\{v_2,v_3,v_5\},\{v_1,v_3,v_5\}\}.
\]
Thus, reflecting $\phi(v_4)$ through the plane that contains $\phi(v_1),\phi(v_2),\phi(v_3),$ as described in Remark~\ref{remark:embVertex}, leads to the following three dimensional coordinate:

\begin{align*}
  x_{v_4}:=\left(\tfrac{-x^2+1}{2},\tfrac{ 4h^2x(x^2 + 1)}{h^2(x^2+1)^2 + x^2 - \tfrac{x}{x^2 + 1}},\tfrac{-4hx^2}{h^2(x^2+1)^2 + x^2 + h}\right)^t.
\end{align*}
 Thus, we obtain a weak $\ell(x,h)$-embedding $\psi:Y_0\to \mathbb{R}^3$ that is determined by $\psi(v_i):=\phi(v_i)$ for $1\leq i \leq 4$ and $\psi(v_5):=x_{v_4}.$ 
With Maple we have been able to verify that  the map $\psi$ is injective for all $x,h>0$ and therefore forms a strong $\ell(x,h)$-embedding of $Y$ for $x,h>0$.
\end{example}
 A question that arises in this context is whether a wild-coloured multi-tetrahedral sphere can have more than one strong $(a,b,c)$-embedding. We address this question in the following lemma.
\begin{lemma}
\label{lemma:1embedding_1colouring}
Let $X$ be a wild-coloured multi-tetrahedral sphere. Then, up to isometries, there exists at most one strong $(a,b,c)$-embedding of $X$ into $\R^3.$
\end{lemma}
\begin{proof}
Since weak $(a,b,c)$-embeddings of wild-coloured multi-tetrahedral spheres do not exist if and only if the triangle with edge lengths $(a,b,c)$ is obtuse-angled (see Remark~\ref{remark:obtuse}), we only consider the case where $(a,b,c)$ form the edge lengths of a right- or acute-angled triangle. Thus, from Remark~\ref{remark:embVertex} we know that there exists at least one weak $(a,b,c)$-embedding of a wild-coloured multi-tetrahedral sphere.
Hence, we have to show that at most one of these weak $(a,b,c)$-embeddings can be a strong $(a,b,c)$-embedding of $X$. Let therefore $\tau=(T_1,\ldots,T_k)$ be a cluster of wild tetrahedra that corresponds to $X$, see Definition~\ref{definition:relation}. We prove the statement by induction over $k.$ If $k=1,$ then $X$ is isomorphic to the wild-coloured simplicial tetrahedron and up to isometry, the uniqueness of the corresponding strong $(a,b,c)$-embedding is established in Theorem~\ref{lemma:embbedingtetrahedron}. Next, let the statement hold for an arbitrary $k\geq 1.$
Then, $\tau'=(T_1,\ldots,T_{k-1})$ is another cluster of wild tetrahedra that gives rise to a corresponding wild-coloured simplicial sphere $X_{\tau '}.$ For simplicity, we assume $\mathcal{V}_{\tau'}\subset \mathcal{V}_{\tau}$ and $\mathcal{E}_{\tau'}\subset \mathcal{E}_{\tau}.$ 
If there exist two strong $(a,b,c)$-embeddings $\phi_1$ and $\phi_2$ of $X_\tau,$ then the restrictions of $\phi_1$ and $\phi_2$ to the vertices of ${X_{\tau'}},$ namely $\psi_1$ and $\psi_2,$ form strong $(a,b,c)$-embeddings of $X_{\tau'},$ respectively. Since there exists at most one strong $(a,b,c)$-embedding of $X_{\tau'},$ we know that up to isometry, $X_{\tau'}^{\psi_1}=X_{\tau'}^{\psi_2}$ holds. Thus, the result follows with Remark~\ref{remark:embVertex}.
\end{proof}
Next, we present another example of a strong embedding of a wild-coloured multi-tetrahedral surface that plays an important role in proving the existence of an infinite family of multi-tetrahedral tori with corresponding strong embeddings in Section~\ref{section:infinitefamily}. 

\begin{construction}
\label{example:helix6}

Our aim is to define a proper wild-coloured multi-tetrahedral sphere $X$ and construct a corresponding embedding such that two faces of $X$ being incident to vertices of degree $3$ are mapped onto triangles that lie on parallel planes in $\mathbb{R}^3$. This multi-tetrahedral sphere is then used in Theorem~\ref{thm:infinitefamily} to establish the existence of infinitely many multi-tetrahedral tori with strong embeddings.
In particular, we define this multi-tetrahedral sphere $X$ by the following faces:
\begin{align*}
  \{&\{v_1,v_2,v_3\},\{v_1, v_2, v_4\}, \{v_1, v_3, v_4\}, \{v_3, v_4, v_6\}, \{v_3, v_5, v_6\}, \{v_2, v_3, v_5\}, \{v_2, v_4, v_5\}\\
  &\{v_4, v_5, v_7\}, \{v_4, v_6, v_7\}, \{v_6, v_7, v_9\}, \{v_6, v_8, v_9\}, \{v_5, v_6, v_8\}, \{v_5, v_7, v_8\}, \{v_7, v_8, v_9\}\}.
\end{align*}
Note, $X$ is a multi-tetrahedral sphere that describes a tetra-helix consisting of $6$ wild tetrahedra.
Using Theorem~\ref{lemma:embbedingtetrahedron} and Remark~\ref{remark:embVertex} we can construct a desired strong embedding of $X$ 
such that the faces $\{v_1, v_2,v_3\}$ and $\{v_7,v_8,v_9\}$ yield triangles in $\R^3$ that lie on parallel planes. This strong embedding $\phi$ is constructed as follows:

Let $p(x)=x^4-10x^3-5x^2+6x+1$ be a real polynomial. By intermediate value theorem, there exists $r_1\in [\tfrac{175013}{262144}, \tfrac{21877}{32768}]$ satisfying $p(r_1)=0.$ Moreover, let $q(x)=x^2-r_1$ be a real polynomial. Again by the intermediate value theorem, there exists a real root $r_2$ of $q(x)$ such that $r_2\in [\tfrac{13387}{16384}, \tfrac{6695}{8192}].$
With this information, we obtain a strong $\ell(r_2,1)$-embedding $\phi$ of $X$, see Equation~\ref{lengthtetra}. 
Note that $\conv(\{\phi(v_1), \phi(v_2),\phi(v_3)\})$ and $\conv(\{\phi(v_7),\phi(v_8),\phi(v_9)\})$ indeed lie on parallel planes. For illustration, we give approximate values of the edge lengths $\ell(r_2,1)$ of $X^{\phi}_{\ell(r_2,1)}$, namely 
\begin{align*}
\ell(r_2,1)\approx (1,1.183362177,1.264774172),
\end{align*}
 and approximate values of the coordinates contained in $\phi(X_0).$ The values of the coordinates are given by 
\begin{align*}
[&\phi(v_1),\ldots,\phi(v_9)]\approx
  [(0.5, 0, 0)^t, (-0.5, 0, 0)^t, (-0.099, 0.49, 1)^t,  (0.1, -0.49, 1)^t, (-1.05, -0.36, 0.77)^t, \\&(-0.85, -0.07, 1.72)^t,(-0.45, -1.16, 1.5)^t, (-1.45, -1.16, 1.5)^t, (-1.05, -0.67, 2.5)^t].
\end{align*}
\end{construction}
\begin{figure}[H]
  \centering  
  \includegraphics[height=5cm]{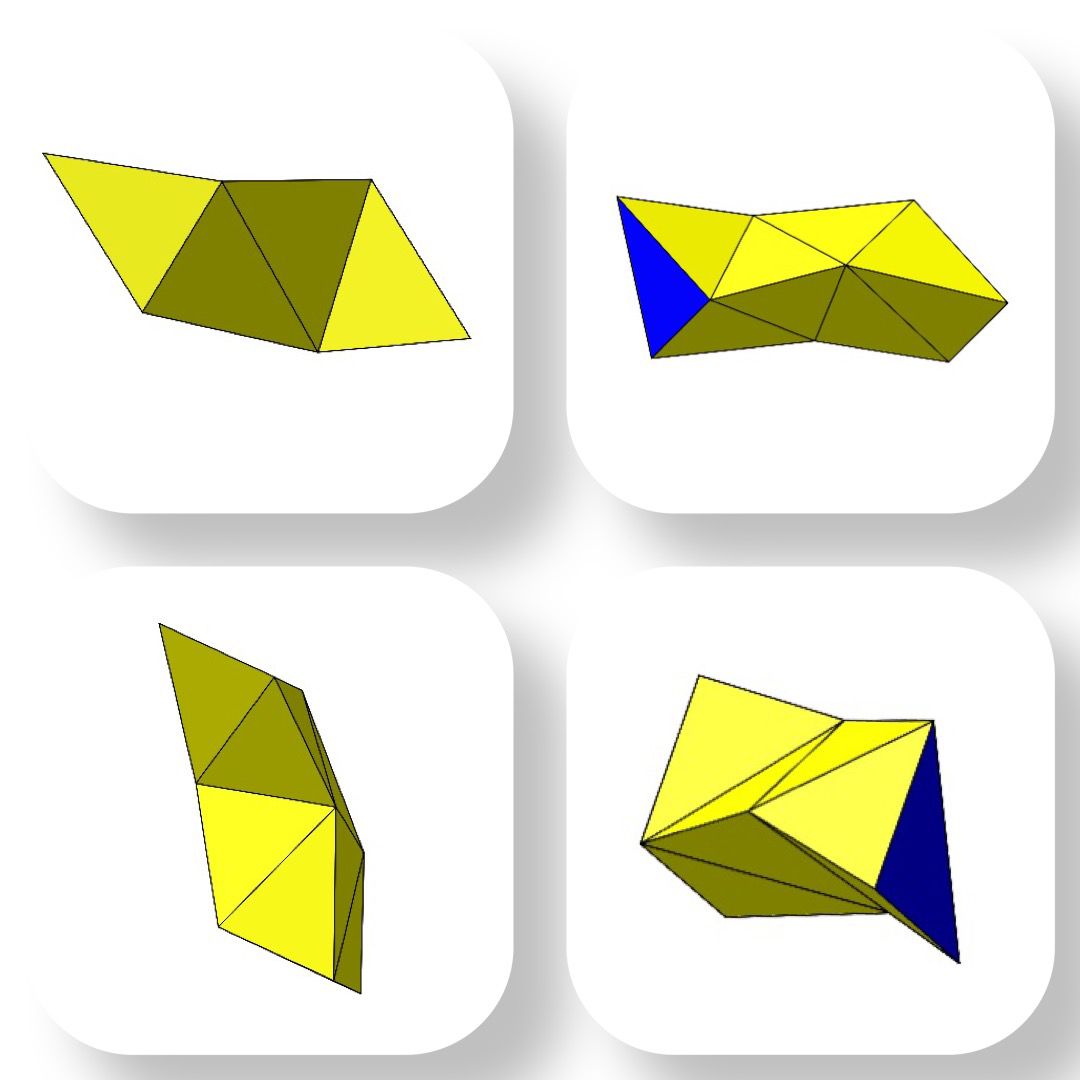}
  \caption{Various views of the polyhedron that results from embedding the wild-coloured multi-tetrahedral sphere corresponding to $\mathcal{H}_6^{\ell(r_2,1)}$ into $\mathbb{R}^3$.}
  \label{fig:6helix}
\end{figure} 

Now, we describe how affine maps can be used to obtain a wild-coloured simplicial surface with corresponding embedding from the embeddings of two given wild-coloured simplicial surfaces.
Broadly, we attach the faces of the polyhedron that arises from the embedding of the second wild-coloured simplicial surface to the faces of the polyhedron of the first one. Thus, we obtain a polyhedron whose incidence structure represents a wild-coloured simplicial surface. The procedure is described in Remark~\ref{remark:construction} and forms one of our key ingredients in constructing multi-tetrahedral tori in Sections ~\ref{section:construction}, \ref{section:infinitefamily} and \ref{section:highergenus}.

\begin{remark}
\label{remark:construction}
Let $X$ and $Y$ be wild-coloured simplicial surfaces with faces $f_X=\{v_1,v_2,v_3\}$ and $f_Y=\{w_1,w_2,w_3\},$ respectively. For simplicity, we assume that $X_0$ and $Y_0$ are disjoint sets. Furthermore, let the map $\phi_X: X_0\to \mathbb{R}^3$ be a strong $(a,b,c)$-embedding of $X$ and $\phi_Y: Y_0\to \mathbb{R}^3$ be a strong $(a,b,c)$-embedding of $Y$ with corresponding polyhedra $\mathcal{P}_X$ and $\mathcal{P}_Y,$ respectively. 
We assume that the following distance equations hold: 
\begin{align*}
  &\Vert \phi_X(v_1)-\phi_X(v_2)\Vert = \Vert \phi_Y(w_1)-\phi_Y(w_2)\Vert =a,\\
  &\Vert \phi_X(v_1)-\phi_X(v_3)\Vert = \Vert \phi_Y(w_1)-\phi_Y(w_3)\Vert =b, \\
    & \Vert \phi_X(v_2)-\phi_X(v_3)\Vert = \Vert \phi_Y(w_2)-\phi_Y(w_3)\Vert =c.
\end{align*}
Our goal is to construct a surface by applying an affine transformation (isometry) to the polyhedron $\mathcal{P}_Y,$ so that the polyhedron $\mathcal{P}_X$ and the transformed polyhedron have some faces in common, namely the faces of the two polyhedra that are contained in the plane that is determined by $\phi_X(v_1),\phi_X(v_2),\phi_X(v_3)$.
In particular, we define
\begin{align*}
&v:=\phi_X(v_2) - \phi_X(v_1), \phantom{aa}v':=\phi_X(v_3)- \phi_X(v_1),\\
&w:=\phi_Y(w_2)- \phi_Y(w_1), \phantom{a} w':=\phi_Y(w_3)- \phi_Y(w_1).
\end{align*}
Thus, there exist two unit normal vectors $n_X,n_Y\in \mathbb{R}^3$ such that $n_X$ is orthogonal to the vectors $v$ and $v',$ and $n_Y$ is orthogonal to the vectors $w$ and $w'$.
Since the vectors $v,v',n_X$ and $w,w',n_Y$ are linearly independent in $\mathbb{R}^3$, the matrices $M_X=(v,v',n_X)\in \mathbb{R}^{3\times3}$ and $M_Y=(w,w',n_y)\in \mathbb{R}^{3\times3}$ satisfy $\det(M_1)=\det(M_2)\neq 0$. With $M:=M_X{M_Y}^{-1}$, we can give the following map:
\[
\psi:X_0 \cup Y_0 \to\mathbb{R}^3, \,z\mapsto \left\{
\begin{array}{ll}
    \phi_X(z), & z\in X_0\\
    M(\phi_Y(z)-\phi_Y(w_1))+\phi_X(v_1), &z \in Y_0\\ 
\end{array} \right\}.
\] 

Thus, if the set
\begin{align*}
  Z=&\underset{=:Z_0}{\underbrace{(\psi^{-1}(\psi(X_0\cup Y_0)))}} \cup \underset{=:Z_1}{\underbrace{(\psi^{-1}(\psi(X_1\cup Y_1)))}} \cup \underset{=:Z_2}{\underbrace{(\psi^{-1}(\psi(X_2 \cup Y_2))\setminus \psi^{-1}(\psi(X_2) \cap \psi(Y_2))) }}
\end{align*}
defines a wild-coloured simplicial surface, we say that $Z$ is constructed by attaching the surface $Y$ onto $X$ via the faces $f_X$ and $f_Y$. If $X$ is isomorphic to $Y$ and the polyhedron $\mathcal{P}_X$ arises from applying isometries to the polyhedron $\mathcal{P}_Y,$ then we say that $Z$ is constructed by reflecting $X$ at the plane that is determined by $\phi_X(v_1),\phi_X(v_2),\phi_X(v_3).$

\end{remark}

Depending on the choice of the normal vectors we can obtain two different strong $(a,b,c)$-embeddings of the surface $Z.$
For instance, if we take two vertex disjoint wild-coloured simplicial tetrahedra with corresponding strong embeddings, then the construction in Remark~\ref{remark:construction} generates a wild-coloured simplicial double-tetrahedron with corresponding strong embedding. In general, if $X$ is a wild-coloured simplicial surface, $Y$ is the wild-coloured simplicial tetrahedron and the constructed surface $Z$ is well-defined, then $Z$ is the surface that we obtain from attaching the tetrahedron $Y$ to $X,$ i.e.\ $Z=T^f(X)$ for a face $f\in X.$

\section{Constructing toroidal polyhedra}
\label{section:construction}
We deal with the construction of wild-coloured multi-tetrahedral surfaces that can be embedded into the Euclidean three-dimensional space as toroidal polyhedra in this section. For this construction, we make use of proper wild-coloured multi-tetrahedral spheres. Here, we present the criteria to determine whether a proper multi-tetrahedral sphere can be exploited to construct a toroidal polyhedron (see Definition~\ref{definition:relation}). We are interested in classifying the resulting toroidal polyhedra with respect to self-intersections and reflection symmetries. From this section, we assume that the presented multi-tetrahedral surfaces are wild-coloured and proper. In the following, we present a formal definition of a self-intersecting polyhedron. 

\begin{definition}
\label{def:selfintersection}
  Let $X$ be a simplicial surface and $\phi$ be a strong $(a,b,c)$-embedding of $X.$
  We say that $\mathcal{P}:=X_{(a,b,c)}^\phi$ is \emph{self-intersecting}, if there exist an edge 
  $e\in X_1$ and a face $f\in X_2$ with $e\cap f=\emptyset$ such that $(\conv(\phi(e))\cap \conv(\phi(f)))\setminus (\phi(e)\cup \phi(f))\neq  \emptyset$ holds.
  Otherwise, we say that $\mathcal{P}$ is \emph{non self-intersecting}.
\end{definition}
For instance, the great icosahedron also known as the stellated icosahedron is an example of a polyhedron with self-intersections, see \cite{icosahedron}.

\subsection{Toroidal polyhedra with reflection symmetry}
\label{subsection:withmirrorsymmetry}
In this subsection, we describe the construction of multi-tetrahedral tori with corresponding strong embeddings by imposing reflection symmetries on the resulting toroidal polyhedra. The idea is to construct a polyhedron corresponding to a multi-tetrahedral sphere such that two faces of the polyhedron are contained in the same plane (see Definition~\ref{def:propertyt1}). Then we compute the reflection image of the given polyhedron with respect to the derived plane and construct a toroidal polyhedron by combining these two polyhedra. First, we give the following definition.

\begin{definition}
\label{def:propertyt1}
  Let $X$ be a proper multi-tetrahedral sphere. Since $X$ is proper, there are exactly two vertices of degree $3$, namely $v$ and $w.$ Let $f_{v}=\{v,v_2,v_3\}$ be a face incident to the vertex $v$ and $f_{w}=\{w,w_2,w_3\}$ a face incident to the vertex $w$. 
  Moreover, let $\phi$ be a strong $(a,b,c)$-embedding of $X$ and $\hat{n}\in \R^3$ be the vector that is orthogonal to $(\phi(v_2)-\phi(v))$ and $(\phi(v_3)-\phi(v)).$ Since the three vectors $(\phi(v_2)-\phi(v)),(\phi(v_3)-\phi(v))$ and $\hat{n}$ are linearly independent, we know that for every $v'\in X_0$ there exist unique $r_{v'},s_{v'},t_{v'} \in \R$ such that $\phi(v')=\phi(v)+r_{v'}(\phi(v_2)-v)+s_{v'}(\phi(v_3)-v)+t_{v'}\hat{n}.$ We say that the tuple $(X,v,w,f_v,f_w,\phi)$ satisfies the \emph{property \textit{(T1)}}, if
  \begin{enumerate}
    \item for the plane $P= \phi(v)+\langle \phi(v_2)- \phi(v),\phi(v_3)- \phi(v)\rangle,$ the following equality holds:
    \[
    \{v'\in X_0\mid \phi(v')\in P\}=\{v,v_2,v_3,w,w_2,w_3\},
    \]
    \item all edges $e$ of $X$ satisfying $\conv(\phi(e))\subseteq P$ have to be incident to $f_v$ or $f_w,$ and 
    \item  for the set $M:=\{(r_{v'},s_{v'},\vert t_{v'}\vert )\mid v'\in X_0\}$ the equality $\vert M \vert =\vert X_0 \vert$ holds.
  \end{enumerate}
 
\end{definition}
With this definition, we are able to describe a construction of a multi-tetrahedral torus, see Proposition~\ref{prop:torus1}.
\begin{prop}\label{prop:torus1}
Let $X$ be a proper multi-tetrahedral sphere. If there exist vertices $v,w\in X_0$, faces $f_v\in X_2(v),f_w\in X_2(w)$ and a strong $(a,b,c)$-embedding $\phi$ of $X$ such that $(X,f_v,f_w,v,w,\phi)$ satisfies property (T1), then there exists a multi-tetrahedral torus $Y$ with $\vert X_2 \vert -4$ faces.
\end{prop}
\begin{proof}
Since $(X,v,w,f_v,f_w,\phi)$ satisfies property \textit{(T1)}, we know that using Remark~\ref{remark:construction} and reflecting $X$ at the plane that is determined by the face $\phi(f_v)$ results in a proper multi-tetrahedral torus $Y$ with $\vert Y_2 \vert =\vert X_2 \vert -4.$ 
\end{proof}
Given a prescribed number of faces $n$, we present an algorithm to construct all toroidal polyhedra with reflection symmetries and $2n-4$ faces that arise from the proper multi-tetrahedral spheres with $n$ faces.
This algorithm exploits Proposition~\ref{prop:torus1}. In particular, we make use of proper multi-tetrahedral spheres to construct multi-tetrahedral tori with corresponding embeddings and examine the resulting toroidal polyhedra with respect to self-intersections. 

\begin{algorithm}[H]
    \caption{To construct a set of toroidal polyhedra with reflection symmetries}
    \label{alg:torus1}
\KwData{A natural number $n$ satisfying $n \mod 2 \equiv 0$}
\KwResult{The set of all toroidal polyhedra whose corresponding multi-tetrahedral tori consist of $2n-4$ faces}
$\texttt{properSpheres} \gets $ set of all proper multi-tetrahedral spheres with $n$ faces\;
$\texttt{toroidalPolSelfInt}\gets \{\}$\;
$\texttt{toroidalPolNonSelfInt}\gets \{\}$\;
\For{$X \; \text{in} \; \texttt{properSpheres}$}{
   $v\gets v' \in X_0$ with $\deg(v')=3$\;
   $w\gets v'\in X_0\setminus \{v\}$ with $\deg(v')=3$\;
\For{$f_v \; \text{in} \; X_2(v)$}
{
  \For{$f_w \; \text{in} \; X_2(w)$}
  {
    \If{there exists a strong $(a,b,c)$-embedding $\phi$ of $X$ such that $(X,v,w,f_v,f_w,\phi)$ satisfies property (T1)}
    {
      $Y\gets $the surface that is constructed by reflecting $X$ through the plane $P$ that is defined by $\phi(f_v)$ as described in Remark~\ref{remark:construction}\; 
      \If{$Y^\phi_{(a,b,c)}$ has self intersections}
      {
        $\texttt{toroidalPolSelfInt}\gets \texttt{toroidalPolSelfInt}\cup \{Y\}$\;
      }
      \Else{$\texttt{toroidalPolNonSelfInt}\gets \texttt{toroidalPolNonSelfInt}\cup \{Y\}$ }
    }   
  }
}
}
\Return $\texttt{toroidalPolSelfInt},\texttt{toroidalPolNonSelfInt}$\;
\end{algorithm}
\vspace{0.5cm}

The correctness of the above algorithm follows from the proof of Proposition~\ref{prop:torus1}.
Examples of toroidal polyhedra with reflection symmetries are illustrated in Figure~\ref{fig:smallest_torus} and Figure~\ref{fig:torusmirrorselfint}. The polyhedron shown in Figure~\ref{fig:torusmirrorselfint} is an example of a toroidal polyhedron with self-intersections.
\begin{figure}[H]
  \centering  
  \includegraphics[height=5cm]{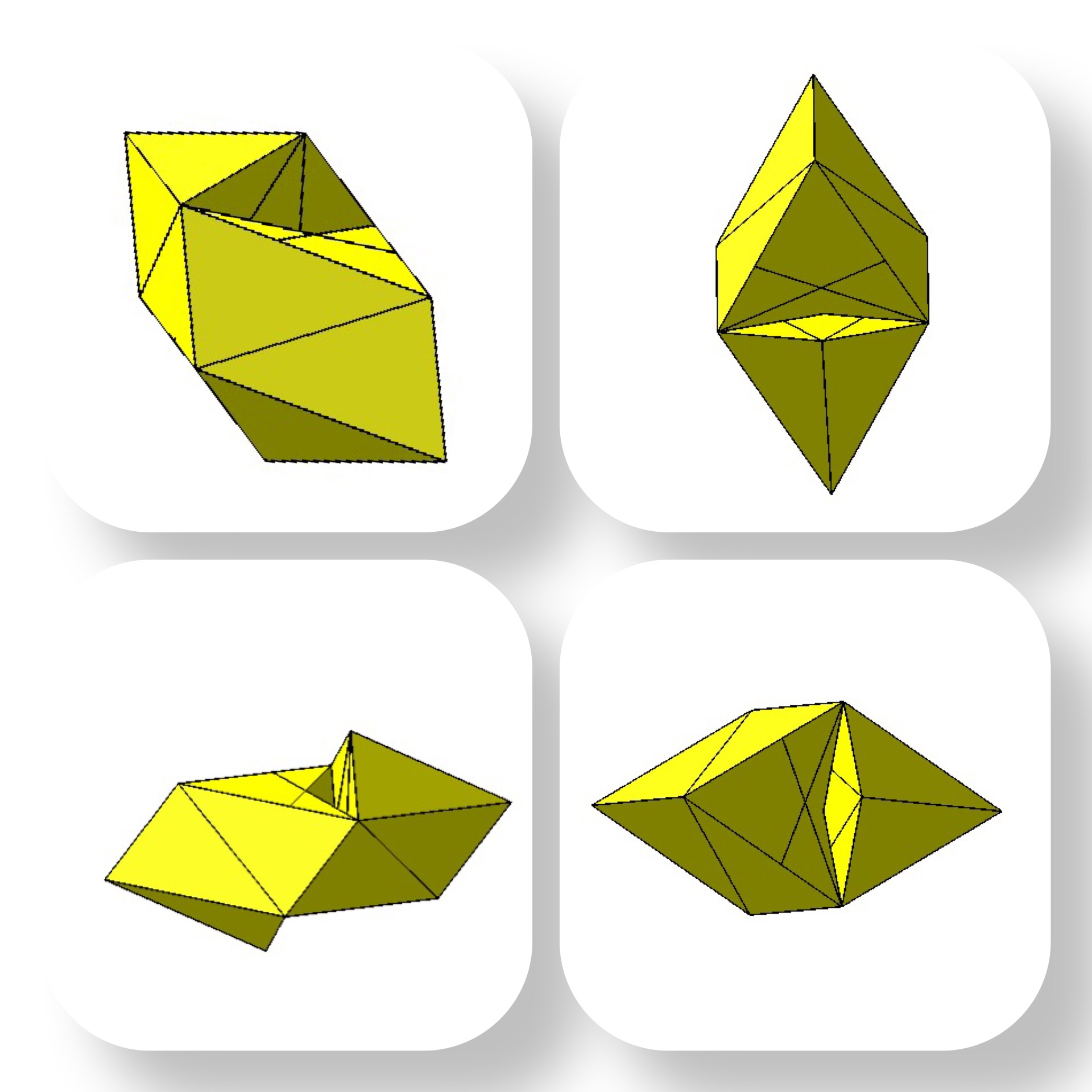}
  \caption{Various views of a toroidal polyhedron with self-intersections and reflection symmetry resulting from a strong $(1, \tfrac{\sqrt{150 - 30\sqrt{5}}}{10}, \tfrac{\sqrt{5}-1}{2})$-embedding of a multi-tetrahedral torus with $28$ faces}
  \label{fig:torusmirrorselfint}
\end{figure} 

\begin{example}
\label{example:doublehelix11}
  Another example of a multi-tetrahedral sphere that gives rise to a toroidal polyhedron is the simplicial surface that corresponds to the double tetra-helix $\mathcal{D}_{11}^{\ell(r_2,1)},$ where $\ell(r_2,1)$ is given as in Example~\ref{example:helix6}. Using \textsc{Maple}, we verified that there exist two faces that are incident to different vertices of degree $3$ such that all their embedded vertices lie on exactly one plane $P$. Additionally, the other two conditions described in Definition~\ref{def:propertyt1} hold. Thus, by reflecting the surface corresponding to $\mathcal{D}_{11}^{\ell(r_2,1)}$ through this plane $P$, we obtain the toroidal polyhedron shown in Figure~\ref{fig:torus12}. Note that this toroidal polyhedron does not contain any self-intersections. 
\end{example}
\begin{figure}[H]
  \centering    \includegraphics[height=5cm]{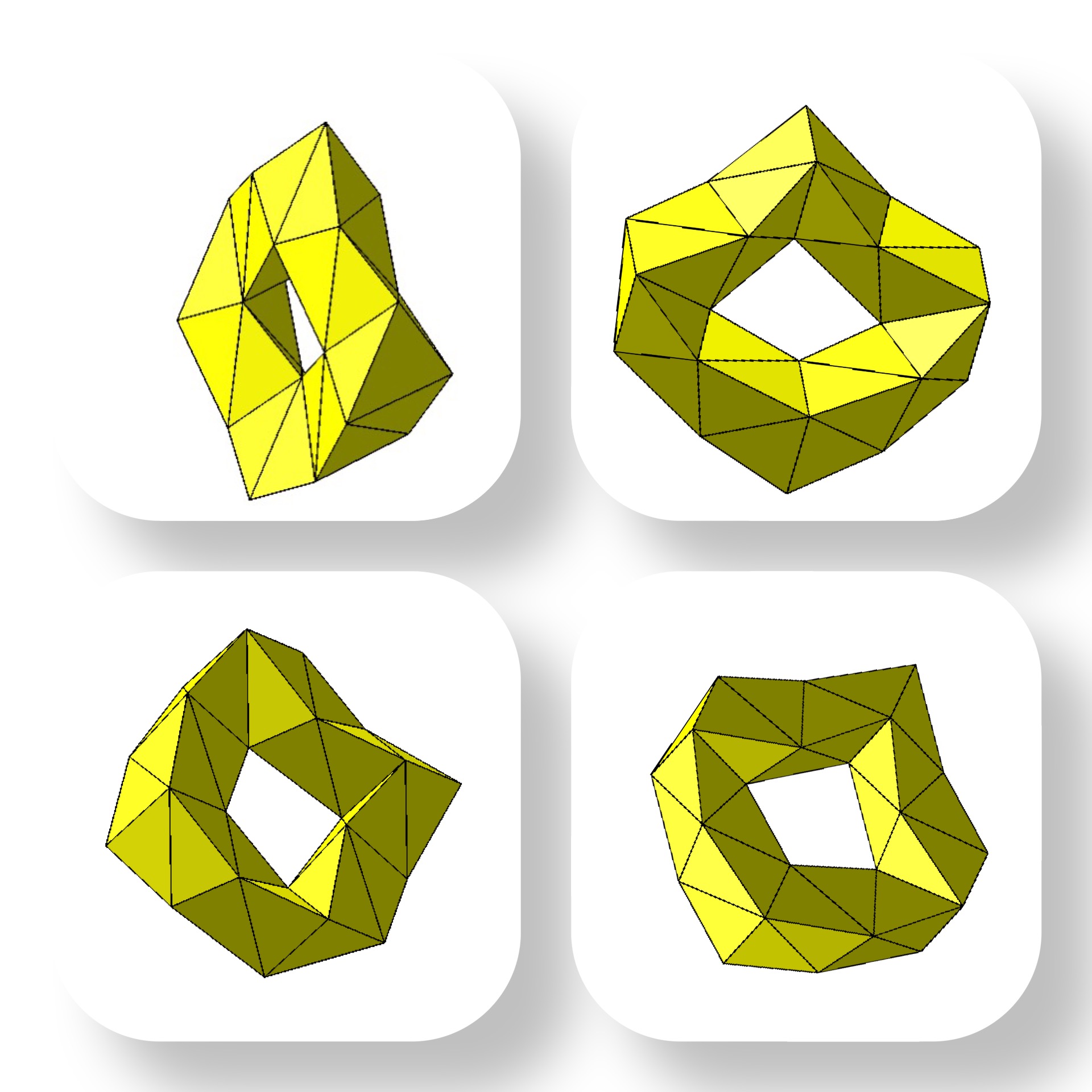}
  \caption{Various views of a toroidal polyhedron without self-intersections and with reflection symmetry consisting of faces with edge lengths $\ell(r_2,1)$ that results from embedding a multi-tetrahedral torus with 44 faces.}
  \label{fig:torus12}
\end{figure} 

\subsection{Toroidal polyhedra without reflection-symmetry}
\label{subsection:withoutmirrorsymmetry}

In this subsection, we present a procedure to construct multi-tetrahedral tori without enforcing reflection symmetries on the resulting toroidal polyhedra, i.e.\ we aim to construct a multi-tetrahedral torus without attaching a copy to a given multi-tetrahedral sphere in the sense of Remark~\ref{remark:construction} as described in Section~\ref{subsection:withmirrorsymmetry}.

\begin{definition}
\label{def:propertyt2}
  Let $X$ be a proper multi-tetrahedral sphere. Since $X$ is proper, it contains exactly two vertices of degree $3$, namely $v$ and $w.$ Let $f_v=\{v,v_2,v_3\}$ be a face that is incident to $v$ and $f_w=\{w,w_2,w_3\}$ a face that is incident to $w$. Moreover, let $\phi$ be a weak $(a,b,c)$-embedding of $X.$ We say that the tuple $(X,v,w,f_v,f_w,\phi)$ satisfies the property \textit{(T2)}, if 
  \begin{enumerate}
\item for every $x\in \{v,v_2,v_3\}$, there exists $y\in \{w,w_2,w_3\}$ such that $\phi(x)=\phi(y)$,
\item the restriction of $\phi$ to $X_0\setminus\{v,v_2,v_3\}$ is injective.
\end{enumerate}
\end{definition}
This definition gives rise to the following proposition that establishes a construction method for toroidal polyhedra without enforcing reflection symmetries.
\begin{prop}
\label{prop:torus2}
Let $X$ be a proper multi-tetrahedral sphere. If there exist vertices $v,w\in X_0$, faces $f_v\in X_2(v),f_w\in X_2(w)$ and a weak $(a,b,c)$-embedding $\phi$ of $X$ such that $(X,v,w,f_v,f_w,\phi)$ satisfies property (T2), then there exists a multi-tetrahedral torus $Y$ with $\vert X_2 \vert -2$ faces.
\end{prop}
\begin{proof}
Let $f_v$ and $f_w$ be given by $f_v=\{v,v_2,v_3\}$ and $f_w=\{w,w_2,w_3\}$. Since $(X,v,w,f_v,f_w,\phi)$ satisfies property \textit{(T2)}, we know that for every $x\in \{v,v_2,v_3\}$, there exists $y\in \{w,w_2,w_3\}$ such that $\phi(x)=\phi(y)$. Thus, we obtain a multi-tetrahedral torus by identifying the three pairs of coinciding vertices and edges of the surface $X$ and deleting the two faces $f_v$ and $f_w$. Hence, the resulting multi-tetrahedral torus $Y$ satisfies $\vert Y_2 \vert =\vert X_2 \vert -2.$ 
\end{proof}

Next, we present an algorithm to construct all toroidal polyhedra without imposing reflection symmetries that arise from proper multi-tetrahedral spheres. Given the number of faces $n$, we obtain such toroidal polyhedra whose corresponding multi-tetrahedral tori consist of $n-2$ faces by using Proposition~\ref{prop:torus2}.

\begin{algorithm}[H]
    \caption{To construct a set of toroidal polyhedra}
    \label{alg:torus2}
\KwData{A natural number $n$ satisfying $n \mod 2 \equiv 0$}
\KwResult{The set of all toroidal polyhedra whose corresponding multi-tetrahedral tori consist of $n-2$ faces}
$\texttt{properSpheres} \gets $ set of all proper multi-tetrahedral spheres with $n$ faces\;
$\texttt{tPolSelfInt}\gets \{\}$\;
$\texttt{tPolNonSelfInt}\gets \{\}$\;
$\texttt{tPolSelfIntMirror}\gets \{\}$\;
$\texttt{tPolNonSelfIntMirror}\gets \{\}$\;
\For{$X \; \text{in} \; \texttt{properSpheres}$}{
   $v\gets v'\in X_0$ with $\deg(v')=3$\;
   $w\gets v'\in X_0\setminus \{v\}$ with $\deg(v')=3$\;
\For{$f_v \; \text{in} \; X_2(v)$}
{
  \For{$f_w \; \text{in} \; X_2(w)$}
  {
    \If{there exists a weak $(a,b,c)$-embedding $\phi$ of $X$ such that $(X,v,w,f_v,f_w,\phi)$ satisfies property \textit{(T2)}}
    {
      $Y\gets $the surface that is constructed by identifying all elements in $X_i(f_v)$ with elements in $X_i(f_w)$ for $i=1,2$ and deleting the faces $f_v$ and $f_w$\; 
      \If{$Y^\phi_{(a,b,c)}$ has self intersections}
      { \If{$Y^\phi_{(a,b,c)}$ has mirror symmetries}
      {
      $\texttt{tPolSelfIntMirror}\gets \texttt{tPolSelfIntMirror}\cup \{Y\}$\;
      }
      \Else{
      $\texttt{tPolSelfInt}\gets \texttt{tPolSelfInt}\cup \{Y\}$
      }
      }
      \Else{
      \If{$Y^\phi_{(a,b,c)}$ has mirror symmetries}
      {
      $\texttt{tPolNonSelfIntMirror}\gets \texttt{tPolNonSelfIntMirror}\cup \{Y\}$\;
      }
      \Else{
      $\texttt{tPolNonSelfInt}\gets \texttt{tPolNonSelfInt}\cup \{Y\}$
      }
       }
    }   
  }
}
}
\Return $\texttt{tPolSelfIntMirror},\texttt{tPolSelfInt},\texttt{tPolNonSelfIntMirror},\texttt{tPolNonSelfInt}$\;

\end{algorithm}

Note that the toroidal polyhedra that result from the above procedure can still have reflection symmetries, although we do not enforce this symmetry on the corresponding polyhedron. The correctness of the above algorithm follows from Proposition~\ref{prop:torus2}.
\section{Infinite family of multi-tetrahedral tori}
\label{section:infinitefamily}
In this section, we present a family of infinitely many multi-tetrahedral surfaces that can be embedded into $\R^3$ as toroidal polyhedra. More precisely, we show that this family is derived from the family of double tetra-helices. We choose $r_2$ to be given as described in Example~\ref{example:helix6}.

\begin{theorem}
\label{thm:infinitefamily}
   For $n\in \mathbb{N}_0,$ let $X$ be the multi-tetrahedral sphere that corresponds to the double tetra-helix $\mathcal{D}_{11+12n}^{\ell(r_2,1)}$ and $\phi$ be a strong $\ell(r_2,1)$-embedding of $X$. Then there exist $v,w\in X_0$ and faces $f_v\in X_2(v),f_w\in X_2(w)$ such that the tuple $(X,v,w,f_v,f_w,\phi)$ satisfies the property \textit{(T1)} (see Definition~\ref{def:propertyt1}).
\end{theorem}
\begin{proof}
We prove the statement by induction over $n$. 
For $n=0,$ we know that $X$ is the multi-tetrahedral sphere that arises from the double tetra-helix $\mathcal{D}_{11}^{\ell(r_2,1)}.$
From Example~\ref{example:doublehelix11}, we conclude that there exist vertices $v,w\in X_0$ of degree $3$ and faces $f_v\in X_2(v),f_w\in X_2(w)$ such that $(X,v,w,f_v,f_w,\phi)$ satisfies property \textit{(T1)}. 
Now, let the statement hold for an arbitrary $n\in \mathbb{N}_{0}$. Let $X$ and $Y$ be the multi-tetrahedral spheres that correspond to the double tetra-helices $\mathcal{D}_{11+12n}^{\ell(r_2,1)}$ and $\mathcal{D}_{11+12(n+1)}^{\ell(r_2,1)}$, respectively. Furthermore, let $\phi_X$ and $\phi_Y$ be the strong $\ell(r_2,1)$-embeddings of $X$ and $Y$ that give rise to the polyhedra $X_{\ell(r_2,1)}$ and $Y_{\ell(r_2,1)}$, respectively. 
By the inductive hypothesis, we know that there exist vertices $v,w\in X_0$, and faces
$f_v\in X_2(v)$ and $f_w\in X_2(w)$ such that the tuple $(X,v,w,f_v,f_w,\phi_X)$ satisfies property \textit{(T1)}. 
Note that $Y_{\ell(r_2,1)}$ is constructed from combining $X_{\ell(r_2,1)}$ with two copies of the polyhedron arising from the tetra-helix $ \mathcal{H}^{\ell(r_2,1)}_6,$ namely $\mathcal{P},\mathcal{P'}.$ Figure~\ref{fig:infinitefamily_helix} illustrates this construction and also indicates that the corresponding embedding $\phi_{Y}$ of $Y$ is indeed injective.
\begin{figure}[H]
  \centering  
  \includegraphics[height=4.5cm]{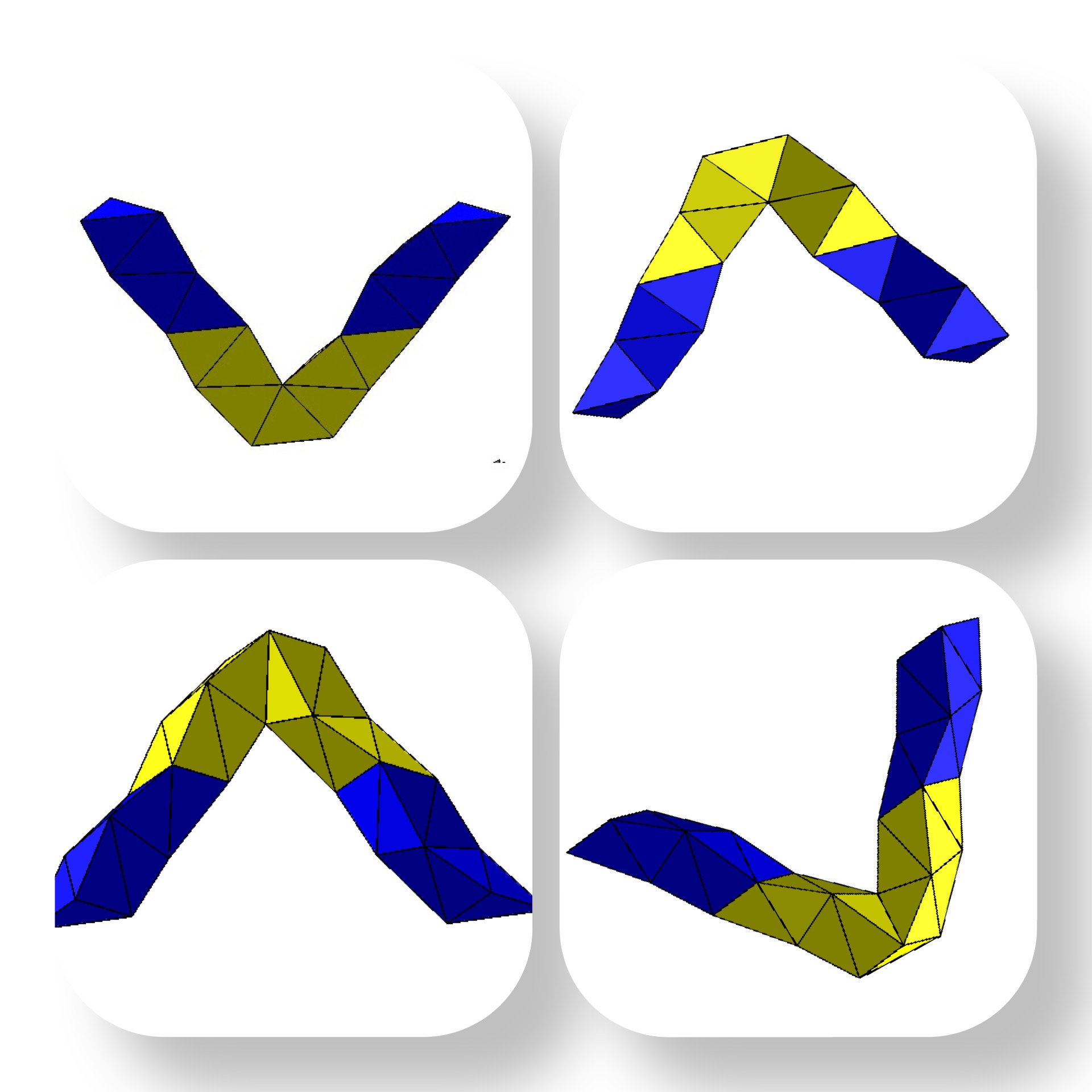} 
  \caption{Various views of constructing a double tetra-helix with $23$ tetrahedra from two copies of $6$-helices (blue) and a double tetra-helix consisting of $11$ tetrahedra (yellow).}
  \label{fig:infinitefamily_helix}
\end{figure} 
From Example~\ref{example:helix6}, we know that the polyhedron $\mathcal{P}$ has two triangular faces $t_1$ and $t_2$ that are incident to embedded vertices of degree $3$ and lie on parallel planes. Similarly, the other copy $\mathcal{P'}$ has triangular faces $t_1'$ and $t_2'$ that are incident to embedded vertices of degree 3 and lie on parallel planes. Next, we apply suitable affine transformations to the polyhedra $\mathcal{P}$ and $\mathcal{P}'$ such that these polyhedra and $X$ have coinciding faces. More precisely, the polyhedron $Y_{\ell(r_2,1)}$ is constructed by identifying the faces $t_1$ and $\conv(\phi_X(f_v))$, and the faces $t_2$ and $\conv(\phi_X(f_w))$ using Remark~\ref{remark:construction}. 
Thus, the faces $t_2$ and $t_2'$ give rise to faces $f_{\overline{v}}$ and $f_{\overline{w}}$ of $Y$ that are incident to vertices $\overline{v}$ and $\overline{w}$ of degree $3$ such that $ 
 \conv(\phi_{Y}(f_{\overline{v}}))$ and $\conv(\phi_{Y}(f_{\overline{w}}))$ lie on the same plane $\overline{P}$.
Since the only edges that are contained in this plane $\overline{P}$ are the edges that are incident to the faces $\conv(\phi_{Y}(f_{\overline{v}}))$ and $\conv(\phi_{Y}(f_{\overline{w}}))$, we have established that $(Y,\overline{v},\overline{w},f_{\overline{v}},f_{\overline{w}},\phi_Y)$ satisfies the property \textit{(T1)}.
\end{proof}

As examples, we illustrate the polyhedra that arise from the double tetra-helices $\mathcal{D}_{11}^{\ell(r_2,1)}$ and $\mathcal{D}_{23}^{\ell(r_2,1)}.$ We also show the toroidal polyhedra that result from reflecting the above polyhedra at the plane that contains the two vertices of degree three in Figure~\ref{fig:infinitefamily}.

\begin{figure}[H]
  \centering
\begin{minipage}{.45\textwidth}
  \begin{subfigure}{.45\textwidth}
  \centering
    \includegraphics[height=4cm]{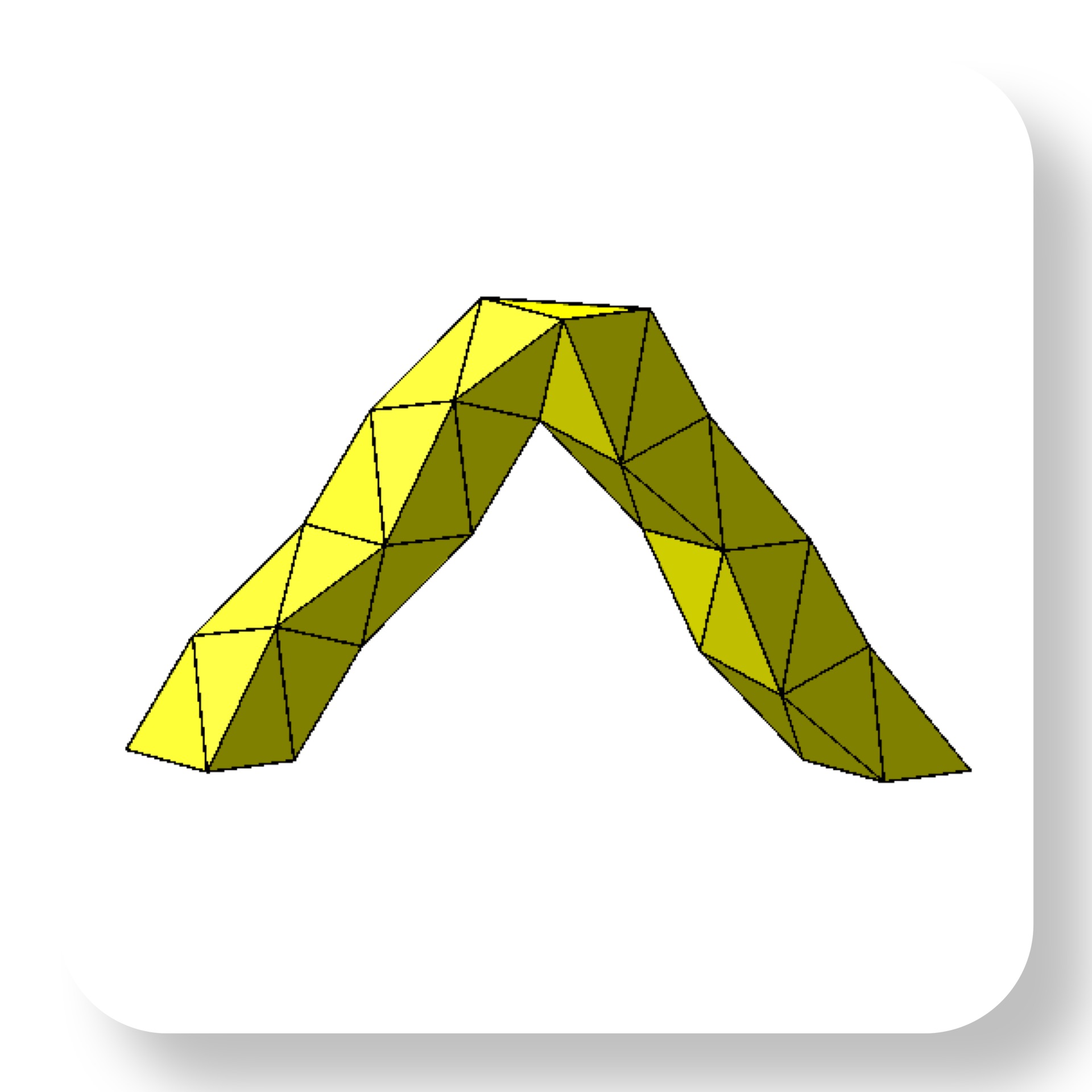}
        \includegraphics[height=4cm]{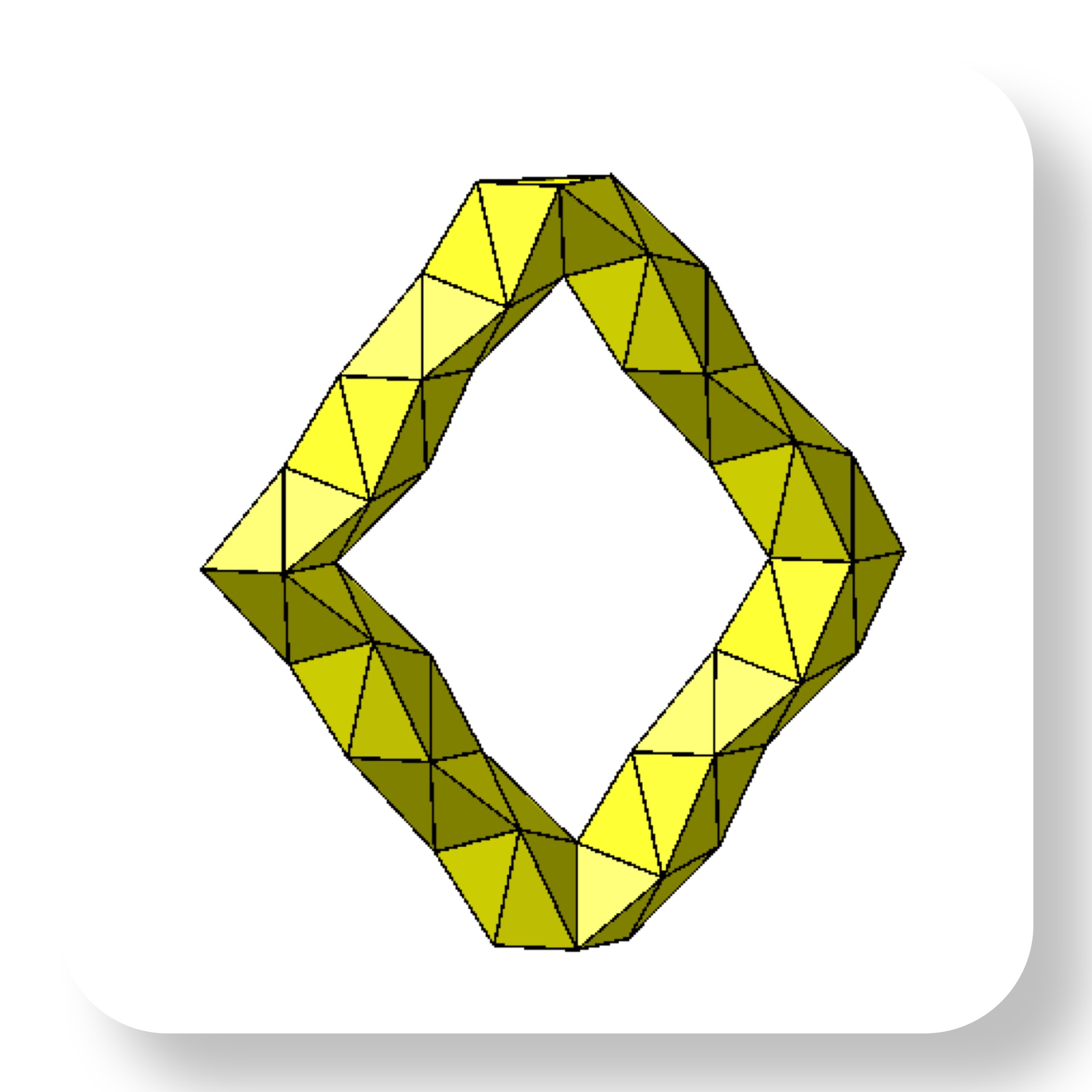}
        \caption{}
  \end{subfigure}
\end{minipage}
\begin{minipage}{.45\textwidth}
  \begin{subfigure}{.45\textwidth}
    \includegraphics[height=4cm]{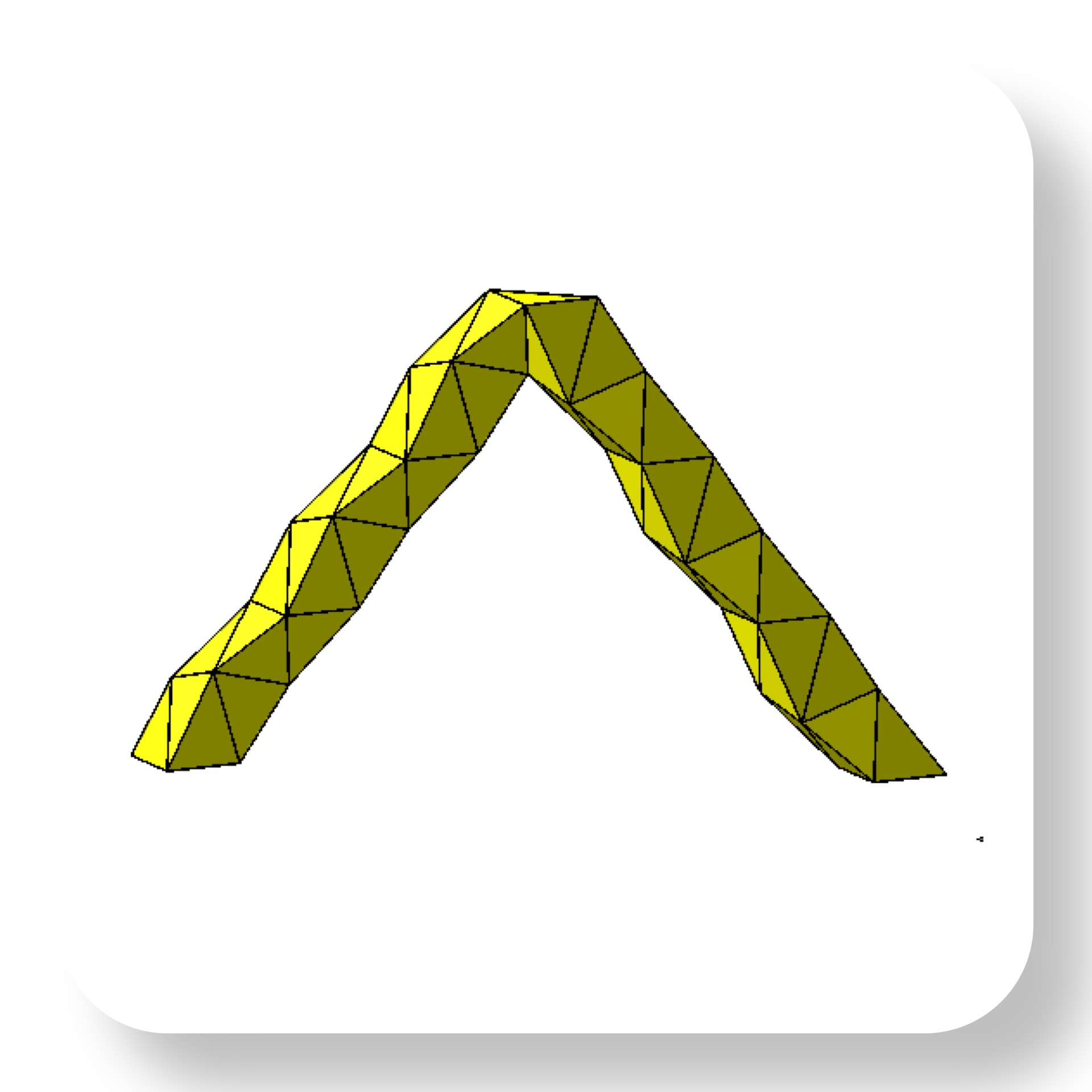}
        \includegraphics[height=4cm]{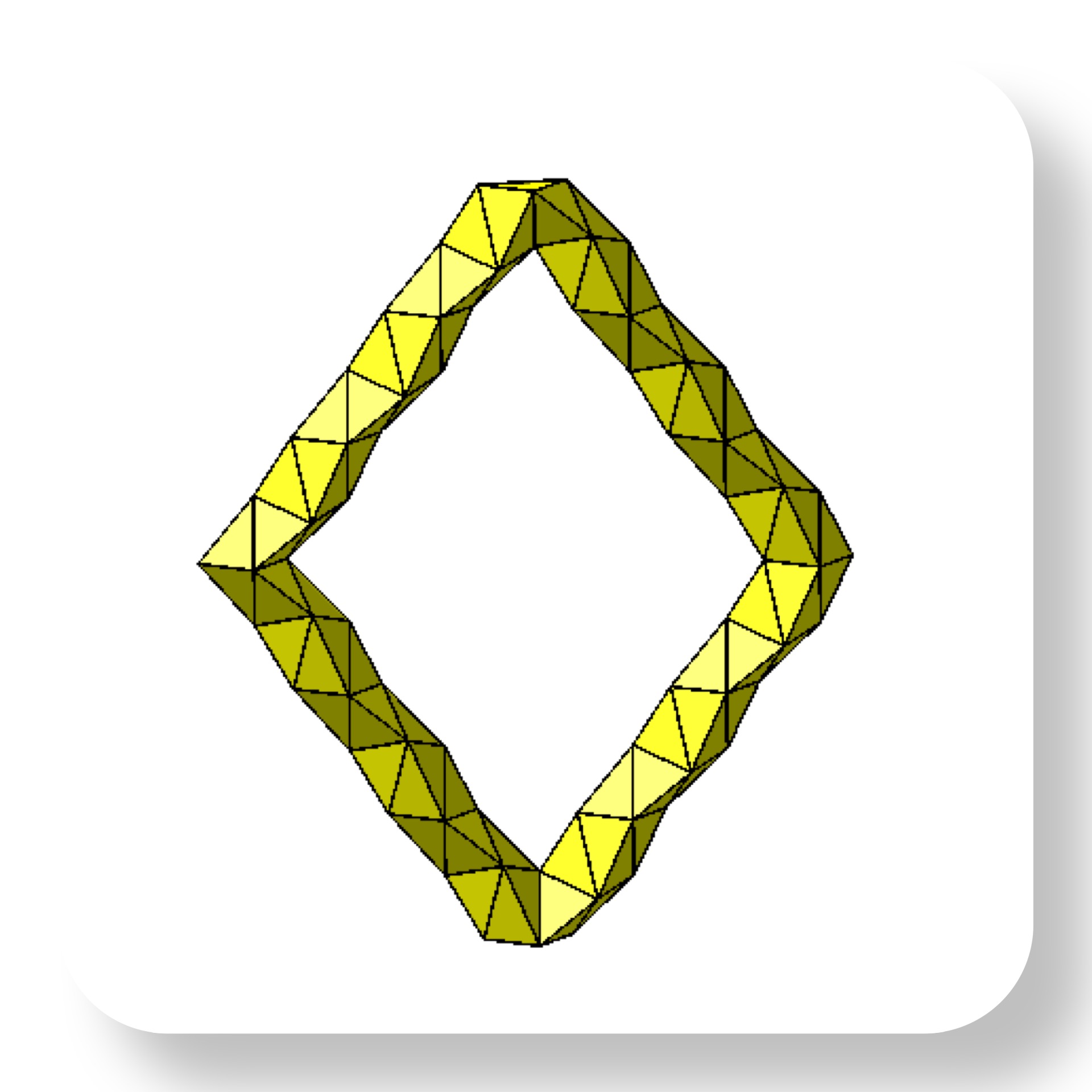}
        \caption{}
  \end{subfigure}
\end{minipage}
  \caption{Examples of polyhedra corresponding to multi-tetrahedral spheres that yield toroidal polyhedra with approximate edge lengths $\ell(r_2,1)$, see Example \ref{example:helix6}.}
  \label{fig:infinitefamily}
\end{figure}
Note that near-perfect chains of regular tetrahedra that result from a similar construction have been investigated in \cite{quadrahelix}. M. Elgersma et al.\ further examined the gap between the planes of two such double tetra-helices to obtain a near-perfect loop.

\section{Multi-tetrahedral surfaces of higher genera}
\label{section:highergenus}
Given a multi-tetrahedral surface of positive genus, larger multi-tetrahedral surfaces of higher genera can be constructed. In this section, we embed certain multi-tetrahedral surfaces with higher genera into $\mathbb{R}^3$ by exploiting strong embeddings of multi-tetrahedral tori. This construction uses the procedure presented in Remark~\ref{remark:construction}. Furthermore, we make use of a space-filling consisting of identical copies of a wild tetrahedron to provide examples of polyhedra of higher genera.
\begin{remark}
Let $X$ be a multi-tetrahedral surface of positive genus and $Y$ be a multi-tetrahedral torus, i.e.\ $\chi(Y)=0$ with corresponding strong $(a,b,c)$-embeddings $\phi_X$ and $\phi_Y,$ respectively.
Furthermore, let $f_X$ and $f_Y$ be faces in $X$ and $Y$ such that there exist no vertices $v\in X_0\setminus f_X$ and $v'\in Y_0 \setminus f_Y$ satisfying that $\phi_X(v)$ and $\phi_Y(v')$ are contained in the planes that are defined by $\phi_X(f_X)$ and $\phi_Y(f_Y),$ respectively.
If all vertex coordinates of vertices in $X$ and $Y$ are contained in exactly one half-space induced by the face $\phi_X(f_X)$ and $\phi_Y(f_Y),$ respectively, then the affine transformations described in Remark~\ref{remark:construction} give rise to a simplicial surface $Z$ with Euler characteristic as follows:
\begin{align*}
&\chi (Z)=\vert Z_0\vert-\vert Z_1\vert+\vert Z_2\vert \\
=&(\vert X_0 \vert+\vert Y_0 \vert-3)-(\vert X_1 \vert+\vert Y_1 \vert-3)+(\vert X_2 \vert+\vert Y_2 \vert-2)\\
&= \chi(X)+\chi(Y)-2=\chi(X)-2.
\end{align*}
Since multi-tetrahedral surfaces are orientable, the genus of $Z$ is given by $g(X)+1.$ 
\end{remark}
By using the above remark, we can use the multi-tetrahedral torus whose corresponding polyhedron is illustrated in Figure~\ref{fig:smallest_torus} to construct polyhedra with increasing genera, see Figure~\ref{fig:genus}.
\begin{figure}[H]
  \centering
  \begin{subfigure}{.3\textwidth}
  \includegraphics[height=4cm]{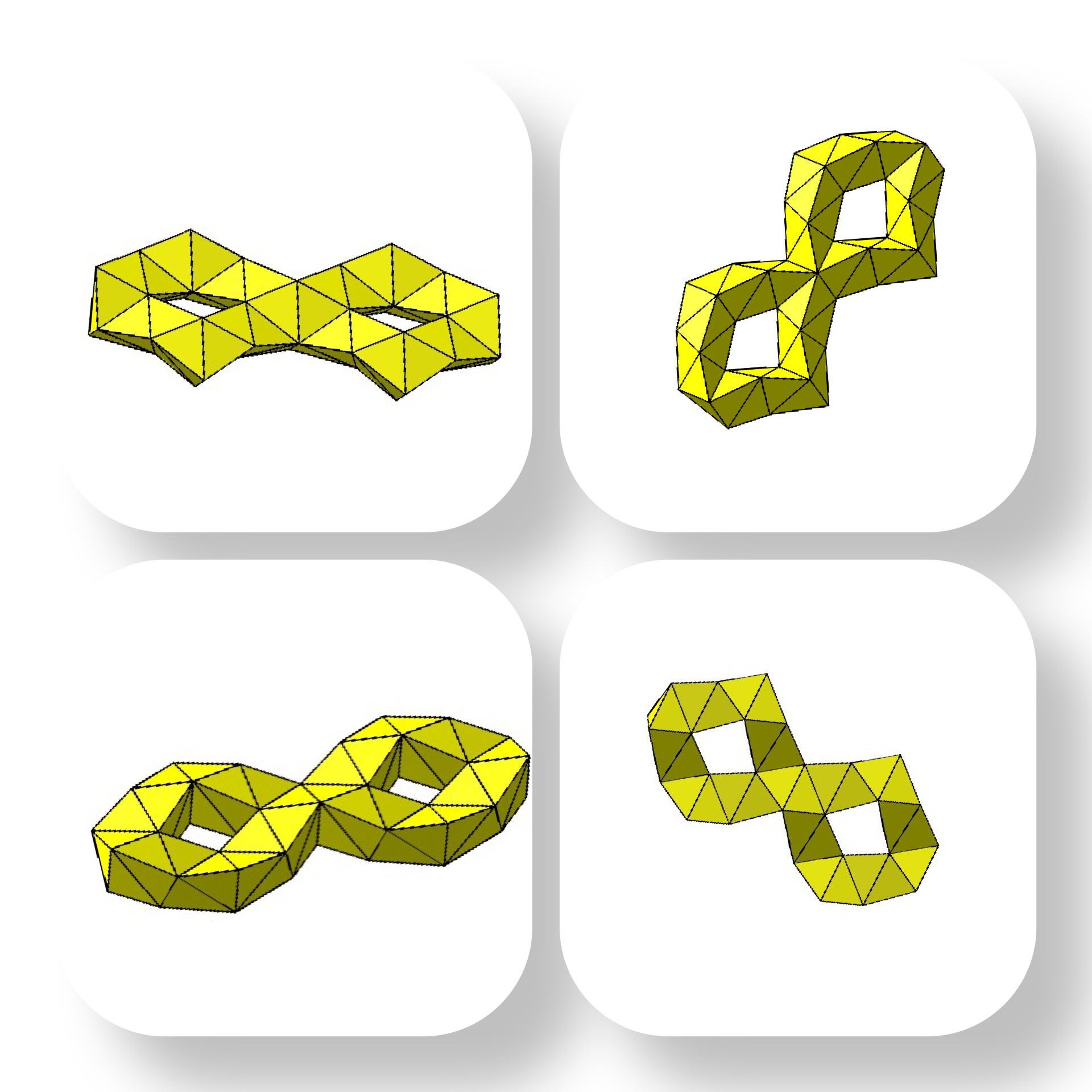}
  \caption{}  
  \end{subfigure}
  \begin{minipage}{.02\textwidth}
    \phantom{a}
  \end{minipage}
  \begin{subfigure}{.3\textwidth}
  \includegraphics[height=4cm]{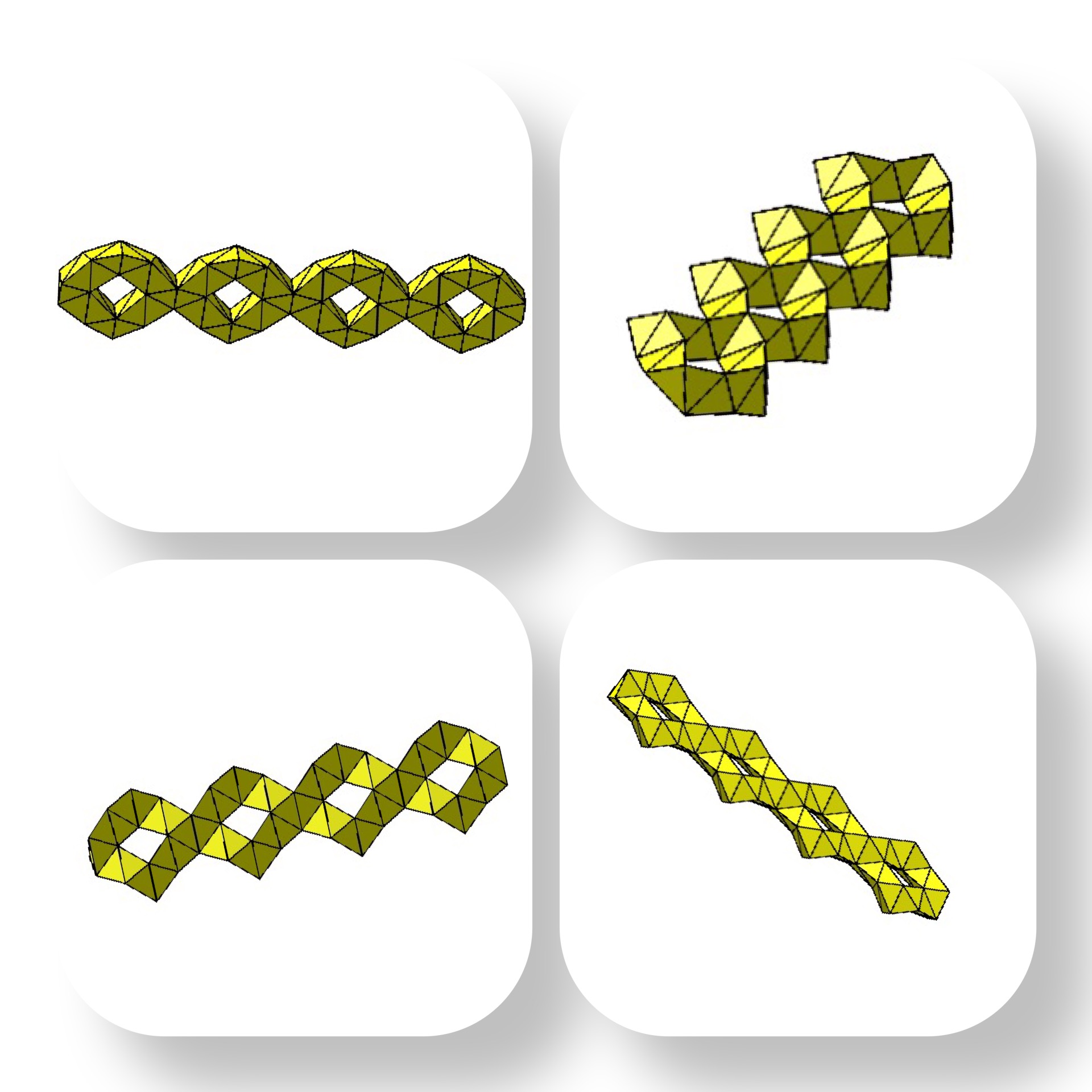}
  \caption{}  
  \end{subfigure}
    \begin{minipage}{.02\textwidth}
      \phantom{a}
  \end{minipage}
  \begin{subfigure}{.3\textwidth}
  \includegraphics[height=4cm]{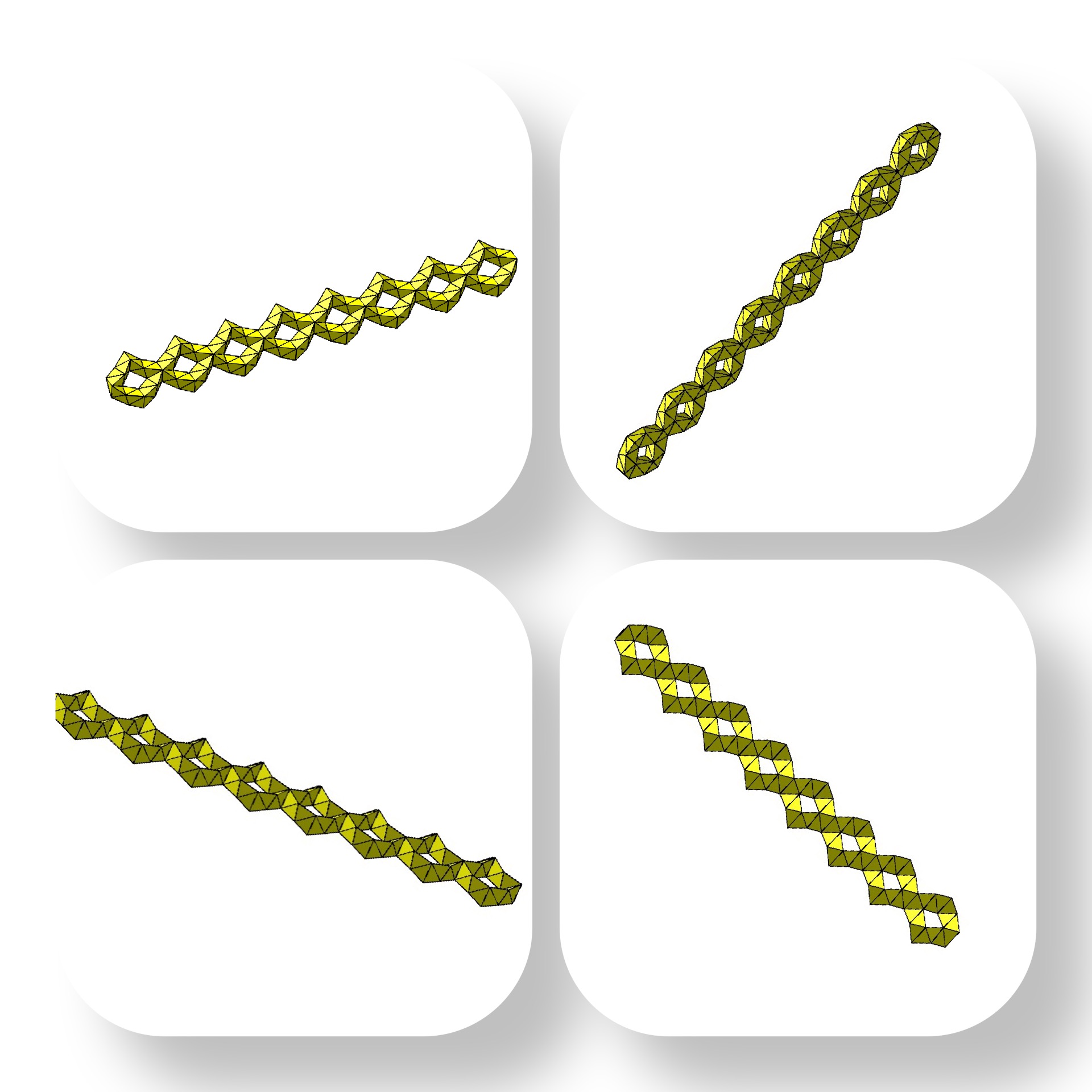}
  \caption{}  
  \end{subfigure}
  \caption{Various views of polyhedra consisting of faces with edge lengths $\ell(r_2,1)$ that result from embedding multi-tetrahedral surfaces of genera $2$ (a), $4$ (b) and $8$ (c) into $\mathbb{R}^3$}
  \label{fig:genus}
\end{figure}


In the following remark, we present the idea of exploiting space-fillings to construct various examples of polyhedra of higher genera. In particular, we make use of a space-filling of the Euclidean $3$-space that consists of congruent tetrahedra with isosceles triangular faces that have been introduced by Sommerville in 1922.
\begin{remark}\label{remark:sommerville}
In \cite{sommerville}, Sommerville presents an \emph{isosceles tetrahedron}, i.e.\ a tetrahedron consisting of isosceles faces with edge lengths $(2,\sqrt{3},\sqrt{3}),$ and further proves that it is possible to arrange copies of this tetrahedron to obtain a space-filling of the Euclidean $3$-space. Here, we associate an undirected edge-coloured graph to this space-filling to obtain various examples of polyhedra of higher genera consisting of congruent isosceles faces. We define the graph $\mathcal{G}=(V,E)$ as follows: Let $\{T_i\subseteq \mathbb{R}^3\mid i\in\mathbb{N}\}$ be the set of isosceles tetrahedra consisting of faces with edge lengths $(2,\sqrt{3},\sqrt{3})$ that give rise to the space-filling in \cite{sommerville}.
     We define the set of vertices $V$ as $V:=\mathbb{N}.$ Furthermore, the set $\{i,j\}$ is an edge in $\mathcal{G}$, i.e.\ $\{i,j\}\in E$, if and only if $i\neq j$ holds and the intersection $T_i\cap T_j$ is non-empty. Moreover, let $\zeta:E\to \{1,2,3\}$  be an arc colouring of $\mathcal{G}$ such that the equality $\zeta(\{i,j\})=1,2,3$ holds if and only if the intersection of $T_i$ and $T_j$ is a vertex, an edge or a face of both tetrahedra, respectively. By reading the subscripts modulo $n$, we observe that every chordless cycle $C=(i_1,\ldots,i_n)$ in $\mathcal{G}$ with $\zeta (\{i_k,i_{k+1}\})=3$ for all $1\leq k\leq n$ corresponds to a perfect chain of wild tetrahedra $\tau=(T_{i_1},\ldots,T_{i_n})$. This yields a toroidal polyhedron consisting of triangular faces with edge lengths $(2,\sqrt{3},\sqrt{3}).$
Furthermore, if $C$ is a cycle (not necessarily chordless) with $\zeta (\{i,j\})=3$ for all $\{i,j\}\in E\cap \{i_1,\ldots,i_n\},$  then $C$ corresponds to a perfect chain of wild tetrahedra. Thus, the space-filling by Sommerville yields various examples of toroidal polyhedra and polyhedra of higher genera. 
\end{remark} 

\section{Notes on implementation}
\label{section:implementation}
In this section, we discuss the implementations (available in \cite{dataTori}) that allowed us to construct a census of toroidal polyhedra. The computations to obtain the desired toroidal polyhedra are all carried out using the computer algebra systems \textsc{GAP} and \textsc{Maple}. 
The \textsc{GAP}-package \emph{SimplicialSurfaces} gives us functionalities to examine the combinatorial structure of the simplicial surfaces in this paper. We further make use of the \textsc{Maple}-package \emph{SimplicialSurfaceEmbeddings} to study the geometric properties of the polyhedra that arise from embedding multi-tetrahedral surfaces into $\R^3$. Moreover, this package contains functionalities to plot and visualise the polyhedra that arise in our studies.
We computed the mentioned census as described in the following:
First, we use \textsc{GAP} to extract proper multi-tetrahedral spheres from established databases of multi-tetrahedral spheres (see \cite{brinkmann}). These spheres form the foundation of our search for toroidal polyhedra. We then construct $(1,b,c)$-embeddings of the obtained proper multi-tetrahedral spheres in \textsc{Maple} as described in Remark~\ref{remark:embVertex}. For a given proper multi-tetrahedral sphere, we check whether there exist $b,c \in \R_{>0}$ such that a tuple derived from the given sphere satisfies property $\textit{(T1)}$ or property $\textit{(T2)}$ (Definition~\ref{def:propertyt1} and Definition~\ref{def:propertyt2}). 
These computations were all carried out algebraically. 
Our search for toroidal polyhedra is divided into two cases, namely finding toroidal polyhedra with and without reflection symmetries. Here, we employ Algorithms \ref{alg:torus1} and \ref{alg:torus2} to perform the required computations.
 In the following table, we give more details on the toroidal polyhedra that have been computed in this study. This table contains the total numbers of toroidal polyhedra resulting from multi-tetrahedral tori with up to $36$ faces with respect to self-intersections and reflection symmetries. These numbers are provided up to isometries, isomorphism of the simplicial surfaces and up to permutation of edge lengths of the toroidal polyhedra. 
 \begin{table}[H]
  \centering
  \begin{tabular}{ |c|c|c|c| } 
    \hline
    & \textbf{with self-inters.} & \textbf{without self-inters}. \\ 
    \hline
    \textbf{with reflection sym.} & 149 & 27\\ 
    \hline
    \textbf{without reflection sym.} & 0 &0 \\ 
    \hline
  \end{tabular}
  \caption{Number of toroidal polyhedra with reflection symmetry consisting of up to $36$ faces and without reflection symmetry up to $18$ faces.}
  \label{tab:totaltori}
\end{table}
More precisely, in the following table, we provide the numbers of constructed toroidal polyhedra for a given number of faces with respect to reflection symmetries. For this purpose, let (a) $A_n$ and (b) $B_n$ be the set of toroidal polyhedra consisting of $n$ faces (a) with self-intersections and (b) without self-intersections.
\begin{table}[H]
  \centering
  \begin{tabular}{ |c|c|c|c|c|c| } 
    \hline
    $\mathbf{n}$ & \textbf{28} & \textbf{32} & \textbf{36}  \\ 
    \hline
    $\mathbf{|A_n|}$ & 11 & 11 &127\\ 
    \hline
    $\mathbf{|B_n|}$ & 2 & 2 &23 \\ 
    \hline
  \end{tabular}
  \caption{Number of toroidal polyhedra with reflection symmetry. }
  \label{tab:totalmirror}
\end{table}
Note that a given multi-tetrahedral sphere can give rise to different non-isometric toroidal polyhedra.
Further, we note that our computations verify that there does not exist a toroidal polyhedron without reflection symmetries and self-intersections with less than or equal to $18$ faces. Nevertheless, such toroidal polyhedra can be obtained by using the construction from Remark~\ref{remark:sommerville}.

\subsection*{Acknowledgement}
R. Akpanya, A. Niemeyer and D. Robertz acknowledge the funding by the Deutsche Forschungsgemeinschaft (DFG, German Research Foundation) in the framework of the Collaborative Research Centre CRC/TRR 280 “Design Strategies for Material-Minimized Carbon Reinforced Concrete Structures – Principles of a New Approach to Construction” (project ID 417002380). The authors thank Meike Weiß and Giulia Iezzi for their useful
comments on earlier versions of the manuscript.

\newpage
\bibliographystyle{abbrv}

\begin{thebibliography}{10}

\bibitem{maRey}
R.~Akpanya.
\newblock {Klassifikation der Sphären ohne Zweier-Taillen}.
\newblock Masterarbeit, RWTH Aachen University, 2021.

\bibitem{simplicialsurfacegap}
R.~Akpanya, M.~Baumeister, T.~Goertzen, A.~C. Niemeyer, and M.~Weiß.
\newblock {SimplicialSurfaces, Version 0.6}.
\newblock \url{https://github.com/gap-packages/SimplicialSurfaces}, 2024.

\bibitem{automorphism}
R.~{Akpanya} and T.~{Goertzen}.
\newblock {Surfaces with given Automorphism Group}.
\newblock {\em arXiv e-prints}, page arXiv:2307.12681, July 2023.

\bibitem{dataTori}
R.~Akpanya and V.~Kirekod.
\newblock Embeddings-of-wild-coloured-surfaces.
\newblock \url{https://github.com/ReymondAkpanya/Embeddings-of-wild-coloured-surfaces.git}, 2024.

\bibitem{gruenbaum1}
M.~O. Albertson, H.~Alpert, s.-m. belcastro, and R.~Haas.
\newblock Gr\"unbaum colorings of toroidal triangulations.
\newblock {\em J. Graph Theory}, 63(1):68--81, 2010.

\bibitem{apolloniannetwork_Intro}
J.~S. Andrade, H.~J. Herrmann, R.~F.~S. Andrade, and L.~R. da~Silva.
\newblock Apollonian networks: Simultaneously scale-free, small world, euclidean, space filling, and with matching graphs.
\newblock {\em Phys. Rev. Lett.}, 94:018702, Jan 2005.

\bibitem{Birkhoff_1930}
G.~D. Birkhoff.
\newblock On the number of ways of colouring a map.
\newblock {\em Proceedings of the Edinburgh Mathematical Society}, 2(2):83–91, 1930.

\bibitem{icosahedron}
K.-H. Brakhage, A.~C. Niemeyer, W.~Plesken, D.~Robertz, and A.~Strzelczyk.
\newblock The icosahedra of edge length 1.
\newblock {\em J. Algebra}, 545:4--26, 2020.

\bibitem{onetriangle}
K.-H. Brakhage, A.~C. Niemeyer, W.~Plesken, and A.~Strzelczyk.
\newblock Simplicial surfaces controlled by one triangle.
\newblock {\em J. Geom. Graph.}, 21(2):141--152, 2017.

\bibitem{brinkmann}
G.~Brinkmann and B.~D. McKay.
\newblock Fast generation of planar graphs.
\newblock {\em MATCH Commun. Math. Comput. Chem.}, 58(2):323--357, 2007.

\bibitem{symmetry}
J.~H. Conway, H.~Burgiel, and C.~Goodman-Strauss.
\newblock {\em The symmetries of things}.
\newblock A K Peters, Ltd., Wellesley, MA, 2008.

\bibitem{ghentpaper}
K.~Coolsaet and S.~Schein.
\newblock Some new symmetric equilateral embeddings of platonic and archimedean polyhedra.
\newblock {\em Symmetry}, 10(9), 2018.

\bibitem{coxeterhelix}
H.~S.~M. Coxeter.
\newblock {\em Regular polytopes}.
\newblock Courier Corporation, 1973.

\bibitem{dekkerproof}
T.~J. Dekker.
\newblock On reflections in {E}uclidean spaces generating free products.
\newblock {\em Nieuw Arch. Wisk. (3)}, 7:57--60, 1959.

\bibitem{platonicgap}
M.~Elgersma and S.~Wagon.
\newblock Closing a platonic gap.
\newblock {\em Math. Intelligencer}, 37(1):54--61, 2015.

\bibitem{quadrahelix}
M.~Elgersma and S.~Wagon.
\newblock The quadrahelix: A nearly perfect loop of tetrahedra, 2016.

\bibitem{asymptotic_helix}
M.~Elgersma and S.~Wagon.
\newblock An asymptotically closed loop of tetrahedra.
\newblock {\em Math. Intelligencer}, 39(3):40--45, 2017.

\bibitem{eppstein}
D.~Eppstein.
\newblock On polyhedral realization with isosceles triangles.
\newblock {\em Graphs Combin.}, 37(4):1247--1269, 2021.

\bibitem{fowler}
T.~G. Fowler.
\newblock {\em Unique coloring of planar graphs}.
\newblock ProQuest LLC, Ann Arbor, MI, 1998.
\newblock Thesis (Ph.D.)--Georgia Institute of Technology.

\bibitem{GAP4}
The GAP~Group.
\newblock {\em {GAP -- Groups, Algorithms, and Programming, Version 4.13.1}}, 2024.

\bibitem{JOUR}
B.~Gr\"{u}nbaum.
\newblock Unambiguous polyhedral graphs.
\newblock {\em Israel J. Math.}, 1:235--238, 1963.

\bibitem{acuteangleold}
J.~Leech.
\newblock Some properties of the isosceles tetrahedron.
\newblock {\em Math. Gaz.}, 34:269--271, 1950.

\bibitem{maple}
{Maplesoft, a division of Waterloo Maple Inc..}
\newblock {\em Maple}.
\newblock Waterloo, Ontario.

\bibitem{masonproof}
J.~H. Mason.
\newblock Can regular tetrahedra be glued together face to face to form a ring?
\newblock {\em The Mathematical Gazette}, 56(397):194–197, 1972.

\bibitem{simplicialsurfacebook}
A.~C. Niemeyer, W.~Plesken, and D.~Robertz.
\newblock Simplicial surfaces of congruent triangles.
\newblock {\em In preparation}, 2024.

\bibitem{SimplSurfEmb}
D.~Robertz.
\newblock {\it SimplicialSurfaceEmbeddings}: a {M}aple package for the analysis and construction of embedded simplicial surfaces.

\bibitem{sommerville}
D.~M.~Y. Sommerville.
\newblock Space-filling tetrahedra in euclidean space.
\newblock {\em Proceedings of the Edinburgh Mathematical Society}, 41:49–57, 1922.

\bibitem{steinhaus1957}
H.~Steinhaus.
\newblock Probl\`eme 175.
\newblock In {\em Colloq. Math 4}, page 243, 1957.

\bibitem{stewart}
I.~Stewart.
\newblock Tetrahedral chains and a curious semigroup.
\newblock {\em Extracta Math.}, 34(1):99--122, 2019.

\bibitem{ansgar}
A.~W. Strzelczyk.
\newblock {\em {S}impliziale {F}lächen aus kongruenten {D}reiecken : kombinatorische {G}rundlagen und geometrische {B}eispiele}.
\newblock Dissertation, RWTH Aachen University, Aachen, 2019.
\newblock Veröffentlicht auf dem Publikationsserver der RWTH Aachen University; Dissertation, RWTH Aachen University, 2019.

\bibitem{swierczkowski1959chains}
S.~\'{S}wierczkowski.
\newblock On chains of regular tetrahedra.
\newblock {\em Colloq. Math.}, 7:9--10, 1959.

\bibitem{takeo}
F.~Takeo.
\newblock On triangulated graphs. {I}.
\newblock {\em Bull. Fukuoka Gakugei Univ. III}, 10:9--21, 1960.

\bibitem{deltahedra}
C.~W. Trigg.
\newblock An {I}nfinite {C}lass of {D}eltahedra.
\newblock {\em Math. Mag.}, 51(1):55--57, 1978.

\bibitem{grunbaum1969conjecture}
W.~T. Tutte, editor.
\newblock {\em Recent progress in combinatorics}. Academic Press, New York-London, 1969.

\bibitem{menger}
K.~Wirth and A.~S. Dreiding.
\newblock Edge lengths determining tetrahedrons.
\newblock {\em Elem. Math.}, 64(4):160--170, 2009.

\end{thebibliography}

\end{document}